\documentclass[11pt]{amsart}

\usepackage[top=3cm, left=3cm, right=3cm]{geometry}
\usepackage{amsmath}
\usepackage{amssymb}
\usepackage{amsfonts}
\usepackage{mathrsfs}
\allowdisplaybreaks
\usepackage[colorlinks]{hyperref}
\hypersetup{
linkcolor=blue,
citecolor=blue,
urlcolor=blue
}

\theoremstyle{definition}
\newtheorem{definition}{Definition}[section]
\newtheorem{remark}[definition]{Remark}

\newtheorem{lemma}[definition]{Lemma}
\newtheorem{proposition}[definition]{Proposition}
\newtheorem{corollary}[definition]{Corollary}
\newtheorem{theorem}[definition]{Theorem}

\newcommand{\SL}{\operatorname{SL}}
\newcommand{\Vol}{\operatorname{Vol}}

\newcommand{\Ad}{\operatorname{Ad}}
\newcommand{\ad}{\operatorname{ad}}

\newcommand{\SO}{\operatorname{SO}}
\newcommand{\diam}{\operatorname{diam}}
\newcommand{\Stab}{\operatorname{Stab}}

\begin{document}

\thanks{RF was supported by SNF grant 200021--182089. CZ acknowledges support by the Institute of Modern Analysis-A Frontier Research Centre of Shanghai.}

\author[R. Fregoli]{Reynold Fregoli}
\address{Department of Mathematics, University of Z\"urich, Switzerland}
\email{reynold.fregoli@math.uzh.ch}

\author[C. Zheng]{Cheng Zheng}
\address{School of Mathematical Sciences, Shanghai Jiao Tong University, China}
\email{zheng.c@sjtu.edu.cn}


\title[shrinking-target problem in the space of unimodular lattices]{A shrinking-target problem in the space of unimodular lattices in the three dimensional Euclidean space}
\maketitle{}

\begin{abstract}
In this paper, we study the shrinking-target problem with target at infinity induced by the injectivity radius function under the action of a regular diagonalizable flow on $\SL_3(\mathbb R)/\SL_3(\mathbb Z)$. In particular, we establish an explicit formula for the Hausdorff dimension of the subset of points $p$ whose orbit approaches the cusp infinitely often with a rate $\gamma\geq0$.
\end{abstract}

\section{Introduction and main result}\label{intro}
\subsection{Introduction}


Let $(X,d)$ be a metric space. The aim of a shrinking-target problem on $X$ is, generally speaking, to determine the size of the set $S\subset X$ of points whose orbit under a given flow approaches a fixed target $T\subset X$ infinitely often with a desired rate. If the space $X$ is equipped with an invariant Borel measure $\mu$, a shrinking-target problem may be studied in two distinct aspects: the measure-theoretical and the dimensional aspects. In the first case, the goal of the problem reduces to computing the measure of the subset $S$, usually by establishing a null-conull law. Questions of this nature are closely related to various Khintchine-type theorems in the metric theory of Diophantine approximation \cite{K24,K26} and have extensively been studied in the past few decades \cite{AM09,GK17,KO21,KY19,KM99,KZ18,S82,T08}. In homogeneous spaces, the measure-theoretical aspect was addressed in full generality by Kleinbock and Margulis \cite{KM99}, who proved that both a null-conull law and a logarithm law hold. 

The dimensional aspect of a shrinking-target problem, on the other hand, aims to describe sets $S$ that are not detected by the measure $\mu$, by, e.g., computing their Hausdorff dimension or by showing that $S\neq\emptyset$. This line of investigation was initiated by Hill and Velani \cite{HV95} and was further developed in different settings \cite{FMSU15,HV97,LWWX14,HR11,SW13,U02}. This body of work is closely related to Jarn\'ik-Besicovitch-type theorems in the metric theory of Diophantine approximation \cite{B34,J31}. In this setting, the analysis becomes more delicate and often requires more information about the interaction between the algebraic and geometric structures on the space $X$ and the dynamics of the flow (e.g. the metric used to define neighbourhoods of the target, counting results, etc.).

Within this framework, we study the dimensional aspect of the shrinking-target problem induced by the injectivity radius function $\eta$ on the space $\SL_3(\mathbb R)/\SL_3(\mathbb Z)$, with respect to the action of a regular diagonalizable flow $\{a_{t}\}_{t\in\mathbb{R}}$. Let us be more precise. Let $G:=\SL_3(\mathbb R)$ and let $\Gamma:=\SL_3(\mathbb Z)$. Let $X=G/\Gamma$. The injectivity radius function $\eta$ on $X:=G/\Gamma$ is defined as $$\eta(x):=\inf_{v\in\Stab(x)\setminus\left\{e\right\}}d_G(v,e)\;(x\in X),$$ where $\Stab(x)$ is the stabilizer of $x$ and $d_G$ is the metric on $G$ induced by a norm $\|\cdot\|_{\mathfrak g}$ on the Lie algebra of $G$. It is well-known \cite[Theorem 1.12]{R72} that the behaviour of the injectivity radius function $\eta$ reflects the divergence of a sequence of points $\left\{p_n\right\}_{n\in\mathbb N}$ in $X$, in the sense that $p_n$ diverges to infinity if and only if $\eta(p_n)\to 0$. Therefore, the injectivity radius function may be used to define a chain of neighborhoods of infinity of the form $$U^{\eta}_{r}:=\left\{x\in X:\eta(x)<r\right\} (r>0).$$
For a regular $\mathbb R$-diagonalizable flow $\left\{a_t\right\}_{t\in\mathbb R}$ on $X$ and a fixed decreasing function $r(t)$, we will be studying the Hausdorff dimension of the subset of points $p$ in $X=G/\Gamma$ whose orbit under the flow $\left\{a_t\right\}_{t\in\mathbb R}$ enters the neighbourhoods $U^{\eta}_{r(t)}$ for an unbounded sequence of times $t$.

For comparison, it will be useful to keep in mind the analogous shrinking-target problem on $X$ induced by the systole function $\delta$. By identifying $X=G/\Gamma$ with the space of unimodular lattices in $\mathbb R^3$, one can define $\delta$ on $X$ as
$$\delta(g\Gamma)=\inf_{v\in g\mathbb Z^3\setminus\left\{0\right\}}\|v\|.$$
By Mahler's compactness criterion, a sequence of points $p_n\in X$ diverges if and only if $\delta(p_n)\to 0$. In light of this, the function $\delta$ may also be used to define a chain of neighborhoods of infinity of the form $$U^{\delta}_{r}:=\left\{x\in X:\delta(x)<r\right\}\; (r>0),$$ inducing a different shrinking-target problem with target at infinity. It is well-known that the function $\delta$ connects the metric theory of Diophantine approximation with certain dynamical systems on $X$, through the so-called Dani correspondence \cite{D85}. In view of this connection, the shrinking-target problem for the function $\delta$ has been analyzed in depth (see, e.g.,\cite{KM99}).   

The main motivation for us to consider the shrinking-target problem induced by the injectivity radius function $\eta$, instead of the systole, is that the former provides a greater amount of information on orbit structures in the space $X$. For instance, it is quite complicated to use the information embedded in the function $\delta$ in order to characterize the periodicity of an orbit. On the other hand, the injectivity radius function $\eta$ carries information about the stabilizer $\Stab(x)$ of a point $x\Gamma$, which, in turn, controls dynamical properties of orbits starting from $x$. For example, the orbit of $x$ under a group action $H\subset G$ is periodic if and only if $H\cap\Stab(x)$ is a lattice in $H$. By studying the co-volume of $H\cap\Stab(x)$ in $H$ or the distances between points in $\Stab(x)$, one is able to grasp the geometric properties of the periodic orbit $H\cdot x$. This phenomenon appears in various instances and plays an important role in the study of the sparse equidistribution problem. By considering the analogous shrinking-target problem in rank-one homogeneous spaces, we can, e.g., prove that certain sparse horocycles are always densely distributed in homogeneous subspaces without any conditions on initial points. One can refer to \cite{Z21,Z16,Z19} for more details. 

On rank-one homogeneous spaces and negatively curved manifolds, various shrinking-target problems have been studied both in the measure-theoretical and dimensional aspects \cite{HP01,HP,HP02,MP93,Z19} (including the decay of the injectivity radius). In higher-rank homogeneous spaces, only the measure-theoretical aspect has been addressed \cite{KM99} and little is known in the dimensional aspect. The only exception is perhaps \cite{D92}, where Dodson studied a shrinking-target problem for certain singular $\mathbb R$-diagonalizable flows on $\SL_n(\mathbb R)/\SL_n(\mathbb Z)$ in an implicit way. With this paper, we aim to further contribute to this picture, by addressing the dimensional aspect of the shrinking-target problem induced by $\eta$ in the higher-rank homogeneous space $\textup{SL}_{3}(\mathbb{R})/\textup{SL}_{3}(\mathbb{Z})$.

Overall, there are rich structures hidden behind the particular shrinking-target problem that we consider in this paper, which reveal previously unknown dynamical properties of regular one-parameter diagonal flows. We expect that the investigation carried out here will lay out a foundation to further study problems related to the injectivity radius function on homogeneous spaces, such its interaction with higher-rank diagonal groups and, e.g., the shape of sets on which it decays exceptionally fast.
 
\subsection{Main Result and Organization of the Paper}

Let $\left\{a_t\right\}_{t\in\mathbb R}$ be a one-parameter $\mathbb R$-diagonalizable subgroup in $G$. We will assume that $\left\{a_t\right\}$ is regular, in the sense that the centralizer of $\left\{a_t\right\}_{t\in\mathbb{R}}$ is a Cartan subgroup $D$ in $G$. We denote by $A$ the connected component of $D$ containing the identity, by $\mathfrak a$ the Lie algebra of $D$, and by $\mathfrak g$ the Lie algebra of $G$.  

The group $A$ acts on the Lie algebra $\mathfrak g$ of $G$ via the adjoint representation $\Ad:G\to\Ad(G)$, and, consequently, $\mathfrak g$ may be decomposed as a direct sum of root spaces $$\mathfrak g=\mathfrak g_0\oplus\sum_{\alpha\in\Phi}\mathfrak g_\alpha,$$ where $\Phi$ is the set of roots of $A$ and $\mathfrak g_0=\mathfrak a$. Denote by $N_G(A)$ and $C_G(A)$ the normalizer and the centralizer of $A$ respectively. Let $W=N_G(A)/C_G(A)$ be the Weyl group of $A$. For the diagonalizable subgroup $\left\{a_t\right\}_{t\in\mathbb R}$ we write $$a_t=\exp(t\cdot X_0)\;(t\in\mathbb R)$$ for some $X_0\in\mathfrak g$. Since $\left\{a_t\right\}_{t\in\mathbb R}$ is regular, we have $\alpha(X_0)\neq0$ for all roots $\alpha\in\Phi$.

We denote by $N_-$ and $N_+$ the stable and unstable horospherical subgroups relative to $\left\{a_t\right\}_{t\in\mathbb R}$ respectively, i.e., the groups
$$N_{-}=\left\{g\in G: a_{t}ga_{-t}\to e,\ t\to +\infty\right\}\quad\mbox{and}\quad N_{+}=\left\{g\in G: a_{t}ga_{-t}\to e,\ t\to -\infty\right\}.$$
Note that $N_-$ is conjugate to the group of lower unipotent matrices, and $N_+$ is conjugate to the group of upper unipotent matrices. We choose the roots contained in the Lie algebra $\mathfrak n_+$ of $N_+$ as positive roots, and the roots contained in the Lie algebra $\mathfrak n_-$ of $N_-$ as negative roots. We denote by $\Phi_+$ and $\Phi_-$ the subsets of positive roots and negative roots respectively, and by $\Delta=\left\{\alpha_0,\beta_0\right\}$ the subset of simple roots in $\Phi_+$.

\begin{definition}\label{d11}
Let $\gamma\geq0$. We will say that a point $p\in X$ is Diophantine of type $\gamma$ if there exists a constant $C>0$ such that for any $t>0$ $$\eta(a_t\cdot p)\geq Ce^{-\gamma t}.$$ We will denote by $S_\gamma$ the subset of points of Diophantine of type $\gamma$, and $S_\gamma^c$ the complement of $S_\gamma$ in $X$.
\end{definition}

This definition is analogous to that given in \cite{Z19}. In this context, we use the word Diophantine to draw an analogy with the shrinking-target problem induced by the function $\delta$, for which the decay rate appearing in Definition \ref{d11} translates directly to Diophantine properties of vectors under suitable identifications. 

We now state the main theorem of this paper, where we denote by $\dim_H$ the Hausdorff dimension.

\begin{theorem}\label{mthm}
Let $\left\{a_t\right\}_{t\in\mathbb R}$ be a regular $\mathbb{R}$-diagonalizable flow on $G/\Gamma$ and let $U\subset N_{+}\Gamma/\Gamma$ be an open set. If $0\leq\gamma<(\alpha_0+\beta_0)(X_0)$, we have
$$\dim_H(S_\gamma^c\cap U)=3-\frac{2\gamma}{(\alpha_0+\beta_0)(X_0)}.$$
If $\gamma\geq (\alpha_0+\beta_0)(X_0)$, we have $S_\gamma^c\cap U=\emptyset$.
\end{theorem}

As a corollary, we deduce the following result.

\begin{corollary}\label{cor:mcor}
Let $\left\{a_t\right\}_{t\in\mathbb R}$ be a regular $\mathbb{R}$-diagonalizable flow on $G/\Gamma$ and let $U\subset G/\Gamma$ be an open set. If $0\leq\gamma<(\alpha_0+\beta_0)(X_0)$, we have
$$\dim_H(S_\gamma^c\cap U)=8-\frac{2\gamma}{(\alpha_0+\beta_0)(X_0)}.$$
If $\gamma\geq (\alpha_0+\beta_0)(X_0)$, we have $S_\gamma^c\cap U=\emptyset$.
\end{corollary}

The overall strategy to prove Theorem \ref{mthm} resembles that of \cite{Z19}. However, in a higher-rank homogeneous space the problem appears much more complicated. In this paper, not only we have to upgrade our method considerably compared to \cite{Z19}, but we also have to add some new ingredients.

We start by introducing the notion of rational point and that of denominator. Rational points are, broadly speaking, points with maximal decay rate for the injectivity radius. The name comes from the fact that, on the expanding submanifold, these points admit a rational $\Gamma$-representative. By applying the mixing property of the diagonalizable flow $\{a_{t}\}_{t\in\mathbb R}$, we establish an asymptotic formula for the number of rational points whose denominator lies between $l/2$ and $l$. We also find an explicit formula for the denominator of rational points on $N_{+}\Gamma/\Gamma$.

Further, we prove that most points $p\in S_{\gamma}^{c}$ satisfy a property known as $\gamma$-condition (see \cite{Z19}). Roughly speaking, the $\gamma$-condition asserts that to any point $p\in S_{\gamma}^{c}$ admitting a sequence of times $t_{n}\to +\infty$ for which $\eta(a_{t_{n}}p)\leq e^{-t_{n}\gamma}$, we may associate a sequence of elements $v_{n}$ in the stabilizer of $a_{t_{n}}p$, tending to $e$ with a rate dictated by $\gamma$. By using this sequence $v_{n}$ and the Bruhat decomposition, we define a "standard way" to isolate a rational point $q_{n}'\in X$ at a bounded distance from the point $a_{t_{n}}p$ for each index $n$. We then contract both the points $a_{t_{n}}p$  and $q_{n}'$ through the action of $a_{-t_{n}}$, thus finding a sequence of rational points $q_{n}$ approaching $p$ in $N_{+}\Gamma/\Gamma$ with a speed determined both by $\gamma$ and the contraction rate of the flow $\{a_{t}\}_{t\in\mathbb{R}}$. It will be apparent that this approximation process occurs in $6$ different ways. Therefore, to each rational point $q$ we attach $6$ different types of shrinking neighbourhoods.

By using coverings of shrinking neighbourhoods, we will be able to give an upper bound for the Hausdorff dimension of the set $S_{\gamma}^{c}$. To prove lower bounds, we will construct Cantor-type subsets of $S_{\gamma}^{c}$, by using the asymptotic formula for the number of rational points mentioned above, together with results from \cite{M87} and \cite{U91}.

The paper is organized as follows:
\begin{enumerate}
\item[$\bullet$] In \S2, we list some notation and preliminary results needed in the remaining sections. 

\item[$\bullet$] In \S3, we introduce the notion of rational point and give several characterizations of rational points in $G/\Gamma$. We also prove that the rational points in $N_+\Gamma/\Gamma$ coincide with the $\Gamma$-cosets of rational matrices in $N_+\Gamma$.

\item[$\bullet$] In \S4, we introduce the notion of denominator of a rational point. By applying the mixing property of $\left\{a_t\right\}_{t\in\mathbb R}$, we also establish an asymptotic formula for the number of rational points whose denominator lies between $l/2$ and $l$ for sufficiently large $l>0$.

\item[$\bullet$] In \S5, we find an explicit formula for the denominator of a rational point lying in $N_+\Gamma/\Gamma$. From \S4 we know that such a rational point corresponds to the $\Gamma$-coset of a rational matrix. We will see, however, that the formula for the denominator established in this section does not entirely correspond to the usual denominator of a rational vector or matrix.

\item[$\bullet$] In \S6, we study Diophantine points of type $\gamma$ in $G/\Gamma$. We prove that most points in $S_{\gamma}^{c}$ satisfy the $\gamma$-condition. As mentioned above, we will see that the subset $S_\gamma^c\cap U$ is of mixed type and contains six Diophantine sets, which will be studied case by case.

\item[$\bullet$] In \S7, we determine for which values of the parameter $\gamma$ the set $S_{\gamma}^{c}$ is non-empty.

\item[$\bullet$] In \S8, we apply the results proved in the previous sections to compute an upper bound for the Hausdorff dimension of the set $S_\gamma^c\cap U$. We will do this by studying different Diophantine sets separately. In the computations, the formula for the denominators of rational points proved in \S5 will play an important role in establishing various counting results.

\item[$\bullet$] In \S9, we construct a Cantor-type subset contained in $S_{\gamma}^c\cap U$ and estimate from below its Hausdorff dimension. This will give us a lower bound for the Hausdorff dimension of $S_{\gamma}^c\cap U$. Here we will make use of the asymptotic formula derived in \S4.

\item[$\bullet$] In \S10, we combine the previous results to prove Theorem \ref{mthm} and Corollary \ref{cor:mcor}.
\end{enumerate}

\section{Preliminaries and Reductions}\label{pre}

To simplify our arguments, in the rest of the paper we will assume that $\left\{a_t\right\}_{t\in\mathbb R}$ is a regular one-parameter diagonal subgroup in $G$. In this case, the Cartan subgroup $D$ of $\left\{a_{t}\right\}$ is the full diagonal group and $A$ is the connected component of $D$ containing the identity. Moreover, $\left\{a_t\right\}_{t\in\mathbb R}$ can be written explicitly as $$a_t=\left(\begin{array}{ccc} e^{\lambda_1 t} & 0 & 0\\0 & e^{\lambda_2 t} & 0\\ 0 & 0 & e^{\lambda_3 t}\end{array}\right)\quad\mbox{with}\quad X_0=\left(\begin{array}{ccc} \lambda_1 & 0 & 0\\0 & \lambda_2 & 0\\ 0 & 0 & \lambda_3\end{array}\right),$$ where $\lambda_i\;(i=1,2,3)$ are different constants. Without loss of generality, we may further assume that $\lambda_1>\lambda_2>\lambda_3$. Note that the stable and unstable horospherical subgroups of $\left\{a_t\right\}_{t\in\mathbb R}$ are the group $N_-$ of lower unipotent matrices and the group $N_+$ of upper unipotent matrices respectively $$N_-=\left(\begin{array}{ccc} 1 & 0 & 0\\ * & 1 & 0\\ * & * & 1\end{array}\right)\quad \textup{and}\quad N_+=\left(\begin{array}{ccc} 1 & * & *\\0 & 1 & *\\ 0 & 0 & 1\end{array}\right).$$

The Weyl group $W=N_G(A)/C_G(A)$ is generated by two elements $s_1$ and $s_2$ of order 2 which correspond to two reflections in the Lie algebra $\mathfrak a$ of $A$, and there is a unique element $\bar w$ in $W$ whose word length with respect to the generating set $\left\{s_1,s_2\right\}$ is the longest. The adjoint action of $A$ via $\Ad$ induces an action of $W$ on the set of roots $\Phi$, and the element $\bar w$, with respect to this action, sends the positive roots in $\Phi_+$ to the negative roots in $\Phi_-$. In other words, $\bar w\cdot N_+=N_-$ and $\bar w\cdot N_-=N_+$.

With the above notation, the following result holds.

\begin{proposition}[Bruhat Decomposition]
\label{prop:Bruhat}
For any choice of representatives $\left\{w_i\right\}_{i\in I}$ of the Weyl group $W$ in $G$, we have
$$G=\bigsqcup_{i\in I} DN_-w_iN_-=\bigsqcup_{i\in I} DN_-w_iN_+,$$
where the union is disjoint.
\end{proposition}

\begin{proof}
The first equality corresponds to the usual version of the Bruhat Decomposition (see \cite[Section 14.12]{Bo91}). To prove the second equality, let $\tilde w$ be a representative of $\bar w$ in $G$. From the action of $\bar w$ on $\Phi$, we know that $\tilde wN_+\tilde w^{-1}=N_-$ and $\tilde wN_-\tilde w^{-1}=N_+$. Then, $$G=G\cdot\tilde w=\bigsqcup_{i\in I}DN_-(w_i\tilde w)(\tilde w^{-1}N_-\tilde w)=\bigsqcup_{i\in I}DN_-(w_i\tilde w)N_+$$ and $\left\{w_i\tilde w\right\}_{i\in I}$ is another set of representatives of $W$. This proves the claim.
\end{proof}

In the sequel, we will choose the representatives $w_i$ as follows: 

\begin{align}
 & w_1=\left(\begin{array}{ccc} 1 & 0 & 0\\ 0& 1 & 0\\ 0& 0& 1\end{array}\right),\quad w_2=\left(\begin{array}{ccc} 1 & 0 & 0\\ 0& 0 & -1\\ 0& 1& 0\end{array}\right),\quad w_3=\left(\begin{array}{ccc} 0 & -1 & 0\\ 1& 0 & 0\\ 0& 0& 1\end{array}\right),\nonumber \\
 & w_4=\left(\begin{array}{ccc} 0 & 0 & 1\\ 1& 0 & 0\\ 0& 1& 0\end{array}\right),\quad w_5=\left(\begin{array}{ccc} 0 & 1 & 0\\ 0& 0 & 1\\ 1& 0& 0\end{array}\right),\quad w_6=\left(\begin{array}{ccc} 0 & 0 & 1\\ 0& -1 & 0\\ 1& 0& 0\end{array}\right).\nonumber
\end{align}
We remark that $\left\{w_i\right\}_{i\in I}\subset\Gamma$.

We conclude this section with the following "transversality" lemma, which will be used several times throughout the paper. For any subgroup $H<G$ and any positive real number $r$, we will denote by $B_{H}(r)$ the set $\{h\in H:d_{G}(h,e)<r\}$. Analogously, for any $x\Gamma\in G/\Gamma$, we will denote by $B_{G/\Gamma}(x\Gamma,r)$ the set $\{y\Gamma\in G/\Gamma:d_{G/\Gamma}(y\Gamma,x\Gamma)<r\}$. Here, $d_{G}$ denotes a right-invariant metric on $G$ inducing the topology and $d_{G/\Gamma}$ is the induced metric on the quotient.

\begin{lemma}
\label{lem:transversality}
Let $H_{1},H_{2}<G$ be subgroups such that $\overline{H}_{1}\cap \overline{H}_{2}=\{e\}$ and assume that $H_{2}$ has a closed right $\Gamma$-orbit. Then, for any relatively compact open set $K\subset H_{2}$, there exists $r>0$ such that the map $(h_{1},h_{2}\Gamma)\mapsto h_{1}h_{2}\Gamma$ from $B_{H_{1}}(r)\times K\Gamma/\Gamma$ to $B_{H_{1}}(r)K\Gamma/\Gamma$ is a homeomorphism.

\begin{proof}
Surjectivity and continuity are obvious. We only need to show that the map is injective and open. We start with injectivity. Fix a point $x\Gamma\in K\Gamma/\Gamma$. Since $H_{2}$ has a closed right $\Gamma$-orbit, there exists $\delta>0$, depending only on $x\Gamma$, such that the map $h_{2}\mapsto h_{2}x\Gamma$ is a homeomorphism between $B_{H_{2}}(\delta)$ and an open set in $K\Gamma/\Gamma$ containing a ball $B_{G/\Gamma}(x\Gamma,\delta')\cap K\Gamma/\Gamma$ for some $\delta'>0$. Since $K\Gamma/\Gamma$ is relatively compact, we can find $\delta$ and $\delta'$ that work for every point $x\Gamma$ in $K\Gamma/\Gamma$. Equally, there exists a radius $\varepsilon>0$ such that the map $g\mapsto gx\Gamma$ from $B_{G}(\varepsilon)$ to $G/\Gamma$ is a homeomorphism between $B_{G}(\varepsilon)$ and an open set in $G/\Gamma$ around $x\Gamma$. Once again, by the relative compactness of $K$, we may assume that $\varepsilon$ works for every point in $K\Gamma/\Gamma$. We may also assume that $\delta<\varepsilon$. Now, fix $r<\min\{\delta',\varepsilon\}/2$ and suppose by contradiction that the map in the lemma is not injective. Then, for some $h_{1},h_{1}'\in B_{H_{1}}(r)$ and $k,k'\in K$ one has 
$$h_{1}k\Gamma=h_{1}'k'\Gamma.$$
Since $h_{1}'^{-1}h_{1}\in B_{G}(2r)$, we have $d(k\Gamma,k'\Gamma)<2r$. Then, from $2r<\delta'$ we deduce that there exists $h_{2}\in B_{H_{2}}(\delta)$ such that $h_{2}k\Gamma=k'\Gamma$. Given that $2r<\varepsilon$ and $\delta<\varepsilon$, it must be $h_{2}=h_{1}'^{-1}h_{1}$, showing $h_{1}'=h_{1}$.
 We conclude by proving that the inverse of the map in the statement is continuous. Let $\varrho>0$ and let $h_{1}k\Gamma$ with $h_{1}\in B_{H_{1}}(r)$ and $k\in K$ be fixed. Consider a point $h_{1}'k'\Gamma$ with $h_{1}'\in B_{H_{1}}(r)$ and $k'\in K$ such that $d_{G}(h_{1}'k'\Gamma,h_{1}k\Gamma)<\varrho$. Assuming that $\varrho$ is small enough, there exists $g\in B_{G}(\varrho)$ such that
$$gh_{1}'k'\Gamma=h_{1}k\Gamma.$$
Since the right $\Gamma$-orbit of the group $H_{2}$ is closed, for small values of $\varrho$, it must also be $h_{1}^{-1}gh_{1}'\in H_{2}\cap B_{G}(2r+\varrho)$. Pick any sequence $\varrho_{n}\to 0$ and elements $k_{n}',h_{1,n}',$ and $g_{n}$ as above. By using the fact that $h_{1}^{-1}h_{1,n}'$ belongs to a compact set, we may extract a sub-sequence $h_{1}^{-1}g_{n_{k}}h_{1,n_{k}}'$ of elements in $H_{2}$ tending to some fixed element in $H_{1}$ (i.e., the limit of the sequence $h_{1}^{-1}h_{1,n_{k}}'$). Since $\overline{H}_{1}\cap \overline{H}_{2}=\{e\}$, we deduce that both $h_{1}^{-1}h_{1,n_{k}}'$ and $h_{1}^{-1}g_{n_{k}}h_{1,n_{k}}'$ must tend to $e$. Given that this is true for arbitrary sequences of radii $\varrho_{n}\to 0$, and elements $k_{n}'$ and $h_{1,n}',$ we deduce the continuity of the inverse map in both components.
\end{proof}

\end{lemma}

\section{Rational points}\label{rationals}
In this section, we introduce the notion of rational points in $G/\Gamma$ and prove some properties which will be used in the next few sections. For any $x\in G/\Gamma$, we will write $\Stab(x)$ for the stabilizer of $x$. Note that if $x=g\Gamma$, then $\Stab(x)=g\Gamma g^{-1}$. For any subset $\left\{a_i\right\}_{i\in I}\subset G$ where $I$ is an index set, we will write $\langle a_i\rangle_{i\in I}$ for the group generated by the set $\left\{a_i\right\}_{i\in I}$.

\begin{definition}\label{d31}
Let $\nu$ be the lowest\footnote{By lowest root we mean the only root $\nu$ such that $\nu+\alpha<0$ for all $\alpha\in\Phi\setminus\left\{\nu\right\}$.} root in $\Phi_-$ and $N_\nu$ the subgroup with Lie algebra $\mathfrak g_\nu$. We say that a point $p\in X=G/\Gamma$ is rational if $\Stab(p)\cap N_\nu\neq\left\{e\right\}$.
\end{definition}

This definition might sound unusual to the reader. However, it will be apparent that the points described in Definition \ref{d31} lying in $N_{+}\Gamma/\Gamma$ all have a $\Gamma$-representative with rational components (see Lemma \ref{l32}), whence the name. It can also be seen that this condition is equivalent to the one given in Definition \ref{d31} for points in $N_{+}\Gamma/\Gamma$, but we will not make use of this fact in the sequel. It should also be noted that an adaptation of the proof of Lemma \ref{l32} shows that Definition \ref{d31} is equivalent to the co-compactness of the lattice $\Stab(p)\cap N_{-}$ in $N_{-}$ for a point $p\in G/\Gamma$, which was our original definition of rational point. We will leave the proof of this fact to the interested reader. 

We now aim to give a characterization of rational points in $G/\Gamma$. The upshot of this section will be the following result.

\begin{proposition}\label{p31}
A point $p\in G/\Gamma$ is rational if and only if $p\in AN_-\Gamma/\Gamma$.
\end{proposition}

The proof of this proposition requires a few intermediate steps. We start with some auxiliary lemmas.

\begin{lemma}\label{l21}
We have that $G/\Gamma=AN_-N_+\Gamma/\Gamma.$
\end{lemma}

\begin{proof}
By Proposition \ref{prop:Bruhat}, we know that $$G=\bigsqcup_{i\in I}DN_-w_i N_-,$$ where the subset $\left\{w_i\right\}_{i\in I}$ is the set of representatives of the Weyl group $W$ defined above. Note that for each $w_i$, the subgroup $w_iN_-w_i^{-1}$ is a semi-direct product of two subgroups $F_i$ and $H_i$ in $N_-$ and $N_+$ respectively, i.e.,
$$w_iN_-w_i^{-1}=F_i\cdot H_i.$$
Therefore, we have 
\begin{align*}
DN_-w_iN_-\Gamma=&DN_-(w_iN_-w_i^{-1})\cdot w_i\Gamma\subset DN_-\cdot F_iH_i\cdot w_i\Gamma\subset DN_-N_+\Gamma.
\end{align*}
So $G/\Gamma=DN_-N_+\Gamma/\Gamma$. Finally, since $$D=A\cdot (D\cap\Gamma)$$ and $D\cap\Gamma$ normalizes $N_-$ and $N_+$, we have $G/\Gamma=AN_-N_+\Gamma/\Gamma$, as required.
\end{proof}

\begin{lemma}\label{l31}
Let $\nu$ be the lowest root in $\Phi_-$. Then, the normalizer of the subgroup $N_\nu$ in $G$ is equal to $DN_-$.
\end{lemma}

\begin{proof}
Let $g$ be an element in $G$ which normalizes $N_\nu$. By Proposition \ref{prop:Bruhat}, there exists a representative $w_i$ of the Weyl group $W$ such that $$g=an_1w_in_2$$ for some $a\in D$ and $n_1,n_2\in N_-$. Since $a$ normalizes $N_\nu$ and $n_1,n_2$ commute with $N_\nu$, it follows from $gN_\nu g^{-1}=N_\nu$ that $$w_i N_\nu w_i^{-1}=N_\nu.$$ This implies that $w_i=e$ is the representative of identity in $W$, and hence $g\in DN_-$.
\end{proof}

Henceforth, we will use the notation $H(\mathbb Q)$ to indicate $H\cap \textup{SL}_{3}(\mathbb Q)$ for any subgroup $H<G$.

\begin{lemma}\label{l32}
If $p=g\Gamma$ is a rational point and $g\in N_+$, then $g\in G(\mathbb Q)$.
\end{lemma}

\begin{proof}
Let $\nu$ be the lowest root in $\Phi_-$. Since $p$ is a rational point in $G/\Gamma$, the intersection $$\Stab(p)\cap N_{\nu}=(g\Gamma g^{-1})\cap N_{\nu}$$ contains an element $x\neq e$. Let $\gamma$ in $\Gamma\setminus\left\{e\right\}$ such that $ g\gamma g^{-1}=x$. Then, we have $$g\langle\gamma\rangle g^{-1}=\langle x\rangle.$$ Given that any element in $N_\nu\setminus\left\{e\right\}$ generates a Zariski dense subgroup in $N_\nu$, we also have 
\begin{equation}\label{eq1}
\overline{g\langle\gamma\rangle g^{-1}}^{Zar}=N_\nu,
\end{equation}
where $\overline{S}^{Zar}$ denotes the Zariski closure of a set $S$.

Now, let $\sigma\in\textup{Gal}(\mathbb R/\mathbb Q)$. Since $N_\nu$ is defined over $\mathbb Q$ and $\gamma$ is an integral matrix, by applying $\sigma$ to Equation~\eqref{eq1}, we have $$\overline{\sigma(g)\langle\gamma\rangle \sigma(g)^{-1}}^{Zar}=N_\nu.$$ Let $h=g\cdot\sigma(g)^{-1}$. Since $g\in N_+$ and $N_+$ is defined over $\mathbb Q$, we have that $\sigma(g)\in N_+$ and hence $h\in N_+$. It follows that $$h N_\nu h^{-1}=h\left(\overline{\sigma(g)\langle\gamma\rangle \sigma(g)^{-1}}^{Zar}\right)h^{-1}=\overline{g\langle\gamma\rangle g^{-1}}^{Zar}=N_\nu.$$ By Lemma~\ref{l31}, this implies that
$$h\in DN_-\cap N_+=\left\{e\right\}.$$ Hence, $\sigma(g)=g$ for all $\sigma\in\textup{Gal}(\mathbb R/\mathbb Q)$ and thus $g\in G(\mathbb Q)$, as required.
\end{proof}

Now, we deviate slightly and consider some results about lattices and flags in $\mathbb R^n$. We recall that a flag $\mathcal F$ in $\mathbb R^n$ is an ascending chain of vector subspaces of $\mathbb R^n$ $$P_0\subsetneq P_1\subsetneq P_2\subsetneq\cdots\subsetneq P_n=\mathbb R^n$$ such that $\dim P_i=i$ for $0\leq i\leq n$. Let $\Lambda$ be a lattice in $\mathbb R^n$. We say that a flag $\mathcal F$ is a $\Lambda$-flag if for any $0\leq i\leq n$ the intersection $\Lambda\cap P_i$ generates the subspace $P_i$ or, in other words, if one can find a basis of $P_i$ in $\Lambda$.

\begin{lemma}\label{l33}
Let $\Lambda$ be a lattice in $\mathbb R^n$ and let $\mathcal F:=P_0\subsetneq P_1\subsetneq P_2\subsetneq\cdots\subsetneq P_n=\mathbb R^n$ be a $\Lambda$-flag. Then, there exist vectors $v_1,v_2,...,v_n\in\Lambda$ such that $v_i\in P_i$ $(1\leq i\leq n)$ and $\left\{v_1,v_2,...,v_n\right\}$ is a basis of $\Lambda$ as a $\mathbb Z$-module.
\end{lemma}

\begin{proof}
We proceed by induction on $n$. The case $n=1$ is trivial. Assume $n>1$ and let $\Delta=P_{n-1}\cap\Lambda$. Then, $\Delta$ is a lattice in $P_{n-1}$ and $$\left\{0\right\}=P_0\subsetneq P_1\subsetneq P_2\subsetneq...\subsetneq P_{n-1}\cong\mathbb R^{n-1}$$ is a $\Delta$-flag in $P_{n-1}$. By the inductive hypothesis, we can find a $\mathbb{Z}$-basis $\left\{v_1,v_2,...,v_{n-1}\right\}$ of $\Delta$ such that $v_i\in P_i$ for $i=1,2,...,n-1$. Let $U=P_{n-1}^\perp$ be the orthogonal complement of $P_{n-1}$ in $\mathbb R^n$ and denote by $\pi_U$ the orthogonal projection onto $U$. Then, $\pi_U(\Lambda)$ must be a discrete subgroup in $U$, since otherwise, we would find infinitely many distinct vectors of $\Lambda$ inside a ball around a fundamental domain of $\Delta$. In view of this, we can find a vector $v_n\in\Lambda$ such that $\pi_U(v_n)$ is a generator of $\pi_U(\Lambda)$. Then, the set $\left\{v_1,v_2,...,v_n\right\}$ is a basis of $\Lambda$, as it generates a simplex of minimal volume.
\end{proof}

By taking $\Lambda=\mathbb Z^n\subset\mathbb R^n$ in Lemma~\ref{l33}, we deduce the following.

\begin{corollary}\label{cor31}
Let $\left\{0\right\}=P_0\subsetneq P_1\subsetneq P_2\subsetneq...\subsetneq P_n=\mathbb R^n$ be a rational flag in $\mathbb R^n$ (i.e., each $P_i$ is a vector space defined over $\mathbb Q$ and $\dim P_i=i$). Then, there exist integral vectors $v_1,v_2,...,v_n\in\mathbb Z^n$ such that $v_i\in P_i$ and $\left\{v_1,v_2,...,v_n\right\}$ is a basis of $\mathbb Z^n$.
\end{corollary}

Now we come back to study the rational points in $G/\Gamma$.
\begin{lemma}\label{l34}
The set of rational points in $N_+\Gamma/\Gamma$ is equal to $AN_-\Gamma/\Gamma\cap N_+\Gamma/\Gamma$.
\end{lemma}
\begin{proof}
Suppose that $g\Gamma\in AN_-\Gamma/\Gamma\cap N_+\Gamma/\Gamma$. Then $$g\Gamma=an_-\Gamma$$ for some $a\in A$ and $n_-\in N_-$. Let $x\in\Gamma\cap N_\nu\neq\left\{e\right\}$. Since $a$ normalizes $N_\nu$ and $n_-$ commutes with $N_\nu$, we have $$an_-x(an_-)^{-1}\in\Stab(g\Gamma)\cap N_\nu$$ and $\Stab(g\Gamma)\cap N_\nu\neq\left\{e\right\}$. By definition, this implies that $g$ is a rational point in $N_+\Gamma/\Gamma$.

Now suppose that $g\Gamma$ is a rational point in $N_+\Gamma/\Gamma$ $(g\in N_+)$. We need to show that $g\Gamma\in AN_-\Gamma/\Gamma$. By Lemma~\ref{l32}, we know that $g$ is a rational matrix. Write $$g=\left(\begin{array}{c} \alpha_1^T \\ \alpha_2^T \\ \alpha_3^T\end{array}\right)$$ where $\alpha_1,\alpha_2,\alpha_3$ are column vectors in $\mathbb Q^3$. We claim that one can find an integral matrix $\gamma$ of determinant 1 $$\gamma=\left(v_1,v_2,v_3\right)$$ such that $g\gamma$ is a lower triangular matrix. This would then imply that $g\Gamma\in DN_-\Gamma/\Gamma$ and $$g\Gamma\in DN_-\Gamma/\Gamma=A\cdot(D\cap\Gamma)N_-\Gamma/\Gamma=AN_-\Gamma/\Gamma$$ as required. 

We have $$g\gamma=(\alpha_i^T\cdot v_j)_{1\leq i,j\leq 3}.$$
Hence, in order to prove that $g\gamma\in DN_-$ for some $\gamma\in\Gamma$, it suffices to show that there is a basis $\left\{v_1,v_2,v_3\right\}$ of $\mathbb Z^3$ such that for each $j\in\left\{1,2,3\right\}$, the vector $v_j$ lies inside the rational vector space $P_j$ defined by $$\alpha_i^T\cdot v=0\mbox{ for }i=1,2,...,j-1.$$ This follows directly from Corollary~\ref{cor31} as $P_3\subsetneq P_2\subsetneq P_1$ is a rational flag in $\mathbb R^3$. 
\end{proof}

We are now ready to prove Proposition \ref{p31}. If $p=g\Gamma\in AN_-\Gamma/\Gamma$, by Lemma~\ref{l31}, one has that $\Gamma\cap N_{\nu}$ is conjugate to $g\Gamma g^{-1}\cap N_{\nu}$, hence, the latter is non-trivial and $p$ is rational. Conversely, let $p\in G/\Gamma$ be a rational point. By Lemma~\ref{l21}, we know that there exist $a\in A$, $n_-\in N_-$, and $n_+\in N_+$, such that $$p=an_-n_+\Gamma.$$ Since $an_-$ normalizes $N_\nu$ (see Lemma \ref{l31}), the point $n_+\Gamma$ is also rational in $G/\Gamma$. Then, Lemma~\ref{l32} implies $n_+\in G(\mathbb Q)$. By Lemma~\ref{l34}, we conclude that $n_+\Gamma\in AN_-\Gamma$, and $p=an_-n_+\Gamma\in AN_-\Gamma$, as required.

\section{Denominators and counting rational points}\label{counting}
In this section, we introduce the notion of denominator of a rational point (see Definition \ref{d31}) and count the number of rational points in a small open set with denominator in a given interval. We  denote by $\exp:\mathfrak g\to G$ the exponential map and by $\log: G\to\mathfrak g$ the logarithm map. We fix a norm $\|\cdot\|_{\mathfrak g}$ on the Lie algebra $\mathfrak g$ of $G$, assuming, without loss of generality, that this norm is induced by norms $\|\cdot\|_{\mathfrak g_\alpha}$ on $\mathfrak g_\alpha$ $(\alpha\in\Phi)$ and a norm  $\|\cdot\|_{\mathfrak a}$ on $\mathfrak a$. 

We will also write $f\lesssim g$ (or $g\gtrsim f$) if there exists a constant $C>0$ such that $$f\leq C\cdot g$$ and $f\sim g$ if $f\lesssim g$ and $g\lesssim f$. We will specify the constants in the context if necessary.

\begin{definition}\label{d41}
Let $\nu$ be the lowest root in $\Phi_-$ and let $N_\nu$ the subgroup in $G$ with Lie algebra $\mathfrak g_{\nu}$. Let $p$ be a rational point in $G/\Gamma$. The denominator $p$ is defined as $$d_p:=\inf_{v\in\Stab(p)\cap N_\nu\setminus\left\{e\right\}}\|\log v\|_{\mathfrak g_\nu}.$$
\end{definition}

In the sequel, to simplify the notation, we will interchangeably write $\alpha(a)$ to denote $\alpha(Y)$ for any $\alpha\in\Phi$ and $a\in A$ such that $a=\exp(Y)$ $(Y\in\mathfrak a)$. In particular, we will write $$\alpha(a_t):=\alpha(t\cdot X_0)$$ for any $\alpha\in\Phi$ and $t\in\mathbb R$.

\begin{lemma}\label{l41}
Let $p\in G/\Gamma$ be a rational point and suppose that $p=an_-\Gamma$ for some $a\in A$ and $n_-\in N_-$. Then, the element $a$ is unique and we have $d_{p}=\lambda e^{\nu(a)}$, where $\nu$ is the lowest root in $\Phi_-$ and $\lambda$ is a constant only depending on $\|\cdot\|_{\mathfrak g_{\nu}}$. \end{lemma}

\begin{proof}
By Proposition~\ref{p31}, we can always write $p=an_-\Gamma$ for $a\in A$ and $n_-\in N_-$. Suppose that $p=\tilde a\tilde n_-\Gamma$ for some $\tilde a\in A$ and $\tilde n_-\in N_-$. Then $$(\tilde a\tilde n_-)^{-1}an_-\in\Gamma$$ and $\tilde a^{-1}a$ is an integral matrix in $A$. This implies that $\tilde a^{-1}a=e$, that is $\tilde a=a$.

 Let $N_\nu$ be the subgroup with Lie algebra $\mathfrak g_\nu$ and let $x\in N_\nu\cap\Gamma\setminus\left\{e\right\}$ be a generator of $N_\nu\cap\Gamma$. Then, by Definition~\ref{d41}, one has $$d_p=\inf_{v\in\Stab(p)\cap N_\nu\setminus\left\{e\right\}}\|\log v\|_{\mathfrak g_\nu}=\|\log(an_-x(an_-)^{-1})\|_{\mathfrak g_\nu}=\|x\|_{\mathfrak g_{\nu}}e^{\nu(a)}$$ as required.
\end{proof}

\begin{remark}
From now on, we will assume $\lambda=1$, by fixing an appropriate norm $\|\cdot\|_{\mathfrak g_{\nu}}$ on $\mathfrak g_{\nu}$. The element $a$ in the equation $p=an_-\Gamma$ will be called the polar-component of the rational point $p$, and will be denoted by $a_p$.
\end{remark}

Let $U\subset N_+$ be a small open subset in $N_+$. We denote by $S(U\Gamma/\Gamma,C_1,C_2)$ the subset of rational points in $U\Gamma/\Gamma$ whose denominators lie between $C_1$ and $C_2$ $(C_1<C_2)$. In what follows, we aim to estimate, under some constraints, the number of rational points $p$ in $S(U\Gamma/\Gamma,C_1,C_2)$.

We start by introducing some notation. We observe that
$$\mathfrak a=\ker(\nu)\oplus\mathbb R\cdot X_0,$$
where $a_{t}=\exp(tX_{0})$. We let $\pi_{\ker(\nu)}$ be the projection map from $\mathfrak a$ onto the subspace $\ker(\nu)$, and for $x\in\mathfrak a$ we write $$x=x_{\nu^\perp}+x_\nu$$ where $x_{\nu^\perp}\in\ker(\nu)$ and $x_\nu\in\mathbb R\cdot X_0$. For any two compact subsets $K_1\subset\ker(\nu)$ and $K_2\subset\mathbb R\cdot X_0$, we define $$\mathfrak a_{K_1,K_2}=\left\{x\in\mathfrak a: x=x_{\nu^\perp}+x_{\nu},\;x_{\nu^\perp}\in K_1,\;x_{\nu}\in K_2\right\}$$
and
$$S_{K_{1}}(U\Gamma/\Gamma,A,B)=\left\{p\in U\Gamma/\Gamma: p\in S(U\Gamma/\Gamma,A,B),\;\pi_{\ker(\nu)}(a_p)\in K_{1}\right\}.$$
Note that $|S_{K_{1}}(U\Gamma/\Gamma,A,B)|\leq |S(U\Gamma/\Gamma,A,B)|$, and for any compact exhaustion $\left\{K_{n}\right\}$ of $\ker\nu$ one has
$$\lim_{n\to\infty}|S_{K_{n}}(U\Gamma/\Gamma,A,B)|=|S(U\Gamma/\Gamma,A,B)|.$$

Proving the following result will be the goal of this section.

\begin{proposition}\label{p41}
Let $\nu$ be the lowest root in $\Phi_-$. Let $U\subset N_+$ be a small open bounded subset in $N_+$, with boundary of measure zero, and let $K$ be a compact subset in $\ker(\nu)$. Then, for all sufficiently large $l>0$, we have that $$|S_K(U\Gamma/\Gamma,l/2,l)|\sim l^2\cdot\mu_{N_+}(U),$$ where $\mu_{N_+}(n)$ is the Haar measure on $N_{+}$ and the implicit constants depend only on $K$, $G$, and $\Gamma$.
\end{proposition}

In order to prove Proposition \ref{p41}, we will make use of the following lemma
\begin{lemma}
\label{lem:equidistr}
Let $W\subset G/\Gamma$ be an open bounded subset, with boundary of measure zero, and let $\chi_{W}$ denote the characteristic function of $W$. Then, we have $$\lim_{t\to\infty}\frac1{\mu_{N_+}(U)}\int_{U}\chi_W(a_t\cdot nx)d\mu_{N_+}(n)=\int_{G/\Gamma}\chi_W d\mu_{G/\Gamma},$$ where $\mu_{N_+}(n)$ is the Haar measure on $N_{+}$ and $\mu_{G/\Gamma}$ is the invariant probability measure on $G/\Gamma$.
\end{lemma}

\begin{proof}
It is well-known that the one-parameter diagonal flow $\left\{a_t\right\}_{t\in\mathbb R}$  is mixing on $G/\Gamma$, i.e., for any $f,g\in L^2(G/\Gamma)$ it holds
$$\lim_{t\to\infty}\int_{G/\Gamma}f(a_t\cdot x)g(x)d\mu_{G/\Gamma}(x)=\int_{G/\Gamma}f(x)d\mu_{G/\Gamma}(x)\cdot\int_{G/\Gamma}g(x)d\mu_{G/\Gamma}(x).$$ Then, by \cite[Propostion 2.2.1]{KM96}), it follows that for any $x\in G/\Gamma$, any open bounded subset $U\subset N_+$, and any compactly supported continuous function $f$ on $G/\Gamma$ $$\lim_{t\to\infty}\frac1{\mu_{N_+}(U)}\int_{U}f(a_t\cdot nx)d\mu_{N_+}(n)=\int_{G/\Gamma}fd\mu_{G/\Gamma}.$$
The proof follows by an approximation argument.
\end{proof}

The remainder of this section will be dedicated to prove Proposition \ref{p41}. Let $\nu$ be the lowest root in $\Phi_-$. For any $l>1$, define $\tau:=\tau(l)>0$ such that $$\nu(a_\tau)=-\ln l.$$ Let $p$ be a rational point in $N_+\Gamma/\Gamma$. By Proposition~\ref{p31} and Lemma~\ref{l41}, we may write $$p=a_pn_-\Gamma$$ where $a_p$ is the polar-component of $p$ and $n_-\in N_-$.  Then, we have
\begin{align*}
p\in S_K(U\Gamma/\Gamma,l,l/2)\iff& p\in U\Gamma/\Gamma,\;l/2\leq d_p\leq l,\textrm{ and }\pi_{\ker(\nu)}(a_p)\in K\\
\iff& p=a_pn_-\Gamma\in U\Gamma/\Gamma,\; l/2\leq e^{\nu(a_p)}\leq l,\textrm{ and }\pi_{\ker(\nu)}(a_p)\in K\\
\iff& a_\tau\cdot p=(a_{\tau}a_p)n_-\Gamma\in a_{\tau}\cdot U\Gamma/\Gamma,\\
& -\ln 2\leq\nu(a_{\tau}a_p)\leq 0,\textrm{ and }\pi_{\ker(\nu)}(a_{\tau}a_p)\in K\\
\iff& a_{\tau}\cdot p\in a_{\tau}\cdot U\Gamma/\Gamma\textrm{ and } a_{\tau}\cdot p\in\exp(\mathfrak a_{K, I_0})N_{-}\Gamma/\Gamma,
\end{align*}
where $I_0$ is the following compact interval in the Lie algebra $\mathbb R\cdot X_0$ of $\left\{a_t\right\}_{t\in\mathbb R}$ $$I_0:=\left\{x\in\mathbb R\cdot X_0: -\ln2\leq\nu(x)\leq0\right\}.$$ This implies that $$a_\tau\cdot S_K(U\Gamma/\Gamma,l/2,l)=a_{\tau}\cdot U\Gamma/\Gamma\cap\exp(\mathfrak a_{K,I_0})N_{-}\Gamma/\Gamma.$$ Since $\exp(\mathfrak a_{K,I_0})N_{-}\Gamma/\Gamma$ is a compact subset in $G/\Gamma$, by Lemma \ref{lem:transversality}, there exist $\delta>0$ and a small neighborhood of identity $B(\delta)$ in $N_+$ $$B(\delta)=\left\{\exp(x): x=\sum_{\alpha\in\Phi_+} x_\alpha, x_\alpha\in\mathfrak g_\alpha, \|x_\alpha\|_{\mathfrak g}<\delta\right\}$$ such that $$B(\delta)\times\exp(\mathfrak a_{K,I_0})N_{-}\Gamma/\Gamma\to B(\delta)\exp(\mathfrak a_{K,I_0})N_{-}\Gamma/\Gamma$$ is a homeomorphism.
We conclude that for any two rational points $p,q$ in $S_K(U\Gamma/\Gamma,l/2,l)$ the subsets $$B(\delta)\cdot a_{\tau}\cdot p\textrm{ and }B(\delta)\cdot a_{\tau}\cdot q$$ are disjoint. Consequently, also the sets $B_{\delta,\tau}\cdot p$ and $B_{\delta,\tau}\cdot q$ are disjoint, where $$B_{\delta,\tau}=a_{-\tau}\cdot B(\delta)\cdot a_{\tau}.$$ 

Now, we can estimate the number of points in $S_K(U\Gamma/\Gamma, l,l/2)$. First, we prove an upper bound for $|S_K(U\Gamma/\Gamma,l,l/2)|$. Fix a sufficiently small number $\epsilon_0>0$ such that $$\mu_{N_+}(U)\leq\mu_{N_+}(B(\epsilon_0)\cdot U)\leq 2\mu_{N_+}(U).$$ Then for sufficiently large $l>0$, we have $$B_{\delta,\tau}\subset  B(\epsilon_0)\textup{ and }B_{\delta,\tau}\cdot S_K(U\Gamma/\Gamma,l,l/2)\subset B(\epsilon_0)\cdot U\Gamma/\Gamma.$$ Since $\left\{B_{\delta,\tau}\cdot p:p\in S_K(U\Gamma/\Gamma, l/2, l)\right\}$ is a collection of disjoint subsets, by the mixing property of $\left\{a_t\right\}_{t\in\mathbb R}$, we deduce that 
\begin{align*}
|S_K(U\Gamma/\Gamma,l/2,l)|\mu_{N_+}(B_{\delta,\tau})\leq&\int_{B(\epsilon_0)\cdot U} \chi_{B(\delta)\exp(\mathfrak a_{K,I_0})N_{-}\Gamma/\Gamma}(a_{\tau}u\Gamma)d\mu_{N_+}(u)\\
\sim&\mu_{N_+}(B(\epsilon_0)\cdot U)\cdot\mu_{G/\Gamma}(B(\delta)\exp(\mathfrak a_{K,I_0})N_{-}\Gamma/\Gamma)\\
\sim&\mu_{N_+}(U)\cdot\mu_{N_+}(B(\delta))
\end{align*}
as $\tau\to\infty$. Note that $\mu_{N_+}(B_{\delta,\tau})=\mu_{N_+}(B(\delta)) e^{2\nu(a_\tau)}=\mu_{N_+}(B(\delta))/{l^2}$. So $$|S_K(U\Gamma/\Gamma,l/2,l)|\lesssim l^2\cdot\mu_{N_+}(U),$$ where the implicit constant depends only on $K$, $G$ and $\Gamma$.

Now, we proceed to prove a lower bound for $|S_K(U\Gamma/\Gamma, l/2, l)|$. We fix a sufficiently small number $\epsilon_0>0$ and an open subset $\tilde U$ inside $U$ such that $$B(\epsilon_0)\cdot\tilde U\subset U\textup{ and }\mu_{N_+}(\tilde U)\geq\frac12\mu_{N_+}(U).$$ Then for sufficiently large $l>0$ we have $B_{\delta,\tau}\subset B(\epsilon_0)$. We also have
\begin{align*}
&u\in\tilde U\textup{ and }\chi_{B(\delta)\exp(\mathfrak a_{K,I_0})N_{-}\Gamma/\Gamma}(a_{\tau}u\Gamma)=1\\
\implies& u\in\tilde U\textup{ and }u\Gamma\in (a_{-\tau}B(\delta)a_{\tau})\cdot a_{-\tau}\exp(\mathfrak a_{K,I_0})N_{-}\Gamma/\Gamma\\
\implies& u\in\tilde U\textup{ and there exists }p\in S_K(U\Gamma/\Gamma,l/2,l)\textup{ such that }u\in B_{\delta,\tau}\cdot p.
\end{align*}
By the mixing property of $\left\{a_t\right\}_{t\in\mathbb R}$, this implies that
\begin{align*}
|S_K(U\Gamma/\Gamma,l/2,l)|\mu_{N_+}(B_{\delta,\tau})\geq&\int_{\tilde U} \chi_{B(\delta)\exp(\mathfrak a_{K,I_0})N_{-}\Gamma/\Gamma}(a_{\tau}u\Gamma)d\mu_{N_+}(u)\\
\sim&\mu_{N_+}(\tilde U)\cdot\mu_{G/\Gamma}(B(\delta)\exp(\mathfrak a_{K,I_0})N_{-}\Gamma/\Gamma)\\
\sim&\mu_{N_+}(U)\cdot\mu_{N_+}(B(\delta))
\end{align*}
as $\tau\to\infty$. Note that $\mu_{N_+}(B_{\delta,\tau})=\mu_{N_+}(B(\delta)) e^{2\nu(a_\tau)}=\mu_{N_+}(B(\delta))/{l^2}$. So $$|S_K(U\Gamma/\Gamma,l/2,l)|\gtrsim l^2\cdot\mu_{N_+}(U),$$ where the implicit constant depends only on $K$, $G$ and $\Gamma$. This completes the proof of the proposition.

\section{A formula for denominator of rational points in $N_{+}\Gamma/\Gamma$}\label{denominators}

In this section, we deduce an explicit formula for the denominator of rational points lying in $N_+\Gamma/\Gamma$, which will be crucial in various counting results in \S\ref{estimates}. The formula that we aim to prove reads as follows.

\begin{theorem}\label{thm51}
Let $p=g\Gamma$ $(g\in N_+)$ be a rational point, with $$g=\left(\begin{array}{ccc}1 & \frac ab & 0\\ 0 & 1 & 0\\ 0 & 0 & 1\end{array}\right)\left(\begin{array}{ccc}1 & 0 & \frac{p_1}q\\ 0 & 1 & \frac{p_2}q\\ 0 & 0 & 1\end{array}\right),$$ where $\gcd(a,b)=1$ and $\gcd(p_1,p_2,q)=1$. Then, we have $$d_p=\frac{bq^2}{d},$$ where $d=\gcd(bq, aq, bp_1+ap_2)=\gcd(q, bp_1+ap_2)$.
\end{theorem}

\begin{proof}
By Lemma~\ref{l32} and Lemma~\ref{l34}, we know that a point $p=g\Gamma$ $(g\in N_+)$ is rational if and only if the entries of $g$ are rational and $$p\in AN_-\Gamma\cap N_+\Gamma.$$ If $p=g\Gamma=an_-\Gamma$ for some $a\in A$ and $n_-\in N_-$, where $$a=\left(\begin{array}{ccc} e^{t_1} & 0 & 0\\ 0 & e^{t_2} & 0\\ 0 & 0 & e^{t_3}\end{array}\right),$$ by Lemma~\ref{l41}, the denominator $d_p$ of $p$ is equal to $e^{\nu(a)}=e^{t_3-t_1}$.

On the other hand, we may write $$g=\left(\begin{array}{ccc}1 & \frac ab & 0\\ 0 & 1 & 0\\ 0 & 0 & 1\end{array}\right)\left(\begin{array}{ccc}1 & 0 & \frac{p_1}q\\ 0 & 1 & \frac{p_2}q\\ 0 & 0 & 1\end{array}\right)$$ where $\gcd(a,b)=1$ and $\gcd(p_1,p_2,q)=1$. Thus, there must be a matrix $\gamma$ in $\SL_3(\mathbb Z)$ $$\gamma=\left(\begin{array}{ccc}a_1 & b_1 & c_1\\ a_2 & b_2 & c_2\\ a_3 & b_3 & c_3\end{array}\right),$$ such that $g\cdot\gamma$ is a lower triangular matrix, i.e. $$g\cdot\gamma=\left(\begin{array}{ccc}1 & \frac ab & 0\\ 0 & 1 & 0\\ 0 & 0 & 1\end{array}\right)\left(\begin{array}{ccc}1 & 0 & \frac{p_1}q\\ 0 & 1 & \frac{p_2}q\\ 0 & 0 & 1\end{array}\right)\cdot\left(\begin{array}{ccc}a_1 & b_1 & c_1\\ a_2 & b_2 & c_2\\ a_3 & b_3 & c_3\end{array}\right)=\left(\begin{array}{ccc} * & 0 & 0\\ * & * & 0\\ * & * & *\end{array}\right).$$

To prove Theorem \ref{thm51}, we will follow the strategy used to prove Lemmas \ref{l33} and \ref{l34}. We consider the rational flag $$\mathcal F:=\left\{0\right\}\subsetneq P_1\subsetneq P_2\subsetneq\mathbb R^3$$ defined as follows: $$P_1:\begin{cases} q\cdot x+p_1\cdot z=0 \\ q\cdot y+p_2\cdot z=0\end{cases}$$ $$P_2: bq\cdot x+aq\cdot y+(bp_1+ap_2)z=0,$$
and we aim to find a basis 
$$\left\{(c_1,c_2,c_3)^{T},(b_1,b_2,b_3)^{T},(a_1,a_2,a_3)^{T}\right\}$$
of $\mathbb{Z}^{3}$ with vectors in this flag. 
The vector $(c_1,c_2,c_3)^T$ must be primitive in the rational line $P_1$. We choose $$(c_1,c_2,c_3)^T:=(-p_1,-p_2,q)^T.$$ Then, we have
\begin{align*}
g\cdot\gamma=&\left(\begin{array}{ccc}1 & \frac ab & 0\\ 0 & 1 & 0\\ 0 & 0 & 1\end{array}\right)\left(\begin{array}{ccc}1 & 0 & \frac{p_1}q\\ 0 & 1 & \frac{p_2}q\\ 0 & 0 & 1\end{array}\right)\cdot\left(\begin{array}{ccc}a_1 & b_1 & -p_1\\ a_2 & b_2 & -p_2\\ a_3 & b_3 & q\end{array}\right)\\
=&\left(\begin{array}{ccc}1 & \frac ab & 0\\ 0 & 1 & 0\\ 0 & 0 & 1\end{array}\right)\left(\begin{array}{ccc}a_1+a_3\cdot\frac{p_1}q & b_1+b_3\cdot\frac{p_1}q & 0\\ a_2+a_3\cdot\frac{p_2}q & b_2+b_3\cdot\frac{p_2}q & 0\\ a_3 & b_3 & q\end{array}\right)
\end{align*}
Next, we consider the rational plane $P_2$. Since $(c_1,c_2,c_3)^T$ is already a primitive integral point in $P_2$, we simply choose the vector $(b_1,b_2,b_3)^T$ to be another primitive integral vector in $P_2$ such that $\left\{(b_1,b_2,b_3)^T,(c_1,c_2,c_3)^T\right\}$ forms a basis of $P_2\cap\mathbb Z^3$. We also choose the vector $(a_1,a_2,a_3)^T\in\mathbb Z^3$ so that $\left\{(c_1,c_2,c_3),(b_1,b_2,b_3),(a_1,a_2,a_3)\right\}$ is a basis of $\mathbb Z^3$ with positive orientation. This is possible in view of Lemma \ref{l34}. Consequently, we have 
\begin{align*}
g\cdot\gamma=&\left(\begin{array}{ccc}1 & \frac ab & 0\\ 0 & 1 & 0\\ 0 & 0 & 1\end{array}\right)\left(\begin{array}{ccc}1 & 0 & \frac{p_1}q\\ 0 & 1 & \frac{p_2}q\\ 0 & 0 & 1\end{array}\right)\cdot\left(\begin{array}{ccc}a_1 & b_1 & -p_1\\ a_2 & b_2 & -p_2\\ a_3 & b_3 & q\end{array}\right)\\
=&\left(\begin{array}{ccc}1 & \frac ab & 0\\ 0 & 1 & 0\\ 0 & 0 & 1\end{array}\right)\left(\begin{array}{ccc}a_1+a_3\cdot\frac{p_1}q & b_1+b_3\cdot\frac{p_1}q & 0\\ a_2+a_3\cdot\frac{p_2}q & b_2+b_3\cdot\frac{p_2}q & 0\\ a_3 & b_3 & q\end{array}\right)\\
=&\left(\begin{array}{ccc}(a_1+a_3\cdot\frac{p_1}q)+\frac ab(a_2+a_3\cdot\frac{p_2}q) & 0 & 0\\ a_2+a_3\cdot\frac{p_2}q & b_2+b_3\cdot\frac{p_2}q & 0\\ a_3 & b_3 & q\end{array}\right).
\end{align*}
Then, the denominator $d_p$ of $p$ is equal to 
\begin{align*}
d_p=&\left|\frac{q}{(a_1+a_3\cdot\frac{p_1}q)+\frac ab(a_2+a_3\cdot\frac{p_2}q)}\right|\\
=&\left|\frac{bq^2}{a_1\cdot bq+a_2\cdot aq+a_3\cdot(bp_1+ap_2)}\right|.
\end{align*}

We now aim to find an alternative expression for $$a_1\cdot bq+a_2\cdot aq+a_3\cdot(bp_1+ap_2).$$ Set $u=(a_1,a_2,a_3)^T$ and $v=(bq, aq, bp_1+ap_2)^T$. Then, $$a_1\cdot bq+a_2\cdot aq+a_3\cdot(bp_1+ap_2)=u^T\cdot v.$$ Note that the vector $v$ is orthogonal to the rational plane $P_2$ and, hence, to the vectors $(b_1,b_2,b_3)$ and $(c_1,c_2,c_3)$. For convenience, we denote by $\Delta$ the discrete subgroup $P_2\cap\mathbb Z^3$ generated by $(b_1,b_2,b_3)$ and $(c_1,c_2,c_3)$, and by $d$ the greatest common divisor of $(bq, aq, bp_1+ap_2)$.

First, we compute the co-volume $\Vol(\Delta)$ of $\Delta$ in $P_2$. Let $\rho:\mathbb R^3\to\mathbb R$ be the map defined by $$\rho(x)=x^T\cdot v.$$ By Bezout's theorem, $\rho$ defines a surjective homomorphism, which we denote by $\sigma$, from $\mathbb Z^3$ to $d\cdot\mathbb Z$, such that $\sigma(v)=\|v\|^2$. By definition, we have $\ker\sigma=\Delta$ and $$\tilde\Delta:=\sigma^{-1}(\|v\|^2\cdot\mathbb Z)=\mathbb Z\cdot v\oplus\Delta.$$ We know that $\tilde\Delta=\sigma^{-1}(\|v\|^2\cdot\mathbb Z)$ is a subgroup of finite index in $\mathbb Z^3$ and its co-volume in $\mathbb R^3$ is given by $$\|v\|\cdot\Vol(\Delta)=\Vol(\tilde\Delta)=[\mathbb Z^3:\tilde\Delta]=[d\cdot\mathbb Z:\|v\|^2\cdot\mathbb Z]=\frac{\|v\|^2}d,$$
where the third equality follows from the fact that $\mathbb Z^{3}/\tilde{\Delta}$ is isomorphic to $\textup{Im}(\sigma)/\|v\|^{2}\mathbb{Z}$. Thus, we have $$\Vol(\Delta)=\frac{\|v\|}d.$$

Note that $\Delta$ coincides with $\mathbb{Z}^{3}\cap P_2$, and $u$ and $\Delta$ generate $\mathbb Z^3$ as a $\mathbb Z$-module. Since $\Vol(\mathbb Z^3)=1$, one can conclude that the distance between $u$ and the plane $P_2$ is equal to $1/\Vol(\Delta)=\frac d{\|v\|}$. As $v$ is orthogonal to $P_2$, we deduce that $$u^T\cdot v=\frac d{\|v\|}\cdot\|v\|=d.$$ Hence, $$d_p=\frac{bq^2}{d}.$$ We remark that here we have used some of the arguments at page 2 of \cite{AES16}.
\end{proof}

\section{Diophantine points and Diophantine approximation}\label{Diophantine}

In this section, we will study the different properties of Diophantine points defined in \S\ref{intro}. Note that, once we have proved Theorem~\ref{mthm} for $\gamma>0$, the case $\gamma=0$ will follow immediately, as for any open subset $U\subset N_{+}\Gamma/\Gamma$ and for any $\gamma>0$ one has $$\dim_H(S_\gamma^c\cap U)\leq\dim_H(S_0^c\cap U)\leq\dim_H N_{+}\Gamma/\Gamma=3.$$ Therefore, in the following, we will assume that $\gamma>0$.

\subsection{Preliminaries}
It is well-known that the fundamental domains of $G/\Gamma$ can be written as Siegel domains in $G$ \cite{B69,BH62}. Let $\Phi$ be the root system of $G$, let $\Phi_+$ the set of positive roots, and let $\Phi_-$ the set of negative roots as defined in \S\ref{intro} and \S\ref{pre}. Let $\Delta=\left\{\alpha_0,\beta_0\right\}$ be the set of simple roots. Let also $K_0=\SO(3)$ be the standard maximal compact subgroup in $G$. For any $t\in\mathbb R$, we define $$A_t=\left\{a\in A:\alpha_0(a)\geq t,\;\beta_0(a)\geq t\right\}.$$ Then, a Siegel domain is a subset of $G$ of the form $$\mathcal S_{t,\Omega}=K_0\cdot A_t\cdot\Omega$$ for some $t\in\mathbb R$ and some compact subset $\Omega$ in $N_-$. By \cite{B69,BH62}, we know that there exist $t_0\in\mathbb R$ and a compact subset $\Omega_0$ in $N_-$ such that $$G/\Gamma=\mathcal S_{t_0,\Omega_0}\cdot\Gamma.$$ Note that by the definition of $A_{t_0}$, the subset $$\left\{ana^{-1}:a\in A_{t_0}, n\in\Omega_0\right\}$$ is a compact subset in $N_-$. Hence, there exists a compact subset $L_0\subset G$ depending only on $K_0$, $t_0$, and $\Omega_0$, such that $$K_0\cdot\left\{ana^{-1}:a\in A_{t_0}, n\in\Omega_0\right\}\subset L_0\textup{ and }\mathcal S_{t_0,\Omega_0}\subset L_0\cdot A_{t_0}.$$

In the following, we will fix a constant $\kappa>1$ depending only on $G$, a sufficiently small neighborhood of identity $B(r_0)$ of radius $r_0>0$ in $G$ and a neighborhood $\mathcal B(r_0)$ of 0 in $\mathfrak g$ such that\vspace{2mm}
\begin{enumerate}
\item the exponential map $\exp:\mathcal B(r_0)\to B(r_0)$ is a diffeomorphism;\vspace{2mm}
\item for any $x\in\mathcal B(r_0)$ we have $$\frac1\kappa\cdot d_G(\exp(x),e)\leq\|x\|_{\mathfrak g}\leq\kappa\cdot d_G(\exp(x),e);$$\vspace{2mm}
\item for any $x\in\mathfrak g$, with $x=x_0+\sum_{\alpha\in\Phi}x_\alpha\;(x_0\in\mathfrak g_0, x_\alpha\in\mathfrak g_\alpha)$, we have $$\|x_0\|_{\mathfrak g},\;\|x_\alpha\|_\mathfrak g\leq\kappa\cdot\|x\|_{\mathfrak g}\quad(\forall\alpha\in\Phi).$$
\end{enumerate}
As in \S\ref{intro}, we write $$\mathfrak n_-=\sum_{\alpha\in\Phi_-}\mathfrak g_\alpha\textup{ and }\mathfrak n_+=\sum_{\alpha\in\Phi_+}\mathfrak g_\alpha.$$

We now state and prove two lemmas that will provide a characterisation of non-Diophantine points.

\begin{lemma}\label{l61}
Let $p\in G/\Gamma$ and suppose that $$\eta(p)<r$$ for a sufficiently small constant $r<r_0$. Then, there exists a constant $C>0$ depending only on $G/\Gamma$ and an element $v\in\Stab(p)\setminus\left\{e\right\}$ such that\vspace{2mm}
\begin{enumerate}
\item $d_G(v,e)<C\cdot r$;\vspace{2mm}
\item $v$ is primitive\footnote{By a primitive element of $\Stab(p)$, we mean an element $g$ that is not the power of any other element $h$ in $\Stab(p)$.} in $\textup{Stab}(p)\cap\Gamma$ and is conjugate to an element of $N_\nu$, where $\nu$ is the lowest root in $\Phi_-$ and $N_\nu=\exp(\mathfrak g_\nu)$.
\end{enumerate}
\end{lemma}
\begin{proof}
From the  above discussion, we know that if $p=k\cdot a\cdot n\Gamma$ for some $k\in K_0, a\in A_{t_0}$, and $n\in\Omega_0$, then $$k(ana^{-1})\in L_0\textup{ and }p=k(ana^{-1})\cdot a\Gamma.$$ Since $L_0$ is a compact subset in $G$, there exists a constant $C_1>0$ depending only on $L_0$ such that $$\eta(a\Gamma)\leq C_1\cdot\eta(p)<C_1\cdot r.$$ Let $v_1\in\Gamma\setminus\left\{e\right\}$ such that 
\begin{align}\label{eq2}
d_G(av_1a^{-1},e)=\eta(a\Gamma)<C_1\cdot r.
\end{align}
Since $r$ is arbitrarily small and $v_{1}\in\Gamma$, Inequality (\ref{eq2}) shows that at least one among $\alpha_0(a)$ and $\beta_0(a)$ must tend to infinity and that $v_1\in N_-$.

Let $v_2$ be a generator of the discrete subgroup $\Gamma\cap N_{\nu}$ in $N_\nu$. Then, we have $$av_2 a^{-1}=e^{-\alpha_0(a)-\beta_0(a)}v_2$$ and hence, $$d_G(av_2a^{-1},e)=e^{-\alpha_0(a)-\beta_0(a)}.$$ On the other hand, if we write $$v_1=\exp(x_{-\alpha_0}+x_{-\beta_0}+x_{-\alpha_0-\beta_0})\in N_-=\exp(\mathfrak n_-)$$ for $x_{-\alpha_0}\in\mathfrak g_{-\alpha_0}, x_{-\beta_0}\in\mathfrak g_{-\beta_0}, x_{-\alpha_0-\beta_0}\in\mathfrak g_{-\alpha_0-\beta_0}$, we deduce $$d_G(av_1a^{-1},e)\geq\frac1{\kappa^2}\max\left\{e^{-\alpha_0(a)}\|x_{-\alpha_0}\|_{\mathfrak g},e^{-\beta_0(a)}\|x_{-\beta_0}\|_{\mathfrak g},e^{-\alpha_0(a)-\beta_0(a)}\|x_{-\alpha_0-\beta_0}\|_{\mathfrak g}\right\}.$$ By the definition of $A_{t_0}$ and the fact that $v_1\in\Gamma\setminus\left\{e\right\}$, we can conclude that there exists a constant $C_2>0$ depending only on $G$ and $\Gamma$ such that $$d_G(av_2a^{-1},e)\leq C_2\cdot d_G(av_1a^{-1},e).$$ Now let $h=k(ana^{-1})\in L_0$. Then, there exists a constant $C_3>0$, depending only on $L_0$, such that $$\|\Ad(h)(x)\|_{\mathfrak g}\leq C_3\cdot\|x\|_{\mathfrak g}$$ for any $x\in\mathfrak g$. If an element $\exp(x)\in G$ $(x\in\mathfrak g)$ is sufficiently close to identity $e$, we have $$d_G(h\exp(x)h^{-1},e)\leq\kappa\|\Ad(h)(x)\|_{\mathfrak g}\leq\kappa C_3\|x\|_{\mathfrak g}\leq \kappa^2C_3\cdot d_G(\exp(x),e).$$ From here, we can deduce that, if $r$ is sufficiently small,
\begin{align*}
d_G(hav_2a^{-1}h^{-1},e)\leq\kappa^2C_3\cdot d_G(av_2a^{-1},e)\leq\kappa^2C_3C_2d_G(av_1a^{-1},e)<\kappa^2C_3C_2C_1\cdot r
\end{align*}
and $hav_2a^{-1}h^{-1}=(kan)v_2(kan)^{-1}\in\Stab(p)$. Set $v=hav_2a^{-1}h^{-1}$. Then, $v$ satisfies $(1)$ with $C=\kappa^2C_3C_2C_1$. Note that if, by contradiction, $v=u^{m}$, with $u\in\textup{Stab}(p)$ and $m\in\mathbb{N}$ ($m\geq 2$), we would have $\Ad(kan)u\in N_{\nu}\cap\Gamma$, with $(\Ad(kan)u)^{m}=v_{2}$. However, this is impossible, as $v_{2}$ generates $N_{\nu}\cap\Gamma$. This shows $(2)$, completing the proof of the lemma.
\end{proof}

\begin{lemma}\label{l62}
Let $p\in G/\Gamma$ and $\gamma>0$. Suppose that $p$ is not Diophantine of type $\gamma$ and that $\Stab(p)\cap N_-=\left\{e\right\}$. Then, there exist a sequence $\left\{t_n\right\}\to\infty$, and a sequence $v_n\in\Stab(a_{t_n}\cdot p)\setminus\left\{e\right\}$, and a constant $\kappa'$ only depending on $\kappa$, $\gamma$, and $(\alpha_{0}+\beta_{0})(X_{0})$, such that\vspace{2mm}
\begin{enumerate}
\item $v_n\in B(r_0)$, and if we write $$v_n=\exp(x_{n,-}+x_{n,0}+x_{n,+})$$ for $x_{n,-}\in\mathfrak n_-$, $x_{n,0}\in\mathfrak g_0$ and $x_{n,+}\in\mathfrak n_+$, we have $$\frac1{\kappa'}e^{-\gamma t_n}\leq\|x_{n,-}\|_{\mathfrak g}\leq\kappa' e^{-\gamma t_n},\|x_{n,0}\|_{\mathfrak g}\leq\kappa' e^{-\gamma t_n},\mbox{ and }\|x_{n,+}\|_{\mathfrak g}\leq\kappa' e^{-\gamma t_n};$$\vspace{2mm}
\item $v_n$ is primitive and is conjugate to an element of $N_\nu\cap\Gamma$, where $\nu$ is the lowest root in $\Phi_-$.
\end{enumerate}
\end{lemma}
\begin{proof}
By definition, if $p$ is not Diophantine of type $\gamma$, then for any $\epsilon>0$ there exists a sequence $t_n\to\infty$ such that $$\eta(a_{t_n}\cdot p)\leq\epsilon e^{-\gamma t_n}.$$ We may choose $\epsilon=1/C$, where $C>0$ is the constant in part $(1)$ of Lemma~\ref{l61}. By part $(2)$ of Lemma~\ref{l61}, for each $t_n$, there exists an element $u_n\in\Stab(p)\setminus\left\{e\right\}$ such that\vspace{2mm}
\begin{enumerate}
\item $a_{t_n}u_na_{t_n}^{-1}\in\Stab(a_{t_n}\cdot p)$;\vspace{2mm}
\item $d_G(a_{t_n}u_na_{t_n}^{-1},e)\leq C\cdot \epsilon e^{-\gamma t_n}=e^{-\gamma t_n}$;\vspace{2mm}
\item $u_n$ is primitive in $\textup{Stab}(p)$ and is conjugate to an element of $N_\nu$, where $\nu$ is the lowest root in $\Phi_-$.
\end{enumerate}
Since $u_n\notin N_-$, the condition $$d_G(a_tu_na_t^{-1},e)\leq e^{-\gamma t}$$ cannot hold for arbitrary large $t\in\mathbb R_+$. This implies that the elements in the collection $\left\{u_n\right\}_{n\in\mathbb N}$ are pairwise different. For each $u_n$, we define $$\tau_{n}=\inf\left\{t\geq0: d_G(a_{t}u_na_{-t},e)\leq e^{-\gamma t}\right\}.$$ We claim that $\tau_n\to\infty.$ Suppose by contradiction that $\tau_n\to\tau_0$ for some $\tau_0\geq0$, after passing to a subsequence. Then, we have $$d_G(a_{\tau_n}u_na_{-\tau_n},e)\leq e^{-\gamma\tau_n}$$ and for sufficiently large $n\in\mathbb N$ $$d_G(a_{\tau_0}u_n a_{-\tau_0},e)\leq 10 e^{-\gamma\tau_0}.$$ This implies that $\left\{u_n\right\}_{n\in\mathbb N}$ is a bounded subset in the discrete subset $\Stab(p)$, which contradicts the condition that $\left\{u_n\right\}_{n\in\mathbb N}$ are pairwise different. Now by replacing $t_n$ with $\tau_{n}$, we can assume that $$t_n=\inf\left\{t\geq0: d_G(a_{t}u_na_{-t},e)\leq e^{-\gamma t}\right\}\to\infty$$ and $$d_G(a_{t_n}u_na_{-t_n},e)=e^{-\gamma t_n}.$$ For each $u_n$, we write $$u_n=\exp(y_{n,-}+y_{n,0}+y_{n,+}),$$ where $y_{n,-}\in\mathfrak n_-$, $y_{n,0}\in\mathfrak g_0$ and $y_{n,+}\in\mathfrak n_+$. Then, $$a_{t_n}u_na_{t_n}^{-1}=\exp(x_{n,-}+x_{n,0}+x_{n,+})$$ where $x_{n,-}=\Ad(a_{t_n})y_{n,-}\in\mathfrak n_-,x_{n,0}=\Ad(a_{t_n})y_{n,0}\in\mathfrak g_0,x_{n,+}=\Ad(a_{t_n})y_{n,+}\in\mathfrak n_+$. It follows that that
\begin{align*}
&\|x_{n,-}\|_{\mathfrak g},\|x_{n,0}\|_{\mathfrak g},\|x_{n,+}\|_{\mathfrak g}\\
\leq&\kappa\|x_{n,-}+x_{n,0}+x_{n,+}\|_{\mathfrak g}\\
\leq&\kappa^2 d_G(a_{t_n}u_na_{t_n}^{-1},e)\\
=&\kappa^2 e^{-\gamma t_n}.
\end{align*}
Finally, since $\left\{a_t\right\}_{t\in\mathbb R}$ expands $\mathfrak n_+$, centralizes $\mathfrak g_0$, and contracts $\mathfrak n_-$, and
$$t_n=\inf\left\{t\geq0: d_G(a_{t}u_na_{-t},e)\leq e^{-\gamma t}\right\}\to +\infty$$ we have that
for $A\geq 1$ with $e^{A\gamma}\geq 3\kappa^{2}$, it holds
$$\|\Ad(a_{t_n-A})(y_{n,0}+y_{n,+})\|_{\mathfrak g}\leq \|x_{n,0}+x_{n,+}\|_{\mathfrak g}\leq 2\kappa^{2}e^{-\gamma t_{n}}$$
and
$$\|\Ad(a_{t_n-A})(y_{n,-}+y_{n,0}+y_{n,+})\|_{\mathfrak g}\geq e^{-\gamma(t_{n}-A)}\geq e^{A\gamma}e^{-\gamma t_{n}},$$
whence
$$e^{(\alpha_{0}+\beta_{0})(A)}\|\Ad(a_{t_n})(y_{n,-})\|_{\mathfrak g}\geq \|\Ad(a_{t_n-A})(y_{n,-})\|_{\mathfrak g}\geq \kappa^{2}e^{-\gamma t_{n}}.$$
Thus,
$$\|x_{n,-}\|_{\mathfrak g}\geq\frac1{\kappa'} e^{-\gamma t_n},$$
with $\kappa'$ depending on $\kappa$, $\gamma$, and $(\alpha_{0}+\beta_{0})(X_{0})$. We complete the proof of the lemma by setting $v_n=a_{t_n}u_na_{t_n}^{-1}$.
\end{proof}

In view of the statement of Lemma \ref{l62}, we give the following definition.

\begin{definition}\label{d61}
Let $p\in G/\Gamma$, $\gamma,t>0$ and $v\in\Stab(a_t\cdot p)$. Let $\kappa '$ be the constant in part $(1)$ of Lemma \ref{l62}. We say that $t$ and $v$ satisfy the $\gamma$-condition if\vspace{2mm}
\begin{enumerate}
\item $v$ is primitive in $\textup{Stab}(p)\cap\Gamma$ and is conjugate to an element of $N_\nu$, where $\nu$ is the lowest root in $\Phi_-$;\vspace{2mm}
\item assuming $v=\exp(v_-+v_0+v_+)$, where $v_-\in\mathfrak n_-,v_0\in\mathfrak g_0$ and $v_+\in\mathfrak n_+$, we have $$\frac1{\kappa'} e^{-\gamma t_n}\leq\|v_-\|_{\mathfrak g}\leq\kappa' e^{-\gamma t},\;\|v_0\|_{\mathfrak g}\leq\kappa' e^{-\gamma t},\;\|v_+\|_{\mathfrak g}\leq\kappa' e^{-\gamma t}.$$
\end{enumerate}
\end{definition}

We can now reformulate Lemma~\ref{l62} as follows.
\begin{proposition}\label{p61}
Let $p\in G/\Gamma$ and $\gamma>0$. Suppose that $p$ is not Diophantine of type $\gamma$ and that $\Stab(p)\cap N_-=\left\{e\right\}$. Then, there exist a sequence $t_n\to\infty$ and a sequence $v_n\in\Stab(a_{t_n}\cdot p)\setminus\left\{e\right\}$ such that\vspace{2mm}
\begin{enumerate}
\item $t_n$ and $v_n$ satisfy the $\gamma$-condition;\vspace{2mm}
\item $a_{-t_n}\cdot v_n\cdot a_{t_n}$ and $a_{-t_m}\cdot v_m\cdot a_{t_m}$ are two different elements in $\Stab(p)$ for any $m\neq n\in\mathbb N$.
\end{enumerate}
\end{proposition}

\subsection{Diophantine approximation}

In this subsection, we aim to characterize non-Diophantine points on $N_{+}\Gamma/\Gamma$ in terms of proximity to rational points and/or "rational objects". This is made precise in the next proposition, which constitutes the main result of this subsection. Hereafter, we will write 
\begin{align*}
B(r_1,r_2,r_3):=&\left\{n\in N_+: n=\exp(x_1+x_2+x_3),x_1\in\mathfrak g_{\alpha_0},x_2\in\mathfrak g_{\beta_0},x_3\in\mathfrak g_{\alpha_0+\beta_0},\right.\\ &\quad\left.\|x_1\|_{\mathfrak g}\leq r_1,\|x_2\|_{\mathfrak g}\leq r_2,\|x_3\|_{\mathfrak g}\leq r_3\right\}.
\end{align*}

\begin{proposition}\label{p62}
Let $p\in G/\Gamma$ and let $\gamma>0$. Suppose that $p$ is not Diophantine of type $\gamma$ and that $\Stab(p)\cap N_-=\left\{e\right\}$. Then, there exist constants $C,\kappa''\geq 1$ depending only on $G$, $(\alpha_{0}+\beta_{0})(X_{0})$, and $\gamma$, a sequence $t_n\to\infty$ of real numbers, and a sequence $q_n\in G/\Gamma$ of rational points such that one of the following Diophantine conditions holds for all pairs $(t_n,q_n)$:\vspace{2mm}
\begin{enumerate}
\item $p\in B\left(0,Ce^{\beta_0(a_{-t_n})},Ce^{(\alpha_0+\beta_0)(a_{-t_n})}\right)w_3^{-1}\cdot q_n$ and $$\frac1{\kappa''}e^{\beta_0(a_{t_n})-\gamma t_n}\leq d_{q_n}\leq\kappa'' e^{\beta_0(a_{t_n})-\gamma t_n};$$\vspace{2mm} 
\item $p\in B\left(Ce^{\alpha_0(a_{-t_n})},0,Ce^{(\alpha_0+\beta_0)(a_{-t_n})}\right)w_2^{-1}\cdot q_n$ and $$\frac1{\kappa''}e^{\alpha_0(a_{t_n})-\gamma t_n}\leq d_{q_n}\leq\kappa'' e^{\alpha_0(a_{t_n})-\gamma t_n};$$\vspace{2mm}
\item $p\in B\left(Ce^{\alpha_0(a_{-t_n})},Ce^{\beta_0(a_{-t_n})},Ce^{(\alpha_0+\beta_0)(a_{-t_n})}\right)\cdot q_n$ and $$\frac1{3\kappa''}e^{(\alpha_0+\beta_0)(a_{t_n})-\gamma t_n}\leq d_{q_n}\leq\kappa'' e^{(\alpha_0+\beta_0)(a_{t_n})-\gamma t_n};$$\vspace{2mm}
\item $p\in B_n\cdot q_n$, with 
\begin{align*}
B_n=&\bigg\{x\in N_+: x\in B\left(0,Ce^{\beta_0(a_{-t_n})},Ce^{(\alpha_0+\beta_0)(a_{-t_n})}\right)\cdot y, y\in N_{\alpha_0},\\
&\left.\frac1{3\kappa''}e^{-\gamma t_n}e^{\beta_0(a_{t_n})}\leq d_G(y,e)\cdot d_{q_n}\leq\kappa'' e^{-\gamma t_n}e^{\beta_0(a_{t_n})}\right\}
\end{align*}
and
$$d_{q_n}\leq\kappa''e^{(\alpha_0+\beta_0)(a_{t_n})-\gamma t_n};$$\vspace{2mm}
\item $p\in B_n\cdot q_n$, with   
\begin{align*}
B_n=&\bigg\{x\in N_+: x\in B\left(Ce^{\alpha_0(a_{-t_n})},0,Ce^{(\alpha_0+\beta_0)(a_{-t_n})}\right)\cdot y,y\in N_{\beta_0},\\
&\left.\frac1{3\kappa''}e^{-\gamma t_n}e^{\alpha_0(a_{t_n})}\leq d_G(y,e)\cdot d_{q_n}\leq\kappa'' e^{-\gamma t_n}e^{\alpha_0(a_{t_n})}\right\}.
\end{align*}
and $$d_{q_n}\leq\kappa''e^{(\alpha_0+\beta_0)(a_{t_n})-\gamma t_n}.$$ 
\end{enumerate}
\end{proposition}

Before proceeding to prove Proposition \ref{p62}, we recall that by Proposition \ref{prop:Bruhat} $$G=\bigsqcup_{i\in I} DN_-w_iN_+,$$ where $\left\{w_i\right\}_{i\in I}$ is the set of representatives of the Weyl group $W$ in $G$ given by $$w_1=\left(\begin{array}{ccc} 1 & 0 & 0\\ 0& 1 & 0\\ 0& 0& 1\end{array}\right),\quad w_2=\left(\begin{array}{ccc} 1 & 0 & 0\\ 0& 0 & -1\\ 0& 1& 0\end{array}\right),\quad w_3=\left(\begin{array}{ccc} 0 & -1 & 0\\ 1& 0 & 0\\ 0& 0& 1\end{array}\right),$$

$$w_4=\left(\begin{array}{ccc} 0 & 0 & 1\\ 1& 0 & 0\\ 0& 1& 0\end{array}\right),\quad w_5=\left(\begin{array}{ccc} 0 & 1 & 0\\ 0& 0 & 1\\ 1& 0& 0\end{array}\right),\quad w_6=\left(\begin{array}{ccc} 0 & 0 & 1\\ 0& -1 & 0\\ 1& 0& 0\end{array}\right).$$

For our purpose, we will further simplify $w_iN_+$ for each $w_i$ ($1\leq i\leq 6$) as follows: $w_1N_+=N_+$ and
\begin{align*}
w_2N_+&=w_2\left(\begin{array}{ccc}1 & * & *\\ 0 & 1 & *\\ 0 & 0 &1\end{array}\right)=w_2\left(\begin{array}{ccc}1 & 0 & 0\\ 0 & 1 & *\\ 0 & 0 &1\end{array}\right)\left(\begin{array}{ccc}1 & * & *\\ 0 & 1 & 0\\ 0 & 0 &1\end{array}\right)\\
&=\left(\begin{array}{ccc}1 & 0 & 0\\ 0 & 1 & 0\\ 0 & * &1\end{array}\right)\cdot w_2\left(\begin{array}{ccc}1 & * & *\\ 0 & 1 & 0\\ 0 & 0 &1\end{array}\right)
\end{align*}
\begin{align*}
w_3N_+&=w_3\left(\begin{array}{ccc}1 & * & *\\ 0 & 1 & *\\ 0 & 0 &1\end{array}\right)=w_3\left(\begin{array}{ccc}1 & * & 0\\ 0 & 1 & 0\\ 0 & 0 &1\end{array}\right)\left(\begin{array}{ccc}1 & 0 & *\\ 0 & 1 & *\\ 0 & 0 &1\end{array}\right)\\
&=\left(\begin{array}{ccc}1 & 0 & 0\\ * & 1 & 0\\ 0 & 0 &1\end{array}\right)\cdot w_3\left(\begin{array}{ccc}1 & 0 & *\\ 0 & 1 & *\\ 0 & 0 &1\end{array}\right)
\end{align*}
\begin{align*}
w_4N_+&=w_4\left(\begin{array}{ccc}1 & * & *\\ 0 & 1 & *\\ 0 & 0 &1\end{array}\right)=w_4\left(\begin{array}{ccc}1 & 0 & *\\ 0 & 1 & *\\ 0 & 0 &1\end{array}\right)\left(\begin{array}{ccc}1 & * & 0\\ 0 & 1 & 0\\ 0 & 0 &1\end{array}\right)\\
&=\left(\begin{array}{ccc}1 & 0 & 0\\ * & 1 & 0\\ * & 0 &1\end{array}\right)\cdot w_4\left(\begin{array}{ccc}1 & * & 0\\ 0 & 1 & 0\\ 0 & 0 &1\end{array}\right)
\end{align*}
\begin{align*}
w_5N_+&=w_5\left(\begin{array}{ccc}1 & * & *\\ 0 & 1 & *\\ 0 & 0 &1\end{array}\right)=w_5\left(\begin{array}{ccc}1 & * & *\\ 0 & 1 & 0\\ 0 & 0 &1\end{array}\right)\left(\begin{array}{ccc}1 & 0 & 0\\ 0 & 1 & *\\ 0 & 0 &1\end{array}\right)\\
&=\left(\begin{array}{ccc}1 & 0 & 0\\ 0 & 1 & 0\\ * & * &1\end{array}\right)\cdot w_5\left(\begin{array}{ccc}1 & 0 & 0\\ 0 & 1 & *\\ 0 & 0 &1\end{array}\right)
\end{align*}
\begin{align*}
w_6N_+&=w_6\left(\begin{array}{ccc}1 & * & *\\ 0 & 1 & *\\ 0 & 0 &1\end{array}\right)=\left(\begin{array}{ccc}1 & 0 & 0\\ * & 1 & 0\\ * & * &1\end{array}\right)\cdot w_6.
\end{align*}
For simple roots $\alpha_0,\beta_0\in\Delta$, we will denote the subgroups in $G$ corresponding to $\mathfrak g_{\alpha_0}$ and $\mathfrak g_{\beta_0}$ by $$N_{\alpha_0}=\left(\begin{array}{ccc}1 & * & 0\\ 0 & 1 & 0\\ 0 & 0 &1\end{array}\right),\;\mathfrak g_{\alpha_0}=\left(\begin{array}{ccc}0 & * & 0\\ 0 & 0 & 0\\ 0 & 0 &0\end{array}\right)$$
$$N_{\beta_0}=\left(\begin{array}{ccc}1 & 0 & 0\\ 0 & 1 & *\\ 0 & 0 &1\end{array}\right),\;\mathfrak g_{\beta_0}=\left(\begin{array}{ccc}0 & 0 & 0\\ 0 & 0 & *\\ 0 & 0 &0\end{array}\right)$$ and for $\alpha_0+\beta_0\in\Delta$ by $$N_{\alpha_0+\beta_0}=\left(\begin{array}{ccc}1 & 0 & *\\ 0 & 1 & 0\\ 0 & 0 &1\end{array}\right),\;\mathfrak g_{\alpha_0+\beta_0}=\left(\begin{array}{ccc}0 & 0 & *\\ 0 & 0 & 0\\ 0 & 0 &0\end{array}\right).$$ Similarly, we write $$N_{-\alpha_0}=\left(\begin{array}{ccc}1 & 0 & 0\\ * & 1 & 0\\ 0 & 0 &1\end{array}\right),\;\mathfrak g_{-\alpha_0}=\left(\begin{array}{ccc}0 & 0 & 0\\ * & 0 & 0\\ 0 & 0 &0\end{array}\right)$$
$$N_{-\beta_0}=\left(\begin{array}{ccc}1 & 0 & 0\\ 0 & 1 & 0\\ 0 & * &1\end{array}\right),\;\mathfrak g_{-\beta_0}=\left(\begin{array}{ccc}0 & 0 & 0\\ 0 & 0 & 0\\ 0 & * &0\end{array}\right)$$
$$N_{-\alpha_0-\beta_0}=\left(\begin{array}{ccc}1 & 0 & 0\\ 0 & 1 & 0\\ * & 0 &1\end{array}\right),\;\mathfrak g_{-\alpha_0-\beta_0}=\left(\begin{array}{ccc}0 & 0 & 0\\ 0 & 0 & 0\\ * & 0 &0\end{array}\right).$$ Note that $\nu=-\alpha_0-\beta_0$ is the lowest root in $\Phi_-$. From these computations, we deduce that
\begin{align}
G=DN_-N_+&\sqcup DN_-w_2N_{\alpha_0}N_{\alpha_0+\beta_0}\sqcup DN_-w_3N_{\beta_0}N_{\alpha_0+\beta_0}\label{eq:BD}\\
&\sqcup DN_-w_4N_{\alpha_0}\sqcup DN_-w_5N_{\beta_0}\sqcup DN_-w_6.\nonumber
\end{align}
We set 
\begin{align*}
\mathscr F&:=N_+\sqcup w_2N_{\alpha_0}N_{\alpha_0+\beta_0}\sqcup w_3N_{\beta_0}N_{\alpha_0+\beta_0}
\sqcup w_4N_{\alpha_0}\sqcup w_5N_{\beta_0}\sqcup w_6\\
&:=\mathscr F_1\sqcup\mathscr F_2\sqcup\mathscr F_3\sqcup\mathscr F_4\sqcup\mathscr F_5\sqcup\mathscr F_6,
\end{align*}
whence $G=DN_-\cdot\mathscr F$. Note that $\mathscr F$ is the collection of representatives for the space of right-cosets of $DN_-$.

The proof of Proposition \ref{p62} relies on the existence of sequences $(t_{n},v_{n})$ satisfying the $\gamma$-condition (see Proposition \ref{p61}). Part $(1)$ in the definition of $\gamma$-condition, ensures that each element $v_{n}$ is conjugated to a generator of the group $N_{\nu}\cap\Gamma$. By using the decomposition $G=DN_{-}\mathscr{F}$, we aim to introduce a classification of vectors $v_{n}$ depending on "how" they are conjugated to a generator of $N_{\nu}\cap\Gamma$. This "conjugacy type" will then translate into different Diophantine properties for the point $p$.

\begin{lemma}\label{l64}
Let $x\in G\setminus\left\{e\right\}$ and suppose that $x$ is conjugate to an element in $N_\nu$ where $\nu$ is the lowest root in $\Phi_-$. Then, there exists a unique element $h$ in $\mathscr F$ such that $$hxh^{-1}\in N_{\nu}.$$
\end{lemma}
\begin{proof}
Since $x$ is conjugate to an element in $N_{\nu}$, there exists $g\in G$ such that $$gxg^{-1}\in N_{\nu}.$$ By (\ref{eq:BD}), there exist $a\in D$, $n\in N_-$, and $h\in\mathscr F$ such that $$g=an\cdot h.$$ Since $an$ normalizes $N_\nu$, we conclude that $$hxh^{-1}\in(an)^{-1}N_\nu(an)=N_\nu.$$ Now, suppose that there exists another element $\tilde h\in\mathscr F$ such that $\tilde hx\tilde h^{-1}\in N_\nu$. Since $\dim N_{\nu}=1$, we have $$\overline{h\langle x\rangle h^{-1}}^{Zar}=N_{\nu}\quad\textup{and}\quad\overline{\tilde h\langle x\rangle\tilde h^{-1}}^{Zar}=N_{\nu}$$ which implies that $h\tilde h^{-1}$ is in the normalizer of $N_\nu$. By Lemma~\ref{l31}, we have $h\tilde h^{-1}\in DN_-$. As $\mathscr F$ is a set of representatives of right-cosets of $DN_-$, we conclude that $h=\tilde h$, which proves the uniqueness.
\end{proof}

In view of Lemma \ref{l64}, we give the following definition.

\begin{definition}\label{d62}
Suppose that $x\in G\setminus\left\{e\right\}$ is conjugate to an element in $N_\nu$, and let $h\in\mathscr F$ such that $hxh^{-1}\in N_{\nu}$. If $h\in\mathscr F_i$ $(1\leq i\leq 6)$, then we say that $x$ is of type $w_i$, where $w_i$ is the representative of Weyl group corresponding to $\mathscr F_i$.
\end{definition}

Now, we proceed to analyze the different Diophantine properties of points of type $w_{i}$. We start with the following simple lemma about points of type $w_{4},w_{5},w_{6}$.

\begin{lemma}\label{l65}
Let $p\in G/\Gamma$, $\gamma,t>0$ and $v\in\Stab(a_t\cdot p)$. Suppose that $t$ and $v$ satisfy the $\gamma$-condition and $v$ is of type $w_i$ with $i=4,5,6$. Then, $v\in N_+$.
\end{lemma}
\begin{proof}
First, we consider the case $i=4$. By definition, we know that $v$ is conjugate to an element in $N_{\nu}$ and there exists $h\in w_4N_{\alpha_0}$ such that $hvh^{-1}\in N_{\nu}$. Let $h=w_4 n$ where $n\in N_{\alpha_0}$. Then we have
\begin{align*}
hvh^{-1}\in N_{\nu}\implies v\in n^{-1}N_{\beta_0}n\subset N_+.
\end{align*}
The proof for the case $i=5$ is similar. For $i=6$ we have $$w_6 vw_6^{-1}\in N_{\nu}.$$ Since $w_6^{-1}N_\nu w_6\subset N_+$, we conclude that $v\in N_+$.
\end{proof}

In the next lemma we will see that types $w_{4},w_{5},w_{6}$ play no role in the proof of Proposition \ref{p62}.

\begin{lemma}\label{l66}
Let $p\in G/\Gamma$, and let $t_n\to\infty$ and $v_n\in\Stab(a_{t_n}\cdot p)\setminus\left\{e\right\}$ such that $t_n$ and $v_n$ satisfy the $\gamma$-condition for all $n$. Then, there are only finitely many indices $n$ such that $v_{n}$ is of type $w_{i}$ with $i=4,5,6$.
\end{lemma}

\begin{proof}
Suppose that $t_n$ and $v_n$ satisfy the $\gamma$-condition and that (without loss of generality) all vectors $v_n\in\Stab(a_{t_n}\cdot p)\setminus\left\{e\right\}$ are of type $w_i$ for some $i=4,5,6$. By Lemma~\ref{l65}, it follows that $v_n\in N_+$ for all $n$. Let $x_n=a_{-t_n}v_na_{t_n}\in\Stab(p)$. Then, $x_n$ is also in $N_+$ and $$d_G(x_n,e)\geq\eta(p).$$ This implies that $$d_G(v_n,e)=d_G(a_{t_n}x_na_{-t_n},e)\geq d_G(x_n,e)\geq\eta(p)>0,$$ which contradicts the fact that $$d_G(v_n,e)\leq\kappa' e^{-\gamma t_n}\to0.$$
\end{proof}

The next string of three Lemmas (\ref{l67}, \ref{l68}, \ref{l69}) deals with sequences $v_{n}$ of types $w_{1},w_{2},w_{3}$. The combination of these three results will give Proposition \ref{p62}. In the proofs of these lemmas, we will frequently make use of the following simple fact.

\begin{lemma}\label{l63}
There exists a constant $C>0$ depending only on $G$ such that if $\alpha,\beta\in\Phi$ and $[\mathfrak g_\alpha,\mathfrak g_\beta]\neq 0$, we have that for any $u\in\mathfrak g_\alpha$ and $v\in\mathfrak g_\beta$ $$\|[u,v]\|_{\mathfrak g}\geq C\|u\|_{\mathfrak g}\cdot\|v\|_{\mathfrak g}.$$
\end{lemma}

\begin{proof}
Without loss of generality, we may assume that the vectors $u$ and $v$ have length $1$. Then, it suffices to show that the quantity $\|[u,v]\|_{\mathfrak g}$ is bounded below on the unit sphere. Since the vector spaces $\mathfrak g_\alpha$ and $\mathfrak g_\beta$ are one dimensional and we are assuming $[\mathfrak g_\alpha,\mathfrak g_\beta]\neq 0$, it must be $[u,v]\neq 0$ for all $u\in\mathfrak g_\alpha$ and $v\in\mathfrak g_\beta$. The result follows from the compactness of the sphere. 
\end{proof}

\begin{lemma}\label{l67}
Let $p\in G/\Gamma$, $\gamma,t>0$ and $v\in\Stab(a_t\cdot p)$. Suppose that $t$ and $v$ satisfy the $\gamma$-condition and that $v$ is of type $w_3$. Then, there exist a constant $C>0$ depending only on $G$ and $\bar n\in B\left(0,Ce^{\beta_0(a_{-t})},Ce^{(\alpha_0+\beta_0)(a_{-t})}\right)$ such that $(w_3\bar n)\cdot p$ is a rational point in $G/\Gamma$ with $$\frac1{\kappa'}e^{\beta_0(a_t)-\gamma t}\leq d_{w_3\bar n\cdot p}\leq\kappa' e^{\beta_0(a_t)-\gamma t}.$$
\end{lemma}
\begin{proof}
By Definition~\ref{d61}, $v$ is conjugate to a generator in $N_\nu$ where $\nu$ is the lowest root in $\Phi_-$ and there exists $h\in w_3N_{\beta_0}N_{\alpha_0+\beta_0}$ such that $$hvh^{-1}\in N_\nu.$$ Let $h=w_3 n$ where $n\in N_{\beta_0}N_{\alpha_0+\beta_0}$. Then, we have
\begin{align*}
hvh^{-1}\in N_\nu&\implies \Ad(n)v=nvn^{-1}\in w_3^{-1} N_\nu w_3=N_{-\beta_0}\\
&\implies\Ad(n)\log(v)\in\mathfrak g_{-\beta_0}.
\end{align*}
Now, let $x=\Ad(n)\log(v)\in\mathfrak g_{-\beta_0}$ and $\Ad(n^{-1})=e^{\ad(u)}$, where $u=u_{\beta_0}+u_{\alpha_0+\beta_0}\in\mathfrak g_{\beta_0}\oplus\mathfrak g_{\alpha_0+\beta_0}$. It is an easy exercise to check that for such a vector $u$ the endomorphism $\ad(u)$ has order $3$. It follows that
\begin{align}
\label{eq:1}
x+\ad(u)(x)+\frac1{2!}\ad^2(u)(x)=\log(v).
\end{align}
Let $\log(v)=v_0+\sum_{\alpha\in\Phi}v_\alpha$, where $v_0\in\mathfrak g_0$ and $v_\alpha\in\mathfrak g_\alpha$. Then, we have
\begin{align}\label{eq3}
x+[u_{\beta_0}+u_{\alpha_0+\beta_0},x]+\frac1{2}[u_{\beta_0}+u_{\alpha_0+\beta_0},[u_{\beta_0}+u_{\alpha_0+\beta_0},x]]=v_0+\sum_{\alpha\in\Phi}v_\alpha.
\end{align}
By comparing both sides of Equation~\eqref{eq3}, we find 
\begin{equation}
\label{eq:2}
\begin{cases}v_{-\alpha_0-\beta_0}=0\\v_{-\alpha_0}=0\\v_{-\beta_0}=x \\ v_0=[u_{\beta_0},x]\\ v_{\alpha_0}=[u_{\alpha_0+\beta_0},v]\\v_{\beta_0}=\frac12[u_{\beta_0}[u_{\beta_0},x]]\\v_{\alpha_0+\beta_0}=\frac12[u_{\beta_0}[u_{\alpha_0+\beta_0},x]]+\frac12[u_{\alpha_0+\beta_0}[u_{\beta_0},x]] \end{cases}
\end{equation}
Since $t$ and $v$ satisfy the $\gamma$-condition, we deduce by the previous equalities that $$\frac{1}{\kappa'}e^{-\gamma t}\leq\|x\|_{\mathfrak g}\leq\kappa' e^{-\gamma t}$$
and
$$\|[u_{\beta_0},x]\|_{\mathfrak g}\leq\kappa' e^{-\gamma t},\;\|[u_{\alpha_0+\beta_0},x]\|_{\mathfrak g}\leq\kappa' e^{-\gamma t}.$$ By Lemma~\ref{l63}, this implies that
\begin{equation}
\label{eq:nbar}
\|u_{\beta_0}\|_{\mathfrak g}, \|u_{\alpha_0+\beta_0}\|_{\mathfrak g}\leq C
\end{equation}
for some constant $C>0$ depending only on $G$ and $\kappa'$, and hence $d_G(n,e)\leq C.$ 

Since $\Stab(w_3n a_t\cdot p)\cap N_{\nu}\neq\left\{e\right\}$, by Definition~\ref{d31}, $w_3n a_t\cdot p$ is a rational point in $G/\Gamma$. Now, given that $w_{3}$ normalizes $D$, by Lemma \ref{p31}, $w_3(a_{-t}na_t)\cdot p$ is also a rational point (multiplying by a diagonal element preserves rationality). Set $\bar n=a_{-t}na_t$. Then, we have that $w_3\bar n\cdot p$ is a rational point in $G/\Gamma$ and, by (\ref{eq:nbar}), $\bar n\in B\left(0,Ce^{\beta_0(a_{-t})},Ce^{(\alpha_0+\beta_0)(a_{-t})}\right)$.

Let us compute the denominator of the rational point $w_3\bar n\cdot p$. We know that $v$ is primitive, therefore $\Ad(h)v$ is a generator of the discrete subgroup $\textup{Stab}(hp)\cap N_{\nu}$. It follows that
$$d_{w_3na_t\cdot p}=\|\Ad(w_{3}n)\log(v)\|_{\mathfrak g_{\nu}}=\|\Ad(n)\log(v)\|_{\mathfrak g_{\beta_{0}}}=\|x\|_{\mathfrak g},$$ 
where we used (\ref{eq:2}).
Since
$$\frac1{\kappa'}e^{-\gamma t}\leq\|x\|_{\mathfrak g}\leq\kappa' e^{-\gamma t}$$ and the denominator of $w_3\bar n\cdot p$ is equal to $d_{w_3na_t\cdot p}\cdot e^{\nu(w_3a_{-t}w_3^{-1})}$, we deduce $$\frac1{\kappa'}e^{\beta_0(a_t)-\gamma t}\leq d_{w_3\bar n\cdot p}\leq\kappa' e^{\beta_0(a_t)-\gamma t}.$$ This completes the proof of the lemma.
\end{proof}

\begin{lemma}\label{l68}
Let $p\in G/\Gamma$, $\gamma,t>0$ and $v\in\Stab(a_t\cdot p)$. Suppose that $t$ and $v$ satisfy $\gamma$-condition and $v$ is of type $w_2$. Then there exist a constant $C>0$ depending only on $G$ and $\bar n\in B(Ce^{\alpha_0(a_{-t})},0,Ce^{(\alpha_0+\beta_0)(a_{-t})})$ such that $(w_2\bar n)\cdot p$ is a rational point in $G/\Gamma$ with $$\frac1{\kappa'}e^{\alpha_0(a_t)-\gamma t}\leq d_{w_2\bar n\cdot p}\leq\kappa' e^{\alpha_0(a_t)-\gamma t}.$$
\end{lemma}
\begin{proof}
The proof is analogous to that of Lemma~\ref{l67}. 
\end{proof}

The most complex case in this subsection is the following lemma.
\begin{lemma}\label{l69}
Let $p\in G/\Gamma$, $\gamma,t>0$ and $v\in\Stab(a_t\cdot p)$. Suppose that $t$ and $v$ satisfy the  $\gamma$-condition and that $v$ is of type $w_1$. Then, there are constants $C,\kappa''\geq 1$, only depending on $G$, $\gamma$, and $(\alpha_{0}+\beta_{0})(X_{0})$, such that one of the three following cases occurs\vspace{2mm}
\begin{enumerate}
\item there exists $$\bar n\in B\left(Ce^{\alpha_0(a_{-t})},Ce^{\beta_0(a_{-t})},Ce^{(\alpha_0+\beta_0)(a_{-t})}\right)$$ such that $\bar n\cdot p$ is a rational point in $G/\Gamma$ with $$\frac1{\kappa''}e^{(\alpha_0+\beta_0)(a_t)-\gamma t}\leq d_{\bar n\cdot p}\leq\kappa'' e^{(\alpha_0+\beta_0)(a_t)-\gamma t}.$$\vspace{2mm}

\item there exists $\bar n\in N_+$ such that $\bar n\cdot p$ is a rational point in $G/\Gamma$, and 
\begin{align*}
\bar n\in & \bigg\{x\in N_+: x\in y\cdot B\left(0,Ce^{\beta_0(a_{-t})},Ce^{(\alpha_0+\beta_0)(a_{-t})}\right),y\in N_{\alpha_0},\\
&\left.\frac1{\kappa''}e^{-\gamma t}e^{\beta_0(a_t)}\leq d_G(y,e)\cdot d_{\bar n\cdot p}\leq\kappa'' e^{-\gamma t}e^{\beta_0(a_t)}\right\}
\end{align*}
with $$d_{\bar n\cdot p}\leq\kappa''e^{(\alpha_0+\beta_0)(a_t)-\gamma t}.$$\vspace{2mm}

\item there exists $\bar n\in N_+$ such that $\bar n\cdot p$ is a rational point in $G/\Gamma$, and 
\begin{align*}
\bar n\in\bigg\{x\in &N_+: x\in y\cdot B(Ce^{\alpha_0(a_{-t})},0,Ce^{(\alpha_0+\beta_0)(a_{-t})}), y\in N_{\beta_0},\\
&\left.\frac1{\kappa''}e^{-\gamma t}e^{\alpha_0(a_t)}\leq d_G(y,e)\cdot d_{\bar n\cdot p}\leq\kappa'' e^{-\gamma t}e^{\alpha_0(a_t)}\right\}
\end{align*}
with $$d_{\bar n\cdot p}\leq\kappa'' e^{(\alpha_0+\beta_0)(a_t)-\gamma t}.$$
\end{enumerate}
\end{lemma}
\begin{proof}
By Definition~\ref{d61}, $v$ is conjugate to an element in $N_\nu$ where $\nu$ is the lowest root in $\Phi_-$, and there exists $h\in N_+$ such that $$hvh^{-1}\in N_{\nu}.$$ Let $h^{-1}=\exp(u)$ for some $u\in\mathfrak n_+$ and $x=h\log(v)h^{-1}\in \mathfrak g_\nu$. Then, since $\ad(u)$ has order $5$, we have $\Ad(h^{-1})x=\log(v)$, whence 
\begin{align}\label{eq5}
x+\ad(u)(x)+\frac1{2!}\ad^2(u)(x)+\frac1{3!}\ad^3(u)(x)+\frac1{4!}\ad^4(u)(x)=\log(v).
\end{align}
 We choose a coordinate system in $\mathfrak g$ as follows: $$e_{-\alpha_0-\beta_0}=\left(\begin{array}{ccc} 0 & 0 & 0\\ 0& 0 & 0\\1 & 0 & 0\end{array}\right), e_{-\alpha_0}=\left(\begin{array}{ccc} 0 & 0 & 0\\ 1& 0 & 0\\0 & 0 & 0\end{array}\right), e_{-\beta_0}=\left(\begin{array}{ccc} 0 & 0 & 0\\ 0& 0 & 0\\0 & 1 & 0\end{array}\right)$$ $$e_{\alpha_0+\beta_0}=\left(\begin{array}{ccc} 0 & 0 & 1\\ 0& 0 & 0\\0 & 0 & 0\end{array}\right), e_{\alpha_0}=\left(\begin{array}{ccc} 0 & 1 & 0\\ 0& 0 & 0\\0 & 0 & 0\end{array}\right), e_{\beta_0}=\left(\begin{array}{ccc} 0 & 0 & 0\\ 0& 0 & 1\\0 & 0 & 0\end{array}\right)$$ $$e_1=\left(\begin{array}{ccc} 1 & 0 & 0\\ 0& 0 & 0\\0 & 0 & -1\end{array}\right), e_2=\left(\begin{array}{ccc} 1 & 0 & 0\\ 0& -2 & 0\\0 & 0 & 1\end{array}\right).$$ For $x\in\mathfrak g_\nu$ and $u\in\mathfrak n_+$, we write 
\begin{align*}
u=t_{\alpha_0}e_{\alpha_0}+t_{\beta_0}e_{\beta_0}+t_{\alpha_0+\beta_0}e_{\alpha_0+\beta_0}
\end{align*}
and $$x=x_\nu\cdot e_{\nu}$$ where $\nu=-\alpha_0-\beta_0$.
Now let $$\log(v)=v_0+\sum_{\alpha\in\Phi} v_\alpha\in\mathfrak g_0\oplus\bigoplus_{\alpha\in\Phi}\mathfrak g_{\alpha}.$$ From Equation~\eqref{eq5}, we deduce that $$\begin{cases}v_{-\alpha_0-\beta_0}=x_{-\alpha_0-\beta_0}e_{-\alpha_0-\beta_0}\\v_{-\beta_0}=-t_{\alpha_0}x_{-\alpha_0-\beta_0}e_{-\beta_0}\\ v_{-\alpha_0}=t_{\beta_0}x_{-\alpha_0-\beta_0}e_{-\alpha_0}\\v_0=t_{\alpha_0+\beta_0}x_{-\alpha_0-\beta_0}e_1+\frac12 t_{\alpha_0}t_{\beta_0}x_{-\alpha_0-\beta_0}e_2\\v_{\alpha_0}=(-t_{\alpha_0}t_{\alpha_0+\beta_0}-\frac12t_{\alpha_0}^2t_{\beta_0})x_{-\alpha_0-\beta_0}e_{\alpha_0}\\v_{\beta_0}=(-t_{\beta_0}t_{\alpha_0+\beta_0}+\frac12 t_{\beta_0}^2t_{\alpha_0})x_{-\alpha_0-\beta_0}e_{\beta_0}\\v_{\alpha_0+\beta_0}=(-t_{\alpha_0+\beta_0}^2+\frac14t_{\alpha_0}^2t_{\beta_0}^2)x_{-\alpha_0-\beta_0}e_{\alpha_0+\beta_0}\end{cases}.$$
Hence, since $t$ and $v$ satisfy the $\gamma$-condition, we have $$\frac1{\kappa'}e^{-\gamma t}\leq\|v_{-\alpha_0-\beta_0}+v_{-\beta_0}+v_{-\alpha_0}\|_{\mathfrak g}\leq\kappa' e^{-\gamma t}$$ $$\|v_0\|_{\mathfrak g}\leq\kappa' e^{-\gamma t},\;\|v_{\alpha_0}+v_{\beta_0}+v_{\alpha_0+\beta_0}\|_{\mathfrak g}\leq\kappa' e^{-\gamma t}.$$
Then there are three possibilities:\vspace{2mm}
\begin{enumerate}
\item $\frac1{3\kappa'}e^{-\gamma t}\leq\|v_{-\alpha_0-\beta_0}\|_{\mathfrak g}\leq\kappa\kappa'e^{-\gamma t}$;\vspace{2mm}
\item $\frac1{3\kappa'}e^{-\gamma t}\leq\|v_{-\beta_0}\|_{\mathfrak g}\leq\kappa\kappa' e^{-\gamma t}$;\vspace{2mm}
\item $\frac1{3\kappa'}e^{-\gamma t}\leq\|v_{-\alpha_0}\|_{\mathfrak g}\leq\kappa\kappa' e^{-\gamma t}$.
\end{enumerate}

Case 1: $\frac1{3\kappa'}e^{-\gamma t}\leq\|v_{-\alpha_0-\beta_0}\|_{\mathfrak g}\leq\kappa\kappa'e^{-\gamma t}.$ In this case, we have \begin{equation}
\label{eq:3}
\begin{cases}\frac1{3\kappa'}e^{-\gamma t}\leq|x_{-\alpha_0-\beta_0}|\leq\kappa\kappa'e^{-\gamma t}\\ |t_{\alpha_0}x_{-\alpha_0-\beta_0}|\leq\kappa\kappa'e^{-\gamma t}\\ |t_{\beta_0}x_{-\alpha_0-\beta_0}|\leq\kappa\kappa'e^{-\gamma t}\\ |t_{\alpha_0+\beta_0}x_{-\alpha_0-\beta_0}|\leq\kappa\kappa'e^{-\gamma t}\\ |\frac12t_{\alpha_0}t_{\beta_0}x_{-\alpha_0-\beta_0}|\leq\kappa\kappa'e^{-\gamma t}\\|(-t_{\alpha_0}t_{\alpha_0+\beta_0}-\frac12t_{\alpha_0}^2t_{\beta_0})x_{-\alpha_0-\beta_0}|\leq\kappa\kappa'e^{-\gamma t}\\ |(-t_{\beta_0}t_{\alpha_0+\beta_0}+\frac12 t_{\beta_0}^2t_{\alpha_0})x_{-\alpha_0-\beta_0}|\leq\kappa\kappa' e^{-\gamma t}\\ |(-t_{\alpha_0+\beta_0}^2+\frac14t_{\alpha_0}^2t_{\beta_0}^2)x_{-\alpha_0-\beta_0}|\leq\kappa\kappa' e^{-\gamma t}\end{cases}
\end{equation}
and we deduce that $$|t_{\alpha_0}|\leq 3\kappa\kappa'^{2}, |t_{\beta_0}|\leq3\kappa\kappa'^{2},|t_{\alpha_0+\beta_0}|\leq3\kappa\kappa'^{2}.$$
This implies that $\|u\|_{\mathfrak g}\leq C$ for some constant $C>0$ depending only on $G$, $\kappa,\kappa'$, and hence $d_G(h,e)\leq C$.

From the hypothesis, we know that $h\cdot a_tp$ is a rational point, as $\Stab(h\cdot a_t p)\cap N_{\nu}\neq\left\{e\right\}$. Then, by Proposition \ref{p31}, $a_{-t}ha_tp$ is also a rational point. Let $\bar n=a_{-t}ha_t$. Since $d_{G}(h,e)\leq C$, we have that $\bar n\cdot p$ is a rational point and $$\bar n\in B\left(Ce^{\alpha_0(a_{-t})},Ce^{\beta_0(a_{-t})},Ce^{(\alpha_0+\beta_0)(a_{-t})}\right).$$ We conclude by computing the denominator of $\bar n\cdot p$. Since $v$ is primitive, we have that
$$d_{h\cdot a_tp}=\|\Ad(\bar n)\log(v)\|_{\mathfrak g_{\nu}}=\|x\|_{\mathfrak g}=|x_{-\alpha_0-\beta_0}|.$$
Moreover, by (\ref{eq:3}), we know that
$$\frac1{3\kappa'}e^{-\gamma t}\leq|x_{-\alpha_0-\beta_0}|\leq\kappa\kappa'e^{-\gamma t}.$$ Since $d_{\bar n\cdot a_t p}=e^{(\alpha_0+\beta_0)(a_t)}d_{h\cdot a_tp}$, we conclude that $$\frac1{3\kappa'}e^{(\alpha_0+\beta_0)(a_t)-\gamma t}\leq d_{\bar n\cdot p}\leq\kappa\kappa' e^{(\alpha_0+\beta_0)(a_t)-\gamma t}.$$ This completes the proof of the first case.

Case 2: $\frac1{3\kappa'}e^{-\gamma t}\leq\|v_{-\beta_0}\|_{\mathfrak g}\leq\kappa\kappa' e^{-\gamma t}.$ In this case, we have $$\begin{cases}|x_{-\alpha_0-\beta_0}|\leq\kappa\kappa' e^{-\gamma t}\\ \frac1{3\kappa'}e^{-\gamma t}\leq|t_{\alpha_0}x_{-\alpha_0-\beta_0}|\leq\kappa\kappa'e^{-\gamma t}\\ |t_{\beta_0}x_{-\alpha_0-\beta_0}|\leq\kappa\kappa'e^{-\gamma t}\\ |t_{\alpha_0+\beta_0}x_{-\alpha_0-\beta_0}|\leq\kappa\kappa'e^{-\gamma t}\\ |\frac12t_{\alpha_0}t_{\beta_0}x_{-\alpha_0-\beta_0}|\leq\kappa\kappa'e^{-\gamma t}\\|(-t_{\alpha_0}t_{\alpha_0+\beta_0}-\frac12t_{\alpha_0}^2t_{\beta_0})x_{-\alpha_0-\beta_0}|\leq\kappa\kappa'e^{-\gamma t}\\ |(-t_{\beta_0}t_{\alpha_0+\beta_0}+\frac12 t_{\beta_0}^2t_{\alpha_0})x_{-\alpha_0-\beta_0}|\leq\kappa\kappa' e^{-\gamma t}\\ |(-t_{\alpha_0+\beta_0}^2+\frac14t_{\alpha_0}^2t_{\beta_0}^2)x_{-\alpha_0-\beta_0}|\leq\kappa\kappa' e^{-\gamma t}\end{cases}$$

We deduce that $$\begin{cases}|x_{-\alpha_0-\beta_0}|\leq\kappa\kappa' e^{-\gamma t}\\ \frac1{3\kappa'}e^{-\gamma t}\leq|t_{\alpha_0}x_{-\alpha_0-\beta_0}|\leq\kappa\kappa'e^{-\gamma t}\\
|t_{\beta_0}|\leq6\kappa\kappa'^{2}\\|t_{\alpha_0+\beta_0}+\frac12 t_{\alpha_0}t_{\beta_0}|\leq3\kappa\kappa'\end{cases}.$$

Moreover, we have
\begin{align*}
h^{-1}&=e^u=e^{t_{\alpha_0}e_{\alpha_0}+t_{\beta_0}e_{\beta_0}+t_{\alpha_0+\beta_0}e_{\alpha_0+\beta_0}}\\
&=\exp\left(\begin{array}{ccc} 0 & t_{\alpha_0} & t_{\alpha_0+\beta_0}\\ 0 & 0 & t_{\beta_0}\\ 0 & 0 & 0\end{array}\right)\\
&=I_3+\left(\begin{array}{ccc} 0 & t_{\alpha_0} & t_{\alpha_0+\beta_0}\\ 0 & 0 & t_{\beta_0}\\ 0 & 0 & 0\end{array}\right)+\frac1{2!}\left(\begin{array}{ccc} 0 & t_{\alpha_0} & t_{\alpha_0+\beta_0}\\ 0 & 0 & t_{\beta_0}\\ 0 & 0 & 0\end{array}\right)^2\\
&=\left(\begin{array}{ccc} 1 & t_{\alpha_0} & t_{\alpha_0+\beta_0}+\frac12 t_{\alpha_0}t_{\beta_0}\\ 0 & 1 & t_{\beta_0}\\ 0 & 0 & 1\end{array}\right)\\
&=\left(\begin{array}{ccc} 1 & 0 & t_{\alpha_0+\beta_0}+\frac12 t_{\alpha_0}t_{\beta_0}\\ 0 & 1 & t_{\beta_0}\\ 0 & 0 & 1\end{array}\right)\left(\begin{array}{ccc} 1 & t_{\alpha_0} & 0\\ 0 & 1 & 0\\ 0 & 0 & 1\end{array}\right)\\
&:= n_1\cdot n_2.
\end{align*}
In particular, we have that $d_G(n_1,e)\leq C$ for some constant $C>0$ depending only on $G$, and $n_2\in N_{\alpha_0}$ satisfies $$n_2=\left(\begin{array}{ccc} 1 & t_{\alpha_0} & 0\\ 0 & 1 & 0\\ 0 & 0 & 1\end{array}\right)\quad\mbox{and}\quad\frac1{3\kappa'}e^{-\gamma t}\leq|t_{\alpha_0}x_{-\alpha_0-\beta_0}|\leq\kappa\kappa' e^{-\gamma t},$$
with $|x_{-\alpha_0-\beta_0}|\leq\kappa\kappa' e^{-\gamma t}$.
We know that $h\cdot a_tp$ is a rational point as $\Stab(h\cdot a_tp)\cap N_{\nu}\neq\left\{e\right\}$. By Proposition \ref{p31}, $a_{-t}h\cdot a_tp$ is also a rational point. Let $\bar n=a_{-t}ha_t$. Then $\bar n\cdot p$ is a rational point, and 
\begin{align*}
\bar n\in & \left\{x\in N_+: x=y\cdot B\left(0,Ce^{\beta_0(a_{-t})},Ce^{(\alpha_0+\beta_0)(a_{-t})}\right), y\in N_{\alpha_0}\right.\\
&\left.\frac1{3\kappa'}e^{-\gamma t}\leq d_G(y,e)e^{\alpha_0(a_{t})}|x_{-\alpha_0-\beta_0}|\leq\kappa\kappa' e^{-\gamma t}\right\}.
\end{align*}

We conclude by computing the denominator of the rational point $\bar n\cdot p$. Since $v$ is primitive, the denominator of $h\cdot a_tp$ is equal to $\|x\|_{\mathfrak g_\nu}$. Hence, the denominator of $\bar n\cdot p$ is equal to $$d_{\bar n\cdot p}=e^{(\alpha_0+\beta_0)(a_t)}d_{h\cdot a_tp}=e^{(\alpha_0+\beta_0)(a_t)}\|x\|_{\mathfrak g_\nu}=e^{(\alpha_0+\beta_0)(a_t)}|x_{-\alpha_0-\beta_0}|\leq\kappa\kappa 'e^{(\alpha_0+\beta_0)(a_t)-\gamma t}.$$ In other words, we have that $\bar n\cdot p$ is a rational point, and 
\begin{align*}
\bar n\in&\left\{x\in N_+: x=y\cdot B\left(0,Ce^{\beta_0(a_{-t})},Ce^{(\alpha_0+\beta_0)(a_{-t})}\right),\;y\in N_{\alpha_0}\right.\\ 
&\left.\frac1{3\kappa'}e^{-\gamma t}e^{\beta_0(a_t)}\leq d_G(y,e)\cdot d_{\bar n\cdot p}\leq\kappa\kappa' e^{-\gamma t}e^{\beta_0(a_t)}\right\}
\end{align*} 
and $$d_{\bar n\cdot p}\leq\kappa\kappa' e^{(\alpha_0+\beta_0)(a_t)-\gamma t}.$$ This completes the proof of the second case.

Case 3: $\frac1{3\kappa'}e^{-\gamma t}\leq\|v_{-\alpha_0}\|_{\mathfrak g}\leq\kappa^3e^{-\gamma t}.$ The proof is completely analogous to that of Case 2.
\end{proof}

By Proposition~\ref{p61}, we know that if a point $p\in G/\Gamma$ is not Diophantine of type $\gamma$, there exist a sequence $t_n\to\infty$ and a sequence $v_n\in\Stab(a_{t_n}\cdot p)\setminus\left\{e\right\}$ such that $t_n$ and $v_n$ satisfy the $\gamma$-condition. By passing to a sub-sequence, we may assume that $v_n$ is of type $w_i$ for some $1\leq i\leq 6$ and all $n\in\mathbb N$. Then, Proposition \ref{p62} follows from Lemmas~\ref{l66}, \ref{l67}, \ref{l68} and \ref{l69}.

In light of Proposition \ref{p62}, we introduce five Diophantine subsets $\mathcal E_i$ $(i=1,2,3,4,5)$ in $G/\Gamma$, corresponding to points $p$ admitting sequences $t_{n},\ v_{n}$ satisfying the $\gamma$-condition, where $v_{n}$ are of type $w_{i}$.

\begin{definition}\label{d63}
Let $C,\kappa''>0$ be the constants in Proposition~\ref{p62}.\vspace{2mm}
\begin{enumerate}
\item We denote by $\mathcal E_1$ the subset of points $p\in G/\Gamma$ for which there exist sequences $t_n\to\infty$ of real numbers and $q_n\in G/\Gamma$ of rational points such that $$p\in B\left(0,Ce^{\beta_0(a_{-t_n})},Ce^{(\alpha_0+\beta_0)(a_{-t_n})}\right)w_3^{-1}\cdot q_n$$ and $$\frac1{\kappa''}e^{\beta_0(a_{t_n})-\gamma t_n}\leq d_{q_n}\leq\kappa'' e^{\beta_0(a_{t_n})-\gamma t_n}.$$\vspace{2mm}
\item We denote by $\mathcal E_2$ the subset of points $p\in G/\Gamma$ for which there exist sequences $t_n\to\infty$ of real numbers and $q_n\in G/\Gamma$ of rational points such that $$p\in B\left(Ce^{\alpha_0(a_{-t_n})},0,Ce^{(\alpha_0+\beta_0)(a_{-t_n})}\right)w_2^{-1}\cdot q_n$$ and $$\frac1{\kappa''}e^{\alpha_0(a_{t_n})-\gamma t_n}\leq d_{q_n}\leq\kappa'' e^{\alpha_0(a_{t_n})-\gamma t_n}.$$\vspace{2mm}
\item We denote by $\mathcal E_3$ the subset of points $p\in G/\Gamma$ for which there exist sequences $t_n\to\infty$ of real numbers and $q_n\in G/\Gamma$ of rational points such that $$p\in B\left(Ce^{\alpha_0(a_{-t_n})},Ce^{\beta_0(a_{-t_n})},Ce^{(\alpha_0+\beta_0)(a_{-t_n})}\right)\cdot q_n$$ and $$\frac1{3\kappa''}e^{(\alpha_0+\beta_0)(a_{t_n})-\gamma t_n}\leq d_{q_n}\leq\kappa'' e^{(\alpha_0+\beta_0)(a_{t_n})-\gamma t_n}.$$\vspace{2mm} 
\item we denote by $\mathcal E_4$ the subset of points $p\in G/\Gamma$ for which there exist sequences $t_n\to\infty$ of real numbers and $q_n\in G/\Gamma$ of rational points such that $p\in B_n\cdot q_n$, with
\begin{align*}
B_n=&\bigg\{x\in N_+:x\in B\left(0,Ce^{\beta_0(a_{-t_n})},Ce^{(\alpha_0+\beta_0)(a_{-t_n})}\right)\cdot y,y\in N_{\alpha_0},\\
&\left.\frac1{3\kappa''}e^{-\gamma t_n}e^{\beta_0(a_{t_n})}\leq d_G(y,e)\cdot d_{q_n}\leq\kappa'' e^{-\gamma t_n}e^{\beta_0(a_{t_n})}\right\}
\end{align*}
and $$d_{q_n}\leq\kappa''e^{(\alpha_0+\beta_0)(a_{t_n})-\gamma t_n}.$$\vspace{2mm}
\item We denote by $\mathcal E_5$ the subset of points $p\in G/\Gamma$ for which there exist sequences $t_n\to\infty$ of real numbers and $q_n\in G/\Gamma$ of rational points such that $p\in B_n\cdot q_n$, with
\begin{align*}
B_n=&\bigg\{x\in N_+: x\in B\left(Ce^{\alpha_0(a_{-t_n})},0,Ce^{(\alpha_0+\beta_0)(a_{-t_n})}\right)\cdot y,y\in N_{\beta_0},\\
&\left.\frac1{3\kappa''}e^{-\gamma t_n}e^{\alpha_0(a_{t_n})}\leq d_G(y,e)\cdot d_{q_n}\leq\kappa'' e^{-\gamma t_n}e^{\alpha_0(a_{t_n})}\right\}
\end{align*}
and $$d_{q_n}\leq\kappa''e^{(\alpha_0+\beta_0)(a_{t_n})-\gamma t_n}.$$
\end{enumerate}
\end{definition}
\begin{remark}
We will estimate the Hausdorff dimensions of the sets $\mathcal E_i$ $(1\leq i\leq5)$ in the next section.
\end{remark}

The converse of  Proposition \ref{p62} also holds up to constants. More precisely, we have the following result.

\begin{proposition}\label{p63}
Let $p\in G/\Gamma$, $\gamma>0$ and $\epsilon>0$ a sufficiently small number. Suppose that there exist constants $C,\kappa''\geq 1$, a sequence $t_n\to\infty$ of real numbers, and a sequence $q_n\in G/\Gamma$ of rational points such that one of the following Diophantine conditions holds for all pairs $(t_n,q_n)$ $(\forall n\in N)$:\vspace{2mm}
\begin{enumerate}
\item $p\in B\left(0,Ce^{\beta_0(a_{-t_n})},Ce^{(\alpha_0+\beta_0)(a_{-t_n})}\right)w_3^{-1}\cdot q_n$ and $$\frac12e^{\beta_0(a_{t_n})-(\gamma+\epsilon) t_n}\leq d_{q_n}\leq 2e^{\beta_0(a_{t_n})-(\gamma+\epsilon) t_n};$$\vspace{2mm}
\item $p\in B\left(Ce^{\alpha_0(a_{-t_n})},0,Ce^{(\alpha_0+\beta_0)(a_{-t_n})}\right)w_2^{-1}\cdot q_n$ and $$\frac12e^{\alpha_0(a_{t_n})-(\gamma+\epsilon) t_n}\leq d_{q_n}\leq 2e^{\alpha_0(a_{t_n})-(\gamma+\epsilon) t_n};$$\vspace{2mm}
\item $p\in B\left(Ce^{\alpha_0(a_{-t_n})},Ce^{\beta_0(a_{-t_n})},Ce^{(\alpha_0+\beta_0)(a_{-t_n})}\right)\cdot q_n$ and $$\frac12e^{(\alpha_0+\beta_0)(a_{t_n})-(\gamma+\epsilon) t_n}\leq d_{q_n}\leq 2e^{(\alpha_0+\beta_0)(a_{t_n})-(\gamma+\epsilon) t_n};$$\vspace{2mm}
\item $p\in B_n\cdot q_n$ with 
\begin{align*}
B_n=&\bigg\{x\in N_+: x\in B\left(0,Ce^{\beta_0(a_{-t_n})},Ce^{(\alpha_0+\beta_0)(a_{-t_n})}\right)\cdot y,y\in N_{\alpha_0},\\
&\left.\frac1{3\kappa''}e^{-(\gamma+\epsilon) t_n}e^{\beta_0(a_{t_n})}\leq d_G(y,e)\cdot d_{q_n}\leq\kappa'' e^{-(\gamma+\epsilon) t_n}e^{\beta_0(a_{t_n})}\right\}.
\end{align*}
and $$d_{q_n}\leq\kappa''e^{(\alpha_0+\beta_0)(a_{t_n})-(\gamma+\epsilon) t_n};$$\vspace{2mm}
\item $p\in B_n\cdot q_n$ with  
\begin{align*}
B_n=&\bigg\{x\in N_+: x\in B\left(Ce^{\alpha_0(a_{-t_n})},0,Ce^{(\alpha_0+\beta_0)(a_{-t_n})}\right)\cdot y,y\in N_{\alpha_0},\\
&\left.\frac1{3\kappa''}e^{-(\gamma+\epsilon) t_n}e^{\alpha_0(a_{t_n})}\leq d_G(y,e)\cdot d_{q_n}\leq\kappa'' e^{-(\gamma+\epsilon) t_n}e^{\alpha_0(a_{t_n})}\right\}.
\end{align*}
and $$d_{q_n}\leq\kappa''e^{(\alpha_0+\beta_0)(a_{t_n})-(\gamma+\epsilon) t_n}.$$
\end{enumerate}
Then, $p$ is not Diophantine of type $\gamma$.
\end{proposition}

\begin{proof}
We prove only cases (3) and (4). The proofs of the remaining cases are similar.

Suppose that there exist a sequence $t_n\to\infty$ of real numbers and a sequence $q_n\in G/\Gamma$ of rational points such that $p\in B\left(Ce^{\alpha_0(a_{-t_n})},Ce^{\beta_0(a_{-t_n})},Ce^{(\alpha_0+\beta_0)(a_{-t_n})}\right)\cdot q_n$ and $$\frac12e^{(\alpha_0+\beta_0)(a_{t_n})-(\gamma+\epsilon) t_n}\leq d_{q_n}\leq 2e^{(\alpha_0+\beta_0)(a_{t_n})-(\gamma+\epsilon) t_n}.$$ Let $x_n$ be a generator of $\Stab(q_n)\cap N_{\nu}$ where $\nu$ is the lowest root in $\Phi_-$. Then, $$\frac12e^{(\alpha_0+\beta_0)(a_{t_n})-(\gamma+\epsilon) t_n}\leq d_G(x_n,e)\leq 2e^{(\alpha_0+\beta_0)(a_{t_n})-(\gamma+\epsilon) t_n}.$$ Now, by multiplying both $p$ and $q_n$ by $a_{t_n}$, we obtain $$a_{t_n}\cdot p\in B(C,C,C)\cdot(a_{t_n}\cdot q_n).$$ For the point $q_{n}$ it holds $$\eta(a_{t_n}\cdot q_n)\leq d_G(a_{t_n}x_na_{-t_n},e)\leq2e^{-(\gamma+\epsilon)t_n}.$$ It follows that there exists a constant $C'>0$ depending only on $C$ such that $$\eta(a_{t_n}\cdot p)\leq C'\cdot\eta(a_{t_n}\cdot q_n)\leq (2C'e^{-\epsilon\cdot t_n})e^{-\gamma t_n}.$$ By definition, $p$ is not Diophantine of type $\gamma$.

Now, suppose that there exists a sequence $t_n\to\infty$ of real numbers and a sequence $q_n\in G/\Gamma$ of rational points such that $p\in B_n\cdot q_n$, with
\begin{align*}
B_n=\bigg\{x\in &N_+: x\in B\left(0,Ce^{\beta_0(a_{-t_n})},Ce^{(\alpha_0+\beta_0)(a_{-t_n})}\right)\cdot y,y\in N_{\alpha_0},\\
&\left.\frac1{3\kappa'}e^{-(\gamma+\epsilon) t_n}e^{\beta_0(a_{t_n})}\leq d_G(y,e)\cdot d_{q_n}\leq\kappa^3 e^{-(\gamma+\epsilon) t_n}e^{\beta_0(a_{t_n})}\right\}
\end{align*}
and $$d_{q_n}\leq\kappa^3e^{(\alpha_0+\beta_0)(a_{t_n})-(\gamma+\epsilon) t_n}.$$ Let $x_n$ be a generator of $\Stab(q_n)\cap N_{\nu}$. By multiplying both $p$ and $B_n\cdot q_{n}$ by $a_{t_n}$, we deduce that $$a_{t_n}\cdot p\in B(0,C,C)a_{t_n}\cdot y_{n}\cdot q_n,$$ for some $y_n\in N_{\alpha_0}$ satisfying $$\frac1{3\kappa''}e^{-(\gamma+\epsilon) t_n}e^{\beta_0(a_{t_n})}\leq d_G(y_n,e)\cdot d_{q_n}\leq\kappa'' e^{-(\gamma+\epsilon) t_n}e^{\beta_0(a_{t_n})}.$$
Let $x_n$ be a generator of $\Stab(q_n)\cap N_{\nu}$ where $\nu$ is the lowest root in $\Phi_-$. The previous formula implies that there exists a constant $C'>0$ depending only on $G$ such that $$\eta(a_{t_n}\cdot p)\leq C'\cdot\eta(a_{t_n}\cdot y_n\cdot q_n)\leq C'\cdot d_G(a_{t_n}\cdot y_nx_ny_n^{-1}\cdot a_{t_n}^{-1},e).$$ Write $$x_n=\left(\begin{array}{ccc}1 & 0 & 0\\0 & 1 & 0\\ r_n & 0 & 1 \end{array}\right)\quad y_n=\left(\begin{array}{ccc}1 & s_n & 0\\ 0 & 1 & 0\\ 0 & 0 & 1\end{array}\right).$$ Then, we have $$y_nx_ny_n^{-1}=\left(\begin{array}{ccc}1 & 0 & 0\\0 & 1 & 0\\ r_n & -s_nr_n & 1 \end{array}\right).$$ Since $$|r_n|=d_{q_n}\leq\kappa''e^{(\alpha_0+\beta_0)(a_{t_n})-(\gamma+\epsilon) t_n}$$ $$|s_nr_n|=d_G(y_n,e)d_{q_n}\leq\kappa'' e^{-(\gamma+\epsilon) t_n}e^{\beta_0(a_{t_n})},$$ we obtain $$d_G(a_{t_n}\cdot (y_nx_ny_n^{-1})\cdot a_{t_n}^{-1},e)\leq\kappa^3e^{-(\gamma+\epsilon)t_n}$$ and hence, $$\eta(a_{t_n}\cdot p)\leq C'\cdot\kappa^3e^{-(\gamma+\epsilon)t_n}=(C'\cdot\kappa^3e^{-\epsilon t_n})e^{-\gamma t_n}.$$ Thus, by definition, $p$ is not Diophantine of type $\gamma$. This completes the proof of the proposition.
\end{proof}
\begin{remark}
We will use Proposition~\ref{p63} to construct Cantor-type subsets in $S_\gamma^c$, and to obtain a lower bound for $\dim_HS_\gamma^c$.
\end{remark}

\subsection{The case $\Stab(p)\cap N_-\neq\left\{e\right\}$}

The previous two subsections deal with points $p$ such that $\Stab(p)\cap N_-=\left\{e\right\}$. In this subsection, we study the case $\Stab(p)\cap N_-\neq\left\{e\right\}$.

\begin{lemma}\label{l610}
Let $p\in N_+\Gamma/\Gamma$ such that $\Stab(p)\cap N_-\neq\left\{e\right\}$. Then, one of the following two cases occurs:\vspace{2mm}
\begin{enumerate}
\item $p\in N_{\alpha_0}\cdot n\Gamma$ with $n\in N_+(\mathbb Q)$;\vspace{2mm}
\item $p\in N_{\beta_0}\cdot n\Gamma$ with $n\in N_+(\mathbb Q)$.\vspace{2mm}
\end{enumerate}
\end{lemma}
\begin{proof}
Let $p=g\Gamma$ $(g\in N_+)$. We have that
\begin{align}\nonumber
\Stab(p)\cap N_-\neq\left\{e\right\}&\iff\exists n_-\in N_-\setminus\left\{e\right\}\textup{ such that }g^{-1}n_- g\in\Gamma\\ \nonumber
&\iff\exists(x,y,z)\neq(0,0,0)\textup{ such that }g^{-1}\left(\begin{array}{ccc} 1 & 0 & 0\\ x & 1 & 0\\ z & y & 1\end{array}\right)g\in\Gamma\\
&\iff\exists(x,y,z)\neq(0,0,0)\textup{ such that }g^{-1}\left(\begin{array}{ccc} 0 & 0 & 0\\ x & 0 & 0\\ z & y & 0\end{array}\right)g\in M_3(\mathbb Z).\label{eq6}
\end{align}
Let $g=\left(\begin{array}{ccc} 1 & \alpha & \gamma\\0 & 1 & \beta\\ 0 & 0 & 1\end{array}\right)\in N_+$. Then, Equation~\eqref{eq6} implies that  $$\begin{pmatrix} -\alpha x+(\alpha\beta-\gamma)z & y(\alpha\beta-\gamma)+\alpha((\alpha\beta-\gamma)z-\alpha x) & \beta y(\alpha\beta-\gamma)+\gamma(-\alpha x+(\alpha\beta-\gamma)z)\\ x-\beta z & \alpha(x-\beta z)-\beta y & -\beta^2y+\gamma(x-\beta z)\\ z & y+\alpha z & \beta y+\gamma z\end{pmatrix}$$ is an integral matrix in $M_3(\mathbb Z)$. The rest of the proof consists in analyzing this condition, based on which entries in the matrix are non-null.

Case 1: Assume $z=0$. Then, $$\left(\begin{array}{ccc} -\alpha x & y(\alpha\beta-\gamma)-\alpha^2 x & \beta y(\alpha\beta-\gamma)-\gamma\alpha x\\ x & \alpha x-\beta y & -\beta^2y+\gamma x\\ 0 & y & \beta y\end{array}\right)\in M_3(\mathbb Z).$$

If $x\neq0$ and $y\neq0$, we have $$\alpha=-\frac{-\alpha x}{x}\in\mathbb Q,\;\beta=\frac{\beta y}{y}\in\mathbb Q,$$ $$\gamma x=(-\beta^2y+\gamma x)+\beta^2 y\in\mathbb Q,\;\gamma=\frac{\gamma x}{x}\in\mathbb Q,$$ whence $g\in N_+(\mathbb Q)$. 

If $x\neq 0$ and $y=0$, we have $$\left(\begin{array}{ccc} -\alpha x & -\alpha^2 x & -\gamma\alpha x\\ x & \alpha x & \gamma x\\ 0 & 0 & 0\end{array}\right)\in M_3(\mathbb Z).$$ Then, $$\alpha=\frac{\alpha x}x\in\mathbb Q,\;\gamma=\frac{\gamma x}x\in\mathbb Q,$$ and $$g=\left(\begin{array}{ccc} 1 & 0 & 0\\0 & 1 & \beta\\ 0 & 0 & 1\end{array}\right)\left(\begin{array}{ccc} 1 & \alpha & \gamma\\0 & 1 & 0\\ 0 & 0 & 1\end{array}\right)\in N_{\beta_0}\cdot N_+(\mathbb Q).$$

If $x=0$ and $y\neq 0$, we have $$\left(\begin{array}{ccc} 0 & y(\alpha\beta-\gamma) & \beta y(\alpha\beta-\gamma)\\ 0 & -\beta y & -\beta^2y\\ 0 & y & \beta y\end{array}\right)\in M_3(\mathbb Z).$$ Then, $$\beta=\frac{\beta y}y\in\mathbb Q,\;\alpha\beta-\gamma=\frac{(\alpha\beta-\gamma)y}y\in\mathbb Q,$$ and $$g=\left(\begin{array}{ccc} 1 & \alpha & 0\\0 & 1 & 0\\ 0 & 0 & 1\end{array}\right)\left(\begin{array}{ccc} 1 & 0 & \alpha\beta-\gamma\\0 & 1 & -\beta\\ 0 & 0 & 1\end{array}\right)^{-1}\in N_{\alpha_0}\cdot N_+(\mathbb Q).$$

Case 2: Assume $z\neq 0$. Then,
$$(\alpha\beta-\gamma)z=\beta(y+\alpha z)-(\beta y+\gamma z)\in\mathbb Z\cdot\beta+\mathbb Z$$ and hence, $$\alpha\beta-\gamma\in\mathbb Q\cdot\beta+\mathbb Q.$$ Let $\alpha\beta-\gamma=u_0+v_0\cdot\beta$ $(u_0,v_0\in\mathbb Q)$, $a=\alpha-v_0$ and $b=\beta$. With this notation we have $$g=\left(\begin{array}{ccc} 1 & a & ab\\0 & 1 & b\\ 0 & 0 & 1\end{array}\right)\left(\begin{array}{ccc} 1 & v_0 & -u_0\\0 & 1 & 0\\ 0 & 0 & 1\end{array}\right)\in\left(\begin{array}{ccc} 1 & a & ab\\0 & 1 & b\\ 0 & 0 & 1\end{array}\right)\cdot N_+(\mathbb Q).$$ Equation~\eqref{eq6} then implies that $$\left(\begin{array}{ccc} 1 & -a & 0\\ 0 & 1 & -b\\ 0 & 0 & 1\end{array}\right)\left(\begin{array}{ccc} 0 & 0 & 0\\ x & 0 & 0\\ z & y & 0\end{array}\right)\left(\begin{array}{ccc} 1 & a & ab\\ 0 & 1 & b\\ 0 & 0 & 1\end{array}\right)\in M_3(\mathbb Q)$$ and consequently $$\left(\begin{array}{ccc} -ax & -a^2 x & -a^2b x\\ x-bz & a(x-bz)-by & -b^2y+ab(x-bz)\\ z & y+az & by+abz\end{array}\right)\in M_3(\mathbb Q).$$

If $y+az\neq0$, we have $$b=\frac{by+abz}{y+az}\in\mathbb Q$$ and 
\begin{align*}
g=&\left(\begin{array}{ccc} 1 & a & ab\\0 & 1 & b\\ 0 & 0 & 1\end{array}\right)\left(\begin{array}{ccc} 1 & v_0 & -u_0\\0 & 1 & 0\\ 0 & 0 & 1\end{array}\right)\\
=&\left(\begin{array}{ccc} 1 & a & 0\\0 & 1 & 0\\ 0 & 0 & 1\end{array}\right)\left(\begin{array}{ccc} 1 & 0 & 0\\ 0 & 1 & b\\ 0& 0& 1\end{array}\right)\left(\begin{array}{ccc} 1 & v_0 & -u_0\\0 & 1 & 0\\ 0 & 0 & 1\end{array}\right)\in N_{\alpha_0}\cdot N_+(\mathbb Q).
\end{align*}

If $y+az=0$, we have $$\left(\begin{array}{ccc} -ax & -a^2 x & -a^2b x\\ x-bz & ax & abx\\ z & 0 & 0\end{array}\right)\in M_3(\mathbb Q).$$ We now distinguish two more cases. If $ax\neq0$, then $$a=\frac{a^2x}{ax}\in\mathbb Q,\;b=\frac{abx}{ax}\in\mathbb Q,$$ and hence $g\in N_+(\mathbb Q)$. If $ax=0$, then $$\left(\begin{array}{ccc} 0 & 0 & 0\\ x-bz & 0 & 0\\ z & 0 & 0\end{array}\right)\in M_3(\mathbb Q).$$ If $a=0$, we deduce that $$g=\left(\begin{array}{ccc} 1 & 0 & 0\\0 & 1 & b\\ 0 & 0 & 1\end{array}\right)\left(\begin{array}{ccc} 1 & v_0 & -u_0\\0 & 1 & 0\\ 0 & 0 & 1\end{array}\right)\in N_{\beta_0}\cdot N_+(\mathbb Q).$$ If $a\neq0$, then $x=0$ and we have $$-bz\in\mathbb Q,\;b=-\frac{-bz}z\in\mathbb Q$$ and $$g=\left(\begin{array}{ccc} 1 & a & 0\\0 & 1 & 0\\ 0 & 0 & 1\end{array}\right)\left(\begin{array}{ccc} 1 & 0 & 0\\0 & 1 & b\\ 0 & 0 & 1\end{array}\right)\left(\begin{array}{ccc} 1 & v_0 & -u_0\\0 & 1 & 0\\ 0 & 0 & 1\end{array}\right)\in N_{\alpha_0}\cdot N_+(\mathbb Q).$$

In all the cases we discussed above, we have $p\in N_{\alpha_0}N_+(\mathbb Q)\Gamma$ or $N_{\beta_0}N_+(\mathbb Q)\Gamma$. This completes the proof of the lemma.
\end{proof}

Let $\mathcal E$ be the subset of points $p\in G/\Gamma$ such that $\Stab(p)\cap N_-\neq\left\{e\right\}$. Then, we have the following corollary.

\begin{corollary}
\label{cor:E}
The set $\mathcal E\cap N_{+}\Gamma/\Gamma$, where $\mathcal{E}=\left\{p\in G/\Gamma:\Stab(p)\cap N_{-}\neq\left\{e\right\}\right\}$, has Hausdorff dimension at most $1$.
\end{corollary}

\begin{proof}
In view of Lemma \ref{l610}, it is enough to show that the sets $N_{\alpha_0}N_+(\mathbb Q)\Gamma$ and $N_{\beta_0}N_+(\mathbb Q)\Gamma$ have Hausdorff dimension at most $1$. Since the Hausdorff dimension is stable under countable unions, the thesis follows if we show that for each point $q\in N_{+}(\mathbb Q)\Gamma/\Gamma$, the sets $N_{\alpha_0}q\Gamma/\Gamma$ and $N_{\beta_0}q\Gamma/\Gamma$ have Hausdorff dimension at most $1$. This is clearly true for the sets $N_{\alpha_0}q$ and $N_{\beta_0}q$. Since the injectivity radius admits a minimum on $N_{+}\Gamma/\Gamma$, we may split the sets $N_{\alpha_0}q$ and $N_{\beta_0}q$ into countably many subsets homeomorphic to their projection onto $N_{+}\Gamma/\Gamma$. It follows that the projection of $N_{\alpha_0}q$ and $N_{\beta_0}q$ has at most dimension $1$.
\end{proof}

\section{Non-emptiness Conditions}

In this section, we analyze the range of admissible values for the exponents $\gamma$. We aim to prove the following result, which is part of Theorem \ref{mthm}.

\begin{proposition}\label{prop:allDioph}
All points $p$ in $N_{+}\Gamma/\Gamma$ are Diophantine of type $(\alpha_{0}+\beta_{0})(X_{0}).$
\end{proposition}

We subdivide the proof of this result into three steps. We start by showing that Proposition \ref{prop:allDioph} holds for rational points.

\begin{lemma}\label{l71}
All rational points $q$ in $N_{+}\Gamma/\Gamma$ are Diohantine of type $(\alpha_{0}+\beta_{0})(X_{0})$. 
\end{lemma}

\begin{proof}
By Proposition \ref{p31}, we have that any rational point $q\in N_{+}\Gamma/\Gamma$ has the form $q=an_{-}\Gamma$ for some $a\in A$ and $n_{-}\in N_{-}$. It follows that
$$\eta(a_{t}q)=\eta(a_{t}an_{-}\Gamma)=\eta(a_{t}an_{-}a_{-t}\cdot a_{t}\Gamma).$$
Now, for all $t>0$ the elements $a_{t}an_{-}a_{-t}$ belong to a fixed compact set $L\subset G$. Since the map $x\mapsto \Ad(h)x$ from $L\times \left\{x\in\mathfrak g: \|x\|_{\mathfrak g}=1\right\}$ to $\mathfrak g$ is continuous, there exists a constant $C_{L}>0$ such that
$$\|\Ad(a_{t}an_{-}a_{-t})x\|_{\mathfrak g}\geq C_{L}^{-1}\|x\|_{\mathfrak g}$$
independently of $t$. It follows that
$$\eta(a_{t}an_{-}a_{-t}\cdot a_{t}\Gamma)\sim_{C_{L}}\eta(a_{t}\Gamma)=e^{-\alpha_{0}(tX_{0})-\beta_{0}(tX_{0})}.$$
This concludes the proof.
\end{proof}

\begin{remark}
\label{rmk:remainingdim}
Note that the proof of Lemma \ref{l71} shows also that if $p\in G/\Gamma$ is any point, then the Diophantine type of the point $an_{-}p$ is the same as the Diophantine type of the point $p$ for all $a\in A$ and $n_{-}\in N_{-}$.
\end{remark}

The next Lemma deals with the case $\Stab(p)\cap N_{-}=\left\{e\right\}$.

\begin{lemma}\label{l72}
All points $p\in N_{+}\Gamma/\Gamma$ such that $\Stab(p)\cap N_{-}= \left\{e\right\}$ are Diophantine of type $(\alpha_{0}+\beta_{0})(X_{0})$.
\end{lemma}

\begin{proof}
Assume by contradiction that there exists a point $p\in G/\Gamma$ that is not Diophantine of type $\gamma$, with $\gamma\geq (\alpha_{0}+\beta_{0})(X_{0})$ and $\Stab(p)\cap N_{-}=\left\{e\right\}$. By Proposition \ref{p61}, we have $p\in\mathcal E_{i}$ for some $1\leq i\leq 6$. In the following, we will consider only the cases $i=1,3,4$. The proof for the remaining cases is similar.

Suppose that $p\in \mathcal E_1\cap U_0$. Then, there exist a sequence $t_n\to\infty$ of real numbers and a sequence $q_n\in G/\Gamma$ of rational points such that 
\begin{align}\label{eq7}
p\in B\left(0,Ce^{\beta_0(a_{-t_n})}, Ce^{(\alpha_0+\beta_0)(a_{-t_n})}\right)w_3^{-1}\cdot q_n
\end{align}
and 
\begin{align}\label{eq8}
\frac1{\kappa''}e^{\beta_0(a_{t_n})-\gamma t_n}\leq d_{q_n}\leq\kappa'' e^{\beta_0(a_{t_n})-\gamma t_n}.
\end{align}
Since $p\in N_+\Gamma/\Gamma$, we have $$q_n\in w_{3}N_{+}\Gamma/\Gamma=w_{3}N_{+}w_{3}^{-1}\Gamma/\Gamma= N_{-\alpha_0}N_{\beta_0}N_{\alpha_0+\beta_0}\Gamma/\Gamma.$$ Hence, we can write $q_n=x_n\cdot y_n$ with $x_n\in N_{-\alpha_0},$ and $y_n\in N_{\beta_0}N_{\alpha_0+\beta_0}\Gamma.$ We observe that, by Proposition \ref{p31} and Lemma \ref{l41}, acting by an element in $N_{-}$ preserves rationality and denominators. Thus, by Lemma~\ref{l32} and Theorem~\ref{thm51}, we have $$y_n\in N_{\beta_0}(\mathbb Q)N_{\alpha_0+\beta_0}(\mathbb Q)\Gamma,\quad\mbox{and}\quad d_{q_n}=d_{y_n}\in\mathbb N.$$
However, (\ref{eq8}) and the fact that $\gamma\geq (\alpha_{0}+\beta_{0})(X_{0})$ imply that $d_{q_{n}}\to 0$, whence a contradiction. 

Now, suppose that $p\in \mathcal E_3\cap U_0\neq\emptyset$. Then, there exist a sequence $t_n\to\infty$ of real numbers and a sequence $q_n\in G/\Gamma$ of rational points such that 
\begin{align}\label{eq9}
p\in B\left(Ce^{\alpha_0(a_{-t_n})},Ce^{\beta_0(a_{-t_n})},Ce^{(\alpha_0+\beta_0)(a_{-t_n})}\right)\cdot q_n
\end{align}
and 
\begin{align}\label{eq10}
\frac1{3\kappa''}e^{(\alpha_0+\beta_0)(a_{t_n})-\gamma t_n}\leq d_{q_n}\leq\kappa'' e^{(\alpha_0+\beta_0)(a_{t_n})-\gamma t_n}.
\end{align}
Since $p\in N_+\Gamma/\Gamma$, we have $q_n\in N_+\Gamma/\Gamma$ and thus $d_{q_n}\in\mathbb N$, by Theorem~\ref{thm51}. Equation~\eqref{eq10} then implies that $$\gamma=(\alpha_0+\beta_0)(X_{0})$$ and hence, $$\frac1{3\kappa''}\leq d_{q_n}\leq\kappa''.$$ By Theorem~\ref{thm51}, there are only finitely many rational points with this property. Therefore, after passing to a subsequence, we may assume that $q_n=q$ is a fixed rational point in $N_+\Gamma/\Gamma$. Then, Equation~\eqref{eq9} implies $p=q$, which, by Lemma \ref{l71}, contradicts the assumption that $p$ is not Diophantine of type $\gamma$. This completes the proof for $i=3$.

Finally, assume that $p\in\mathcal E_{4}$. Then, there exist a sequence $t_n\to\infty$ of real numbers and a sequence $q_n\in G/\Gamma$ of rational points such that $p\in B_n\cdot q_n$, with 
\begin{align}
\label{eq:11}
B_n=&\bigg\{x\in N_+: x\in B\left(0,Ce^{\beta_0(a_{-t_n})},Ce^{(\alpha_0+\beta_0)(a_{-t_n})}\right)\cdot y, y\in N_{\alpha_0},\nonumber\\
&\left.\frac1{3\kappa''}e^{-\gamma t_n}e^{\beta_0(a_{t_n})}\leq d_G(y,e)\cdot d_{q_n}\leq\kappa'' e^{-\gamma t_n}e^{\beta_0(a_{t_n})}\right\}
\end{align}
and
\begin{equation}
\label{eq:12}
d_{q_n}\leq\kappa''e^{(\alpha_0+\beta_0)(a_{t_n})-\gamma t_n}.
\end{equation}
Since $p\in N_{+}\Gamma/\Gamma$, we also have $q_{n}\in N_{+}\Gamma/\Gamma$. Hence, by Theorem \ref{thm51}, $d_{q_{n}}\in\mathbb{N}$. Then, it follows from (\ref{eq:12}) that $\gamma=(\alpha_{0}+\beta_{0})(X_{0})$, and, as for $i=3$, we may assume that $q_{n}=q$ for some fixed rational point $q\in N_{+}\Gamma/\Gamma$. By (\ref{eq:11}) we deduce that $p=q$ and this contradicts Lemma \ref{l71}, concluding the proof.
\end{proof}

The final lemma of this section deals with points $p$ such that $\Stab(p)\cap N_{-}\neq\left\{e\right\}$.

\begin{lemma}
All points $p\in N_{+}\Gamma/\Gamma$ such that $\Stab(p)\cap N_{-}\neq\left\{e\right\}$ are Diophantine of type $(\alpha_{0}+\beta_{0})(X_{0})$.
\end{lemma}

\begin{proof}
Assume by contradiction that there exists a point $p\in N_{+}\Gamma/\Gamma$ such that $\Stab(p)\cap N_{-}\neq\left\{e\right\}$ and $p$ is not Diophantine of type $\gamma$ for some $\gamma\geq (\alpha_{0}+\beta_{0})(X_{0})$. By the definition of Diophantine points and part $(2)$ of Lemma~\ref{l61}, for any $\varepsilon>0$ there exist a sequence $t_n\to +\infty$ of real numbers and elements $u_n\in\Stab(p)\setminus\left\{e\right\}$ such that\vspace{2mm}
\begin{enumerate}
\item $a_{t_n}u_na_{t_n}^{-1}\in\Stab(a_{t_n}\cdot p)$;\vspace{2mm}
\item $d_G(a_{t_n}u_na_{t_n}^{-1},e)\leq C\cdot \epsilon e^{-\gamma t_n}$;\vspace{2mm}
\item $u_n$ is primitive in $\textup{Stab}(p)$ and is conjugate to an element of $N_\nu$, where $\nu$ is the lowest root in $\Phi_-$.
\end{enumerate}
The proof of Lemma \ref{l62} breaks down here as the elements $u_{n}$ in the sequence above might be not pairwise distinct. If there are infinitely many such elements, up to passing to a sub-sequence we can still run the proof of Lemma \ref{l62} and deduce that, for the point $p$, the conclusion in Proposition \ref{p61} holds. On applying Proposition \ref{p62}, we deduce that $p\in \mathcal E_{i}$ for some $1\leq i\leq 6$ and this contradicts Lemma \ref{l72}. Therefore, we may assume, without loss of generality that $u_{n}=u$ for all $n$, with $u\in \Stab(a_{t_{1}}p)$. If this is the case, we deduce from part $(2)$ that it cannot be $\gamma>(\alpha_{0}+\beta_{0})(X_{0})$, as the maximal contraction rate of $\Ad(a_{t})$ is $\exp((-\alpha_{0}-\beta_{0})(tX_{0}))$. If $\gamma=(\alpha_{0}+\beta_{0})(X_{0})$, we also deduce that it must be $u\in N_{\nu}$ where $\nu=-\alpha_{0}-\beta_{0}$. However, by definition Definition \ref{d31}, this implies that $p$ itself is rational, in contradiction with Lemma \ref{l71}. This concludes the proof.
\end{proof}

\section{Upper Bounds for the Hausdorff dimension}\label{estimates}

In this section, we will estimate the Hausdorff dimension of Diophantine subsets defined in \S\ref{Diophantine}. The main result will be the following.

\begin{proposition}
\label{prop:upperbound}
For $1\leq i\leq 5$ it holds
$$\dim_{H}(\mathcal E_{i}\cap N_{+}\Gamma/\Gamma)\leq 3-\frac{2\gamma}{\alpha_0(X_{0})+\beta_0(X_{0})}.$$
\end{proposition}

As a corollary, we deduce an upper bound for the Hausdorff dimension of $S_{\gamma}^{c}\cap N_{+}\Gamma/\Gamma$.

\begin{corollary}
\label{cor:upperbound}
For any $\gamma< \alpha_0(X_{0})+\beta_0(X_{0})$ we have
$$\dim_{H}(S_{\gamma}^{c}\cap N_{+}\Gamma/\Gamma)\leq 3-\frac{2\gamma}{\alpha_0(X_{0})+\beta_0(X_{0})}.$$
\end{corollary}

\begin{proof}
By Proposition \ref{p62}, we have that
$$S_{\gamma}^{c}\cap N_{+}\Gamma/\Gamma\subset \mathcal E\cap N_{+}\Gamma/\Gamma \cup\bigcup_{i=1}^{5}\mathcal E_{i}\cup N_{+}\Gamma/\Gamma,$$
where $\mathcal E=\left\{p\in G/\Gamma:\Stab(p)\cap N_{-}\neq \left\{e\right\}\right\}$. Moreover, Corollary \ref{cor:E} implies that $\dim_{H}(E\cap N_{+}\Gamma/\Gamma)\leq 1$. Since the upper bound in Proposition \ref{prop:upperbound} is larger than $1$ when $\gamma< \alpha_0(X_{0})+\beta_0(X_{0})$, the claim follows.
\end{proof}

From now on, we will fix a small open box $U_0$ in $N_{+}\Gamma/\Gamma$ and discuss $\dim_H(S_\gamma^c\cap U_0)$. The proof of Proposition \ref{prop:upperbound} will be naturally subdivided into $5$ cases and will occupy the remainder of this section. 

\subsection{The Diophantine sets $\mathcal E_1$ and $\mathcal E_{2}$}
By definition, $\mathcal E_1\cap U_0$ is the subset of points $p\in U_0$ for which there exist a sequence $t_n\to\infty$ of real numbers and a sequence $q_n\in G/\Gamma$ of rational points such that 
\begin{equation}
\label{eq:upperbounds1}
p\in B\left(0,Ce^{\beta_0(a_{-t_n})}, Ce^{(\alpha_0+\beta_0)(a_{-t_n})}\right)w_3^{-1}\cdot q_n
\end{equation}
and 
\begin{equation}
\label{eq:upperbounds2}
\frac1{\kappa''}e^{\beta_0(a_{t_n})-\gamma t_n}\leq d_{q_n}\leq\kappa'' e^{\beta_0(a_{t_n})-\gamma t_n}.
\end{equation}

\begin{remark}
Note that we may always assume $\gamma<\beta_{0}(a_{1})$, since otherwise (\ref{eq:upperbounds1}) and (\ref{eq:upperbounds2}) imply that $w_{3}p$ is rational and thus $\Stab(p)\cap N_{-}\neq \left\{e\right\}$. In this case, the required upper bound follows from Corollary \ref{cor:E}. 
\end{remark}

It follows from (\ref{eq:upperbounds1}) and (\ref{eq:upperbounds2}) that there exists $C'>0$, depending only on $C$ and $\kappa''$, such that $$w_3\cdot p\in B\left(0,C'd_{q_n}^{\frac{(\alpha_0+\beta_0)(a_{-1})}{\beta_0(a_1)-\gamma}},C'd_{q_n}^{\frac{\beta_0(a_{-1})}{\beta_0(a_1)-\gamma}}\right)\cdot q_n.$$

If we write $p=g\Gamma$, where $$g=\left(\begin{array}{ccc} 1 & x & z\\ 0 & 1 & y\\ 0 & 0 & 1\end{array}\right)\in N_+,$$ we have $$w_{3}\cdot g\Gamma=w_{3}gw_{3}^{-1}\Gamma=\left(\begin{array}{ccc} 1 & 0 & -y\\ -x & 1 & z\\ 0 & 0 & 1\end{array}\right)\Gamma\in B\left(0,C'd_{q_n}^{\frac{(\alpha_0+\beta_0)(a_{-1})}{\beta_0(a_1)-\gamma}},C'd_{q_n}^{\frac{\beta_0(a_{-1})}{\beta_0(a_1)-\gamma}}\right)\cdot q_n.$$
Moreover, since $p\in N_{+}\Gamma/\Gamma$, we have
$$q_{n}\in w_{3}N_{+}\Gamma/\Gamma=w_{3}N_{+}w_{3}^{-1}\Gamma/\Gamma= N_{-\alpha_0}N_{\beta_0}N_{\alpha_0+\beta_0}\Gamma/\Gamma.$$
By Lemma~\ref{l32}, we deduce that $$q_n\in N_{-\alpha_0}\cdot N_{\beta_0}(\mathbb Q)N_{\alpha_0+\beta_0}(\mathbb Q)\Gamma$$
and, if $q_{n}=x_{n}\cdot y_{n}$ with $x_{n}\in N_{-\alpha_{0}}$ and $y_{n}\in N_{\beta_0}(\mathbb Q)N_{\alpha_0+\beta_0}(\mathbb Q)\Gamma$, we also have $d_{q_{n}}=d_{y_{n}}$.

This implies that the collection of sets
\begin{align*}
\left\{B\left(0,C'd_{q}^{\frac{(\alpha_0+\beta_0)(a_{-1})}{\beta_0(a_1)-\gamma}},C'd_{q}^{\frac{\beta_0(a_{-1})}{\beta_0(a_1)-\gamma}}\right)\cdot q: q\in(N_{-\alpha_0}\cdot N_{\beta_0}(\mathbb Q)N_{\alpha_0+\beta_0}(\mathbb Q)\Gamma)\right\}
\end{align*}
covers $w_3\cdot(\mathcal E_1\cap U_0)$.
These sets are in turn contained in the larger sets $$\left\{B\left(0,C'd_{q}^{\frac{(\alpha_0+\beta_0)(a_{-1})}{\beta_0(a_1)-\gamma}},C'd_{q}^{\frac{\beta_0(a_{-1})}{\beta_0(a_1)-\gamma}}\right)\cdot N_{-\alpha_0}\cdot q: q\in N_{\beta_0}(\mathbb Q)N_{\alpha_0+\beta_0}(\mathbb Q)\Gamma\right\},$$
which we will use as a covering of $w_3\cdot(\mathcal E_1\cap U_0)$.

We will now need the following lemma counting rational points with denominators less than $l>0$.

\begin{lemma}\label{l73}
For any sufficiently large $l>0$, we have $$|\left\{r\in N_{\beta_0}(\mathbb Q)N_{\alpha_0+\beta_0}(\mathbb Q)\Gamma:d_r\leq l\right\}|\lesssim l^2.$$ 
\end{lemma}
\begin{proof}
For any $r\in N_{\beta_0}(\mathbb Q)N_{\alpha_0+\beta_0}(\mathbb Q)\Gamma$, we may write $$q=\left(\begin{array}{ccc} 1 & 0 & \frac{p_1}q\\ 0 & 1 & \frac{p_2}q\\ 0 & 0 & 1\end{array}\right)\Gamma$$ where $\gcd(p_1,p_2,q)=1$. By Theorem~\ref{thm51}, we know that $d_r=q^2/d$ where $d=\gcd(q,p_1)$. Let $p_1=ad$ and $q=bd$ $(\gcd(a,b)=1,\gcd(d,p_2)=1)$. Then 
\begin{align*}
&|\left\{r\in N_{\beta_0}(\mathbb Q)N_{\alpha_0+\beta_0}(\mathbb Q)\Gamma:d_r\leq l\right\}|\\
=&\sum_{(p_1/q,p_2/q)\in[0,1]^2,q^2/d\leq l}1\\
=&\sum_{(a/b,p_2/bd)\in[0,1]^2,b^2d\leq l}1\\
\leq&\sum_{b^2d\leq l}b^2d\\
\leq&l\cdot\sum_{b\leq\sqrt{l}}\sum_{d\leq l/b^2}1\lesssim l^2.
\end{align*}
This completes the proof of the lemma.
\end{proof}

We can now estimate the Hausdorff dimension of the set $\mathcal E_1\cap U_0$. We know that $$\dim_H(\mathcal E_1\cap U_0)=\dim_H(w_3\cdot(\mathcal E_1\cap U_0))$$ and $$\left\{B\left(0,C'd_{q}^{\frac{(\alpha_0+\beta_0)(a_{-1})}{\beta_0(a_{1})-\gamma}},C'd_{q}^{\frac{\beta_0(a_{-1})}{\beta_0(a_{1})-\gamma}}\right)\cdot\Upsilon_{-\alpha_0}\cdot q: q\in N_{\beta_0}(\mathbb Q)N_{\alpha_0+\beta_0}(\mathbb Q)\Gamma\right\}$$ is a collection which covers $w_3\cdot(\mathcal E_1\cap U_0).$ We subdivide each set $$B\left(0,C'd_{q}^{\frac{(\alpha_0+\beta_0)(a_{-1})}{\beta_0(a_{1})-\gamma}},C'd_{q}^{\frac{\beta_0(a_{-1})}{\beta_0(a_{1})-\gamma}}\right)\cdot\Upsilon_{-\alpha_0}\cdot q$$ in this collection into cubes of side length $$C'd_{q}^{\frac{(\alpha_0+\beta_0)(a_{-1})}{\beta_0(a_{1})-\gamma}}$$ and denote the family of cubes thus obtained by $\mathcal F$. Note that $\mathcal F$ still covers the set $w_3\cdot(\mathcal E_1\cap U_0)$. In order to estimate the Hausdorff dimension of $w_3\cdot(\mathcal E_1\cap U_0)$, we study the convergence of the series

\begin{align*}
&\sum_{B\in\mathcal F}(\diam B)^s\\
\lesssim &\sum_{q\in N_{\beta_0}(\mathbb Q)N_{\alpha_0+\beta_0}(\mathbb Q)\Gamma}\left(d_{q}^{\frac{(\alpha_0+\beta_0)(a_{-1})}{\beta_0(a_{1})-\gamma}}\right)^sd_{q}^{\frac{\alpha_0(a_{1})}{\beta_0(a_{1})-\gamma}}d_{q}^{\frac{(\alpha_0+\beta_0)(a_{1})}{\beta_0(a_{1})-\gamma}}\\
\leq &\sum_{n=0}^\infty\sum_{2^n\leq d_{q}\leq 2^{n+1}}d_{q}^{-\frac{(\alpha_0+\beta_0)(a_{1})}{\beta_0(a_{1})-\gamma}s+\frac{2\alpha_0(a_{1})+\beta_0(a_1)}{\beta_0(a_{1})-\gamma}}\\
\leq&\sum_{n=0}^\infty(2^n)^{-\frac{(\alpha_0+\beta_0)(a_{1})}{\beta_0(a_{1})-\gamma}s+\frac{2\alpha_0(a_{1})+\beta_0(a_1)}{\beta_0(a_{1})-\gamma}}(2^{n+1})^2\\
\sim& \sum_{n=0}^\infty2^{\left(-\frac{(\alpha_0+\beta_0)(a_{1})}{\beta_0(a_{1})-\gamma}s+\frac{2\alpha_0(a_{1})+\beta_0(a_1)}{\beta_0(a_{1})-\gamma}+2\right)n},
\end{align*}
where we have used Lemma~\ref{l73} in the third inequality. This series converges if $$s>\frac{\frac{2\alpha_0(a_{1})+\beta_0(a_1)}{\beta_0(a_{1})-\gamma}+2}{\frac{(\alpha_0+\beta_0)(a_{1})}{\beta_0(a_{1})-\gamma}}=\frac{2\alpha_0(a_{1})+3\beta_0(a_{1})-2\gamma}{\alpha_0(a_1)+\beta_0(a_{1})}.$$ This implies that $$\dim_H(w_3\cdot(\mathcal E_1\cap U_0))\leq\frac{2\alpha_0(a_{1})+3\beta_0(a_{1})-2\gamma}{\alpha_0(a_1)+\beta_0(a_{1})}$$ and hence, $$\dim_H(\mathcal E_1\cap U_0)=\dim_H(w_3\cdot(\mathcal E_1\cap U_0))\leq\frac{2\alpha_0(a_{1})+3\beta_0(a_{1})-2\gamma}{\alpha_0(a_1)+\beta_0(a_{1})}.$$ This completes the computation for the upper bound of the Hausdorff dimension of $\mathcal E_1\cap U_0$.

The proof for the set $\mathcal E_{2}\cap U_{0}$ is completely analogous. We report here the related counting result for completeness.

\begin{lemma}\label{l74}
For any sufficiently large $l>0$, we have $$|\left\{r\in N_{\alpha_0}(\mathbb Q)N_{\alpha_0+\beta_0}(\mathbb Q)\Gamma:d_r\leq l\right\}|\lesssim l^2.$$
\end{lemma}
\begin{proof}
For any $r\in N_{\alpha_0}(\mathbb Q)N_{\alpha_0+\beta_0}(\mathbb Q)\Gamma$, we may write $$r=\left(\begin{array}{ccc} 1 & \frac ab & \frac{p}q\\ 0 & 1 & 0\\ 0 & 0 & 1\end{array}\right)\Gamma,$$ where $\gcd(p,q)=1$ and $\gcd(a,b)=1$. By Theorem~\ref{thm51}, we know that $d_r=bq^2/d$, where $d=\gcd(q,b)$. Let $b=b_1d$ and $q=q_1d$ $(\gcd(b_1,q_1)=1)$. Then, we have
\begin{align*}
&|\left\{r\in N_{\alpha_0}(\mathbb Q)N_{\alpha_0+\beta_0}(\mathbb Q)\Gamma:d_r\leq l\right\}|\\
=&\sum_{(a/b,p/q)\in[0,1]^2,bq^2/d\leq l}1\\
\leq&\sum_{b_1q^2\leq l}bq\\
\leq&\sum_{b_1q^2\leq l}b_1qd\\
\leq&\sum_{b_1q^2\leq l}b_1q^{2}\\
\leq &l\cdot\sum_{q\leq \sqrt{l}}\sum_{b_{1}\leq l/q^{2}}1\lesssim l^2.
\end{align*}
This completes the proof of the lemma.
\end{proof}

\subsection{The Diophantine set $\mathcal E_3\cap U_0$}

By definition, $\mathcal E_3\cap U_0$ is the subset of points $p\in U_0$ such that there exist a sequence $t_n\to\infty$ of real numbers and a sequence $q_n\in G/\Gamma$ of rational points such that $$p\in B\left(Ce^{\alpha_0(a_{-t_n})},Ce^{\beta_0(a_{-t_n})},Ce^{(\alpha_0+\beta_0)(a_{-t_n})}\right)\cdot q_n$$ and $$\frac1{3\kappa''}e^{(\alpha_0+\beta_0)(a_{t_n})-\gamma t_n}\leq d_{q_n}\leq\kappa'' e^{(\alpha_0+\beta_0)(a_{t_n})-\gamma t_n}.$$ This implies that there exists a constant $C'>0$, depending only $C$ and $\kappa''$, such that the following collection of subsets covers $\mathcal E_3\cap U_0$ $$\left\{B\left(C'd_q^{\frac{\alpha_0(a_{-1})}{(\alpha_0+\beta_0)(a_1)-\gamma}},C'd_q^{\frac{\beta_0(a_{-1})}{(\alpha_0+\beta_0)(a_1)-\gamma}},C'd_q^{\frac{(\alpha_0+\beta_0)(a_{-1})}{(\alpha_0+\beta_0)(a_1)-\gamma}}\right)\cdot q: q\in N_{+}(\mathbb Q)\Gamma\right\}.$$ 
 
In the following lemma we count rational points with denominator less than $l>0$.

\begin{lemma}\label{l75}
For any $\epsilon>0$, and for any sufficiently large $l>0$, we have $$|\left\{r\in N_{+}(\mathbb Q)\Gamma:d_r\leq l\right\}|\lesssim_{\varepsilon} l^{2+\epsilon}$$ where the implicit constant depends only on $\epsilon>0$.
\end{lemma}
\begin{proof}
For any $r\in N_{\alpha_0}(\mathbb Q)N_{\alpha_0+\beta_0}(\mathbb Q)\Gamma$, we may write $$q=\left(\begin{array}{ccc} 1 & \frac ab & 0\\ 0 & 1 & 0\\ 0 & 0 & 1\end{array}\right)\left(\begin{array}{ccc} 1 & 0 & \frac{p_1}q\\ 0 & 1 & \frac{p_2}q\\ 0 & 0 & 1\end{array}\right)\Gamma,$$ where $\gcd(p_1,p_2,q)=1$ and $\gcd(a,b)=1$. By Theorem~\ref{thm51}, we know that $$d_r=\frac{bq^2}{d},$$ where $d=\gcd(q, bp_1+ap_2)$. We observe that for fixed $a$ and $b$ the congruence
$$bp_{1}+ap_{2}\equiv 0\pmod{d}$$
has at most $q^{2}/d$ solutions $0\leq p_{1},p_{2}\leq q$. This follows from the fact that $(a,b)=1$ implies $(b,d)|p_{2}$ and, once the value of $p_{2}$ is fixed, the value of $p_{1}$ is uniquely determined modulo $d/(b,d)$. In view of this observation, we deduce
Then, we have
\begin{align*}
&|\left\{r\in N_{+}(\mathbb Q)\Gamma:d_r\leq l\right\}|\\
\leq&\sum_{(a,b)=1,(p_1,p_2,q)=1,bq^2/d\leq l, d=(bp_1+ap_2,q)}1\\
\leq &\sum_{q\leq l}\sum_{d|q}\sum_{bq^{2}/d\leq l}\frac{bq^{2}}{d}\\
\lesssim &\sum_{q\leq l}\sum_{d|q}\frac{dl^{2}}{q^{2}}\\
\leq&l^2\sum_{q\leq l}\frac{\sigma(q)}{q^2}\sim l^{2+\epsilon},
\end{align*}
where $\sigma(q)$ is the sum of divisors of $q$, and we have used the inequalities
$$\sigma(q)\lesssim q\log\log q$$
and
$$\sum_{q\leq l}\frac{\log\log q}{q}\lesssim \log l\log\log l.$$
For a proof of these, see \cite[Theorems 422 and 323]{HW60} respectively. This completes the proof of the lemma.
\end{proof}

Now, we proceed to estimate the Hausdorff dimension of $\mathcal E_3\cap U_0$. We know that $$\left\{B\left(C'd_q^{\frac{\alpha_0(a_{-1})}{(\alpha_0+\beta_0)(a_1)-\gamma}},C'd_q^{\frac{\beta_0(a_{-1})}{(\alpha_0+\beta_0)(a_1)-\gamma}},C'd_q^{\frac{(\alpha_0+\beta_0)(a_{-1})}{(\alpha_0+\beta_0)(a_1)-\gamma}}\right)\cdot q: q\in N_{+}(\mathbb Q)\Gamma\right\}$$ is a covering of $\mathcal E_3\cap U_0.$ We subdivide each set $$B\left(Cd_q^{\frac{\alpha_0(a_{-1})}{(\alpha_0+\beta_0)(a_1)-\gamma}},Cd_q^{\frac{\beta_0(a_{-1})}{(\alpha_0+\beta_0)(a_1)-\gamma}},Cd_q^{\frac{(\alpha_0+\beta_0)(a_{-1})}{(\alpha_0+\beta_0)(a_1)-\gamma}}\right)\cdot q$$ into cubes of side length $C'd_q^{\frac{(\alpha_0+\beta_0)(a_{-1})}{(\alpha_0+\beta_0)(a_1)-\gamma}}$ and we denote the family of such cubes by $\mathcal F$. In order to estimate the Hausdorff dimension of $\mathcal E_3\cap U_0$, we need to study the series
\begin{align*}
&\sum_{B\in\mathcal F}\diam(B)^s\\
\leq&\sum_{q\in N_+(\mathbb Q)/\Gamma}\left(d_{q}^{\frac{(\alpha_0+\beta_0)(a_{-1})}{(\alpha_0+\beta_0)(a_{1})-\gamma}}\right)^sd_{q}^{\frac{(\alpha_0+\beta_0)(a_{1})}{(\alpha_0+\beta_0)(a_{1})-\gamma}}\\
\leq&\sum_{n=0}^\infty\sum_{2^n\leq d_{q}\leq 2^{n+1}}d_{q}^{-\frac{(\alpha_0+\beta_0)(a_{1})}{(\alpha_0+\beta_0)(a_{1})-\gamma}s+\frac{(\alpha_0+\beta_0)(a_{1})}{(\alpha_0+\beta_0)(a_{1})-\gamma}}\\
\leq&\sum_{n=0}^\infty(2^n)^{-\frac{(\alpha_0+\beta_0)(a_{1})}{(\alpha_0+\beta_0)(a_{1})-\gamma}s+\frac{(\alpha_0+\beta_0)(a_{1})}{(\alpha_0+\beta_0)(a_{1})-\gamma}}(2^{n+1})^{2+\epsilon}\\
\sim &\sum_{n=0}^\infty2^{\left(-\frac{(\alpha_0+\beta_0)(a_{1})}{(\alpha_0+\beta_0)(a_{1})-\gamma}s+\frac{(\alpha_0+\beta_0)(a_{1})}{(\alpha_0+\beta_0)(a_{1})-\gamma}+2+\epsilon\right)n},
\end{align*}
where we have used Lemma~\ref{l75}, in the third inequality. This series converges if $$s>\frac{\frac{(\alpha_0+\beta_0)(a_{1})}{(\alpha_0+\beta_0)(a_{1})-\gamma}+2+\epsilon}{\frac{(\alpha_0+\beta_0)(a_{1})}{(\alpha_0+\beta_0)(a_{1})-\gamma}}=\frac{(\alpha_0+\beta_0)(a_{1})+(2+\epsilon)((\alpha_0+\beta_0)(a_{1})-\gamma)}{(\alpha_0+\beta_0)(a_{1})}$$ and hence, $$\dim_H(\mathcal E_3\cap U_0)\leq\frac{(\alpha_0+\beta_0)(a_{1})+(2+\epsilon)((\alpha_0+\beta_0)(a_{1})-\gamma)}{(\alpha_0+\beta_0)(a_{1})}.$$ Since $\epsilon$ can be chosen arbitrarily small, we conclude that $$\dim_H(\mathcal E_3\cap U_0)\leq\frac{3(\alpha_0+\beta_0)(a_{1})-2\gamma}{(\alpha_0+\beta_0)(a_{1})}.$$ This completes the computation of the Hausdorff dimension of $\mathcal E_3\cap U_0$.

\subsection{The Diophantine sets $\mathcal E_4$ and $\mathcal E_{5}$}
Here we will consider the Diophantine subset $\mathcal E_4\cap U_0$. By definition, we know that $\mathcal E_4\cap U_0$ is the subset of points $p\in U_0$ such that there exist a sequence $t_n\to\infty$ of real numbers and a sequence $q_n\in G/\Gamma$ of rational points such that $p\in B_n\cdot q_n$, with
\begin{align*}
B_n=&\bigg\{x\in N_+:x\in B\left(0,Ce^{\beta_0(a_{-t_n})},Ce^{(\alpha_0+\beta_0)(a_{-t_n})}\right)\cdot y,y\in N_{\alpha_0},\\
&\left.\frac1{3\kappa''}e^{-\gamma t_n}e^{\beta_0(a_{t_n})}\leq d_G(y,e)\cdot d_{q_n}\leq\kappa'' e^{-\gamma t_n}e^{\beta_0(a_{t_n})}\right\}
\end{align*}
and $$d_{q_n}\leq\kappa''e^{(\alpha_0+\beta_0)(a_{t_n})-\gamma t_n}.$$ Without loss of generality, we may assume that $t_n\in\mathbb N$. Then we have $$\frac1{3\kappa''}\frac{e^{-\gamma t_n}e^{\beta_0(a_{t_n})}}{d_{q_n}}\leq d_G(y,e)\leq\frac{\kappa'' e^{-\gamma t_n}e^{\beta_0(a_{t_n})}}{d_{q_n}}$$ and $$1\leq d_{q_n}\leq\kappa^3e^{(\alpha_0+\beta_0)(a_{t_n})-\gamma t_n}.$$ For any $n\in\mathbb N$, define $\mathcal G_n$ to be the collection of subsets of the form 
\begin{align*}
\Bigg\{B\left(0,Ce^{\beta_0(a_{-n})},Ce^{(\alpha_0+\beta_0)(a_{-n})}\right)&\cdot y\cdot q: y\in N_{\alpha_0},\\
&\left.\frac1{3\kappa''}\frac{e^{-\gamma n}e^{\beta_0(a_{n})}}{d_{q}}\leq d_G(y,e)\leq\frac{\kappa'' e^{-\gamma n}e^{\beta_0(a_{n})}}{d_{q}}\right\},
\end{align*}
where $q$ is a rational point in $N_+(\mathbb Q)/\Gamma$ with $$1\leq d_{q}\leq\kappa'' e^{(\alpha_0+\beta_0)(a_{n})-\gamma n}.$$ Then, the union $\mathcal G=\bigcup_{n\in\mathbb N}\mathcal G_n$ covers $\mathcal E_4\cap U_0$. 

For each $n\in\mathbb N$ we subdivide each set in the collection
\begin{align*}
\Bigg\{B\left(0,Ce^{\beta_0(a_{-n})},Ce^{(\alpha_0+\beta_0)(a_{-n})}\right)&\cdot y\cdot q: y\in N_{\alpha_0},\\
&\left.\frac1{3\kappa''}\frac{e^{-\gamma n}e^{\beta_0(a_{n})}}{d_{q}}\leq d_G(y,e)\leq\frac{\kappa'' e^{-\gamma n}e^{\beta_0(a_{n})}}{d_{q}}\right\}
\end{align*}
into cubes of side length $Ce^{(\alpha_0+\beta_0)(a_{-n})}$. Thus, we obtain a family of cubes, which we denote by $\mathcal F$, covering $\mathcal E_4\cap U_0$. In order to estimate the Hausdorff dimension of $\mathcal E_4\cap U_0$, we need to study the series
\begin{align*}
&\sum_{B\in\mathcal F}(\diam B)^s\\
\lesssim&\sum_{n\in\mathbb N}\sum_{1\leq d_q\leq \kappa''e^{(\alpha_0+\beta_0)(a_n)-\gamma n}}\frac{e^{\beta_0(a_n)-\gamma n}}{d_q}\cdot  e^{(\alpha_0+\beta_0)(a_{n})}\cdot e^{\alpha_0(a_n)}e^{s(\alpha_0+\beta_0)(a_{-n})}\\
\sim &\sum_{n=0}^\infty\sum_{l=0}^{n-1}\sum_{\kappa^{''}e^{(\alpha_0+\beta_0)(a_l)-\gamma\cdot l}\leq  d_q\leq \kappa''e^{(\alpha_0+\beta_0)(a_{l+1})-\gamma\cdot(l+1)}}\frac{e^{2(\alpha_0+\beta_0)(a_n)-\gamma n-s(\alpha_0+\beta_0)(a_n)}}{d_q}\\
\lesssim&\sum_{n=0}^\infty\sum_{l=0}^{n-1}\frac{e^{2(\alpha_0+\beta_0)(a_n)-\gamma n-s(\alpha_0+\beta_0)(a_n)}}{e^{(\alpha_0+\beta_0)(a_l)-\gamma\cdot l}}\left(e^{(\alpha_0+\beta_0)(a_{l+1})-\gamma\cdot(l+1)}\right)^{2+\epsilon}\\
\sim&\sum_{n=0}^\infty\sum_{l=0}^{n-1}e^{2(\alpha_0+\beta_0)(a_n)-\gamma n-s(\alpha_0+\beta_0)(a_n)}\left(e^{((\alpha_0+\beta_0)(a_{1})-\gamma)l}\right)^{1+\epsilon}\\
\sim&\sum_{n=0}^\infty e^{2(\alpha_0+\beta_0)(a_n)-\gamma n-s(\alpha_0+\beta_0)(a_n)}\left(e^{((\alpha_0+\beta_0)(a_{1})-\gamma)n}\right)^{1+\epsilon},
\end{align*}
where we have used Lemma~\ref{l75} in the third inequality. This series converges if $$2(\alpha_0+\beta_0)(a_1)-\gamma-s(\alpha_0+\beta_0)(a_1)+((\alpha_0+\beta_0)(a_{1})-\gamma)(1+\epsilon)<0,$$ or equivalently if $$s>\frac{2(\alpha_0+\beta_0)(a_1)-\gamma+((\alpha_0+\beta_0)(a_{1})-\gamma)(1+\epsilon)}{(\alpha_0+\beta_0)(a_1)}.$$ This implies that $$\dim_H(\mathcal E_4\cap U_0)\leq\frac{2(\alpha_0+\beta_0)(a_1)-\gamma+((\alpha_0+\beta_0)(a_{1})-\gamma)(1+\epsilon)}{(\alpha_0+\beta_0)(a_1)}.$$ Since $\epsilon>0$ is arbitrarily small, we conclude that $$\dim_H(\mathcal E_4\cap U_0)\leq\frac{3(\alpha_0+\beta_0)(a_1)-2\gamma}{(\alpha_0+\beta_0)(a_1)}.$$ This completes the computation for the upper bound of the Hausdorff dimension of $\mathcal E_4\cap U_0$. The computation for the set $\mathcal E_{5}$ is analogous.

\section{Lower Bounds for the Hausdorff dimension}

In this section, we will prove a lower bound for the Hausdorff dimension of $S^c_\gamma\cap U_0$. The main result of this section reads as follows.

\begin{proposition}
\label{prop:lowerbound}
For all $\gamma<(\alpha_{0}+\beta_{0})(X_{0})$ and all open sets $U_{0}\subset N_{+}\Gamma/\Gamma$ we have
$$\dim_{H}(S^c_\gamma\cap U_0)\geq 3-\frac{2\gamma}{(\alpha_0+\beta_0)(X_{0})}.$$
\end{proposition}

To this end, for any $\varepsilon>0$ we will construct a Cantor-type subset in $S_\gamma^c$ and prove that this subset has Hausdorff dimension bounded below by $3-2\gamma/(\alpha_0+\beta_0)(X_{0})-\varepsilon$. By letting $\varepsilon\to 0$, we will conclude the proof of our main theorem.

We start by recalling a dimensional result for Cantor-type subsets of a Riemannian manifold. One can refer to \cite{KM96,M87,U91} for more details. Let $X$ be a Riemannian manifold, $m$ a volume form on $X$, and $E$ a compact subset of $X$. Let $\diam(S)$ denote the diameter of a set $S\subset X$. A collection $\mathcal A$ of compact subsets of $E$ is said to be tree-like if $\mathcal A$ is the union of finite sub-collections $\mathcal A_j$ such that\vspace{2mm}
\begin{enumerate}
\item $\mathcal A_0=\left\{E\right\}$;\vspace{2mm}
\item For any $j$ and $S_1,S_2\in\mathcal A_j$, either $S_1=S_2$ or $S_1\cap S_2=\emptyset$;\vspace{2mm}
\item For any $j$ and $S_1\in\mathcal A_{j+1}$, there exists $S_2\in\mathcal A_j$ such that $S_1\subset S_2$; \vspace{2mm}
\item $d_j(\mathcal A):=\sup_{S\in\mathcal A_j}\diam(S)\to0$ as $j\to\infty$.\vspace{2mm}
\end{enumerate}

We write $\mathbf A_j:=\bigcup_{A\in\mathcal A_j}A$ and $\mathbf A_\infty:=\bigcap_{j\in\mathbb N}\mathbf A_j$. We also define $$\Delta_j(\mathcal A):=\inf_{S\in\mathcal A_j}\frac{m(\mathbf A_{j+1}\cap S)}{m(S)}.$$

\begin{theorem}\label{thm72}
Let $(X,m)$ be a Riemannian manifold, where $m$ is the volume form on $X$. Then for any tree-like collection $\mathcal A$ of subsets of $E$ $$\dim_H(\mathbf A_\infty)\geq\dim_{H} X-\limsup_{j\to\infty}\frac{\sum_{i=0}^j\log(\Delta_i(\mathcal A))}{\log(d_{j+1}(\mathcal A))}.$$ 
\end{theorem}

The next result will be crucial in the construction of the Cantor-type set. 

\begin{lemma}\label{l76}
Let $\nu$ be the lowest root in $\Phi_-$ and let $K$ be a compact subset of $\ker\nu$. Let $U$ be an open subset of $N_+$ and $\gamma>0$. Then, there exists $\epsilon_0>0$ such that for all sufficiently large $l\in\mathbb N$ the sets in the collection
$$\left\{B\left(\epsilon_0 d_q^{\frac{\alpha_0(a_{-1})}{(\alpha_0+\beta_0)(a_1)-\gamma}},\epsilon_0 d_q^{\frac{\beta_0(a_{-1})}{(\alpha_0+\beta_0)(a_1)-\gamma}},\epsilon_0 d_q^{\frac{(\alpha_0+\beta_0)(a_{-1})}{(\alpha_0+\beta_0)(a_1)-\gamma}}\right)\cdot q:q\in S_{K}(U/\Gamma,l,l/2)\right\}$$ are pairwise disjoint.
\end{lemma}
\begin{proof}
We will use the notation in the proof of Proposition~\ref{p41}. For any $l>1$, define $\tau:=\tau(l)>0$ such that $$\nu(a_\tau)=-\ln l.$$ Note that for the compact subset $K$ in $\ker\nu$, we have $$a_{\tau}\cdot S_{K}(U,l/2,l)=a_{\tau}\cdot U/\Gamma\cap\exp(\mathfrak a_{K,I_0})N_{-}\Gamma,$$ where $$I_0=\left\{x\in\mathbb R\cdot X_0:-\ln2\leq\nu(x)\leq0\right\}.$$ Since $\exp(\mathfrak a_{K,I_0})N_{-}\Gamma$ is a compact subset in $G/\Gamma$, by Lemma \ref{lem:transversality}, there exists a small neighborhood of identity $$B(\delta)=\left\{\exp(x): x=\sum_{\alpha\in\Phi_+} x_\alpha,\ x_\alpha\in\mathfrak g_\alpha,\ \|x_\alpha\|_{\mathfrak g}<\delta\right\}$$ in $N_+$ such that $$B(\delta)\times\exp(\mathfrak a_{K,I_0})N_{-}\Gamma\to B(\delta)\exp(\mathfrak a_{K,I_0})N_{-}\Gamma$$ is a homeomorphism. In view of this, we can conclude that for any two rational points $p,q$ in $S_K(U,l/2,l)$, the sets $$B(\delta)\cdot a_{\tau}\cdot p\textrm{ and }B(\delta)\cdot a_{\tau}\cdot q$$ are disjoint. Consequently, also the sets $B_{\delta,\tau}\cdot p$ and $B_{\delta,\tau}\cdot q$ are disjoint, where $$B_{\delta,\tau}=a_{-\tau}\cdot B(\delta)\cdot a_{\tau}=\left\{\exp(x): x=\sum_{\alpha\in\Phi_+} x_\alpha, x_\alpha\in\mathfrak g_\alpha, \|x_\alpha\|_{\mathfrak g}<e^{\alpha(a_{-\tau})}\cdot\delta\right\}.$$ Note that for any $x\in S_K(U,l/2,l)$ $$d_x\sim l=e^{(\alpha_0+\beta_0)(a_\tau)}\geq e^{((\alpha_0+\beta_0)(a_1)-\gamma)\tau}.$$ Hence, if $\epsilon_0>0$ is chosen to be sufficiently small, then $$B\left(\epsilon_0 d_x^{\frac{\alpha_0(a_{-1})}{(\alpha_0+\beta_0)(a_1)-\gamma}},\epsilon_0 d_x^{\frac{\beta_0(a_{-1})}{(\alpha_0+\beta_0)(a_1)-\gamma}},\epsilon_0 d_x^{\frac{(\alpha_0+\beta_0)(a_{-1})}{(\alpha_0+\beta_0)(a_1)-\gamma}}\right)$$ is contained in $B_{\delta,\tau}$. This implies that for any $p,q\in S_K(U,l/2,l)$ the sets $$B\left(\epsilon_0 d_p^{\frac{\alpha_0(a_{-1})}{(\alpha_0+\beta_0)(a_1)-\gamma}},\epsilon_0 d_p^{\frac{\beta_0(a_{-1})}{(\alpha_0+\beta_0)(a_1)-\gamma}},\epsilon_0 d_p^{\frac{(\alpha_0+\beta_0)(a_{-1})}{(\alpha_0+\beta_0)(a_1)-\gamma}}\right)p$$ and $$B\left(\epsilon_0 d_q^{\frac{\alpha_0(a_{-1})}{(\alpha_0+\beta_0)(a_1)-\gamma}},\epsilon_0 d_q^{\frac{\beta_0(a_{-1})}{(\alpha_0+\beta_0)(a_1)-\gamma}},\epsilon_0 d_q^{\frac{(\alpha_0+\beta_0)(a_{-1})}{(\alpha_0+\beta_0)(a_1)-\gamma}}\right)q$$ are disjoint, proving the thesis.
\end{proof}

We will now construct a Cantor-type subset in $S_\gamma^c\cap U_0$ and estimate its Hausdorff dimension. This will prove Proposition \ref{prop:lowerbound}. We will use Theorem~\ref{thm72}.

We choose a compact subset $K$ in $\ker\nu$ and $\epsilon_0$ as in Lemma~\ref{l76}. We also choose a sufficiently small number $\epsilon>0$.  For $j=0$, we set $\mathcal A_0=\left\{U_0\right\}$. For $j=1$, we choose a sufficiently large number $l_1>0$ such that\vspace{2mm}
\begin{enumerate}
\item the subsets in the collection 
\begin{align*}
\left\{B\left(\epsilon_0 d_q^{\frac{\alpha_0(a_{-1})}{(\alpha_0+\beta_0)(a_1)-(\gamma+\epsilon)}},\epsilon_0 d_q^{\frac{\beta_0(a_{-1})}{(\alpha_0+\beta_0)(a_1)-(\gamma+\epsilon)}},\epsilon_0 d_q^{\frac{(\alpha_0+\beta_0)(a_{-1})}{(\alpha_0+\beta_0)(a_1)-(\gamma+\epsilon)}}\right)\cdot q: q\in S_{K}(U_0,l_1/2,l_1)\right\}
\end{align*}
are disjoint (see Lemma \ref{l76});\vspace{2mm}
\item $|S_K(U_0,l_1/2,l_1)|\sim l_1^2\cdot\mu_{N_+}(U_0)$, where the implicit constants depend only on $K$, $G$ and $\Gamma$ (see Proposition \ref{p41}).\vspace{2mm}
\end{enumerate}
We denote the above collection of sets by $\mathcal F_1(U_0)$. We subdivide each set in the collection $\mathcal F_1(U_0)$ into cubes of side length $$\epsilon_0 d_q^{\frac{(\alpha_0+\beta_0)(a_{-1})}{(\alpha_0+\beta_0)(a_1)-(\gamma+\epsilon)}}.$$ Thus, we obtain a family of cubes constructed from sets in $\mathcal F_1(U_0)$, which we denote by $\mathcal A_1$. 

For $j>1$ we choose a sufficiently large number $l_j>0$ such that for each $S\in\mathcal A_{j-1}$,\vspace{2mm}
\begin{enumerate}
\item the subsets in the collection $$\left\{B\left(\epsilon_0 d_q^{\frac{\alpha_0(a_{-1})}{(\alpha_0+\beta_0)(a_1)-(\gamma+\epsilon)}},\epsilon_0 d_q^{\frac{\beta_0(a_{-1})}{(\alpha_0+\beta_0)(a_1)-(\gamma+\epsilon)}},\epsilon_0 d_q^{\frac{(\alpha_0+\beta_0)(a_{-1})}{(\alpha_0+\beta_0)(a_1)-(\gamma+\epsilon)}}\right)\cdot q: q\in S_{K}(S,l_j/2,l_j)\right\}$$ are disjoint;\vspace{2mm}
\item $|S_K(S,l_j/2,l_j)|\sim l_j^2\cdot\mu_{N_+}(S)$, where the implicit constants depend only on $K$, $G$ and $\Gamma$;\vspace{2mm}
\item $l_j\geq l_{j-1}^{(j-1)^{2}}$.\vspace{2mm}
\end{enumerate}
We denote the above collection of sets by $\mathcal F_j(S)$. We subdivide each set in $\mathcal F_j(S)$ $(S\in\mathcal A_{j-1})$ into cubes of side length $$\epsilon_0 d_q^{\frac{(\alpha_0+\beta_0)(a_{-1})}{(\alpha_0+\beta_0)(a_1)-(\gamma+\epsilon)}}.$$ Thus, we have obtained a sequence $\left\{l_j\right\}_{j\in\mathbb N}$ with $$l_{j+1}\geq l_{j-1}^{(j-1)^{2}}\;(\forall j\in\mathbb N)$$ and a tree-like sequence $\left\{\mathcal A_j\right\}_{j\in\mathbb N}$ of finite collections of cubes.
By using the notation introduced in Proposition~\ref{p41}, we have $$d_j(\mathcal A)\sim l_j^{\frac{(\alpha_0+\beta_0)(a_{-1})}{(\alpha_0+\beta_0)(a_1)-(\gamma+\epsilon)}}\quad\mbox{and}\quad\Delta_j(\mathcal A)\sim l_{j+1}^2\cdot l_{j+1}^{\frac{2(\alpha_0+\beta_0)(a_{-1})}{(\alpha_0+\beta_0)(a_1)-(\gamma+\epsilon)}}\;(\forall j\in\mathbb N),$$
where the second approximate equality follows from Proposition \ref{p41}. Now let $\mathbf A_\infty=\bigcap_{j\in\mathbb N}\mathbf A_j$ and
$$b:=\frac{2(\alpha_0+\beta_0)(a_{-1})}{(\alpha_0+\beta_0)(a_1)-(\gamma+\epsilon)}.$$
By applying Theorem~\ref{thm72}, we find that
\begin{multline}\dim_H(\mathbf A_\infty)\geq \dim_{H} X-\limsup_{j\to\infty}\frac{\sum_{i=0}^j\log(\Delta_i(\mathcal A))}{\log(d_{j+1}(\mathcal A))} \\
\geq 3-\limsup_{j\to\infty}\frac{(b+2)\sum_{i=0}^j\log l_{j+1}}{b\log l_{j+1}}\geq 3-\frac{b+2}{b}\left(1+\limsup_{j\to\infty}\frac{j\log l_{j}}{\log l_{j+1}}\right).
\end{multline}
Since $l_{j+1}\geq l_{j}^{j^{2}}$, the last term is bounded above by $1/j$. It follows that
$$\dim_H(\mathbf A_\infty)\geq 3-\frac{b+2}{b}= 3-\frac{2+\frac{2(\alpha_0+\beta_0)(a_{-1})}{(\alpha_0+\beta_0)(a_1)-(\gamma+\epsilon)}}{\frac{(\alpha_0+\beta_0)(a_{-1})}{(\alpha_0+\beta_0)(a_1)-(\gamma+\epsilon)}}=3-\frac{2(\gamma+\epsilon)}{(\alpha_0+\beta_0)(a_1)}.$$ Note that by applying Proposition~\ref{p63} part $(3)$ with
$$t_{n}:=l_{n}^{\frac{1}{(\alpha_0+\beta_0)(a_1)-(\gamma+\epsilon)}},$$
we have that $\mathbf A_\infty\subset S^c_\gamma\cap U_0$. Hence, $$\dim_H(S^c_\gamma\cap U_0)\geq 3-\frac{2(\gamma+\epsilon)}{(\alpha_0+\beta_0)(a_1)}.$$ We complete the computation by taking $\epsilon\to0$.

\section{Conclusion}

The proof of Theorem \ref{mthm} follows from Proposition \ref{prop:allDioph}, Corollary \ref{cor:upperbound}, and Proposition \ref{prop:lowerbound}.

We conclude this section by proving Corollary \ref{cor:mcor}. Let $U$ be an open subset in $G/\Gamma$. Without loss of generality, we may assume that $$U=V_-V_0V_+\Gamma,$$ where $V_-$, $V_0$, and $V_+$ are open subsets in $N_-$, $A$ and $N_+$ respectively. Note that, by Remark \ref{rmk:remainingdim}, for any $g\in N_-A$ we have $p\in S_\gamma$ if and only if $g\cdot p\in S_\gamma$. Since for sufficiently small sets $V_{-},V_{0},V_{+}$, we have that the map $(v_{-}v_{0},v_{+}\Gamma)\mapsto v_{-}v_{0}v_{+}\Gamma$ is a homeomorphism (see Lemma \ref{lem:transversality}), we deduce that
$$\dim_H(S_\gamma^c\cap U)=\dim_{H} V_-V_0+\dim_H(S_\gamma^c\cap V_+\Gamma).$$
Here we used the fact that if $A\subset \mathbb{R}^{n}$ is any set and $B\subset\mathbb{R}^{m}$ is an open set, then
$$\dim_{H}A\times B=\dim_{H}A+m.$$
See \cite{Ha71} for a proof. By Theorem \ref{mthm}, we conclude that
$$\dim_H(S_\gamma^c\cap U)=5+\frac{3(\alpha_0+\beta_0)(a_1)-2\gamma}{(\alpha_0+\beta_0)(a_1)}.$$

\bibliographystyle{plain}









\end{document}